\documentclass[11pt]{amsart}
\usepackage[T1]{fontenc}
\usepackage[utf8]{inputenc}
\usepackage{lmodern,microtype}
\usepackage{amsmath,amssymb,mathtools}
\usepackage{enumitem,needspace,booktabs,array}
\usepackage{tikz-cd}
\usepackage{xurl,xcolor}
\usepackage[letterpaper,margin=1.02in]{geometry}
\usepackage{hyperref}
\hypersetup{colorlinks=true,linkcolor=blue!45!black,citecolor=blue!45!black,urlcolor=blue!45!black,pdftitle={Poincare duality and module topological complexity},pdfauthor={Paul-Eugène Parent}}
\allowdisplaybreaks[2]
\usepackage{etoolbox}
\makeatletter
\patchcmd{\@maketitle}{\global\topskip42\p@\relax}{\global\topskip\normaltopskip\relax}{}{}
\makeatother
\apptocmd{\thebibliography}{\setlength{\itemsep}{0pt plus 4pt}}{}{}
\newcommand{\Q}{\mathbb Q}
\newcommand{\TC}{\operatorname{TC}}
\newcommand{\MTC}{\operatorname{MTC}}
\newcommand{\HTC}{\operatorname{HTC}}
\newcommand{\secat}{\operatorname{secat}}
\newcommand{\relcat}{\operatorname{relcat}}
\newcommand{\TCM}{\operatorname{TC}^{M}}
\newcommand{\msecat}{\operatorname{Msecat}}
\newcommand{\Mcat}{\operatorname{Mcat}}
\newcommand{\cat}{\operatorname{cat}}
\newcommand{\id}{\mathrm{id}}
\newcommand{\Hom}{\operatorname{Hom}}

\newcommand{\Ann}{\operatorname{Ann}}

\newcommand{\coker}{\operatorname{coker}}
\newcommand{\ev}{\operatorname{ev}}

\newcommand{\Dcat}{\mathcal D}
\newcommand{\DD}{\mathbb D}
\newcommand{\Lt}{\otimes^{\mathbf L}}
\newcommand{\RHom}{\mathbf R\!\operatorname{Hom}}
\theoremstyle{plain}
\newtheorem{theorem}{Theorem}[section]
\newtheorem{proposition}[theorem]{Proposition}
\newtheorem{lemma}[theorem]{Lemma}
\newtheorem{corollary}[theorem]{Corollary}
\theoremstyle{definition}
\newtheorem{definition}[theorem]{Definition}
\newtheorem{example}[theorem]{Example}

\theoremstyle{remark}
\newtheorem{remark}[theorem]{Remark}

\DeclareRobustCommand{\fundingmark}{\textsuperscript{*}}
\title[Poincar\'e duality and module topological complexity]{Poincar\'e duality and module topological complexity}
\author[Paul-Eugène Parent]{Paul-Eugène Parent\fundingmark}
\thanks{\makebox[0pt][r]{\fundingmark\enspace}This research was supported by the Natural Sciences and Engineering Research Council of Canada (NSERC), grant no.~149521.}
\address{Department of Mathematics and Statistics, Faculty of Science,
University of Ottawa, STEM Complex, Room 361, 150 Louis-Pasteur Private,
Ottawa, Ontario K1N 6N5, Canada}
\email{pparent@uottawa.ca}
\date{November 9th, 2026}
\begin{document}
\begin{abstract}
For every simply connected space $X$ of finite rational type with
Poincar\'e duality cohomology, we prove $\MTC(X)=\TC_0(X)=\TC_0^M(X)$:
module, rational and monoidal rational topological complexity coincide.
Neither formality nor finiteness of the total rational homotopy is required.
For every rationally retractive map $f$ between simply connected spaces
of finite rational type, we prove
$\relcat(f_\Q)\le\msecat(f_\Q)+1$. This bound is sharp for
rationally retractive maps.
Both proofs reduce to the same question: can a module nullhomotopy
taking values in a split Sullivan model be chosen with values in its
kernel ideal? Comparing resolutions of consecutive ideal powers achieves
this after increasing the exponent by one; for diagonals, Poincar\'e duality preserves the
exponent. An ideal-preserving form of Hess's construction then converts the module
nullhomotopy into a compatible algebra retraction.
\end{abstract}
\maketitle

\section{Introduction}\label{sec:introduction}

\subsection{From modules back to multiplication}

Can a problem about multiplicative maps be solved by first forgetting
multiplication? Rational Lusternik--Schnirelmann category provides a
striking example: Ganea's formula
\[
 \cat(X\times S^n)=\cat(X)+1,\qquad n\ge1,
\]
fails for general spaces \cite{Iwase1998} but holds rationally, and its
proof passes through a module invariant. F\'elix and Halperin express
rational category $\cat_0(X)$ through algebra retractions of Sullivan models
\cite{FH1982}. Halperin and Lemaire replace algebra retractions by
module retractions, obtaining the lower bound $\Mcat(X)$ \cite{HL1988}.
Jessup proves the rational Ganea formula for this module invariant
\cite{JessupGanea1990} and, in a separate paper, identifies it with
$\cat_0$ under a dimension--connectivity bound \cite{Jessup1990}.
Hess removes the bound: $\Mcat(X)=\cat_0(X)$
\cite{Hess1991}. F\'elix, Halperin and Lemaire then establish product
additivity and identify rational category with the Toomer invariant
$e_0$ under Poincar\'e duality \cite{FHL1998}.

The decisive step is the return from modules to algebras. A module
retraction satisfies linear equations over the model; an algebra
retraction must also preserve products. For rational category, Hess's
theorem shows that this additional requirement does not change the
invariant. We ask what happens for topological complexity, where a
retraction must account for the diagonal of a space rather than the
inclusion of a point.

\subsection{Topological complexity and the question}

Farber's topological complexity $\TC(X)$ measures the number of local
continuous rules needed to plan a path between two points of $X$
\cite{Farber2003}. Equivalently, it is the sectional category of the
diagonal $X\to X\times X$. We use normalized invariants: $\TC(X)\le n$
means that $n+1$ local rules suffice. Its rational version is
$\TC_0(X)=\TC(X_\Q)$, where $X_\Q$ denotes rationalization.

The module approximation $\MTC(X)$ replaces algebra retractions in
models of the diagonal's Ganea fibrations by module retractions
\cite{FGKV2006}. It is a lower bound for $\TC_0(X)$, and the two
invariants agree for formal spaces. Their common value is the rational
zero-divisors cup-length, by Proposition~6.4 of Fern\'andez Su\'arez
et al. and the computation of Lechuga and Murillo \cite{FGKV2006,LM2007}.
The question is how much information the module approximation loses
beyond that case.

We prove that it loses nothing when the rational cohomology of $X$
satisfies Poincar\'e duality, and at most one in general. The same
argument also controls monoidal topological complexity, which requires
the path chosen from a point to itself to be constant. These conclusions
come from a common obstruction, described after the two main theorems.

\subsection{The two main results}

Throughout, spaces have the homotopy type of simply connected CW
complexes of \emph{finite rational type}: $H^k(X;\Q)$ is finite
dimensional for every $k$. Put $\TC_0^M(X)=\TC^M(X_\Q)$; the formal
definition of the monoidal invariant is given in
Definition~\ref{def:tc}.

A connected finite-dimensional graded algebra $H$ has
\emph{Poincar\'e duality in dimension $N$} if $H^N\cong\Q$ and
multiplication $H^k\otimes H^{N-k}\to H^N$ is a perfect pairing for
every $k$.

\begin{theorem}[Poincar\'e duality]\label{thm:poincare-main}
Let $X$ be a simply connected space of finite rational type whose
rational cohomology is a Poincar\'e duality algebra. Then
\[
 \MTC(X)=\TC_0(X)=\TC_0^M(X).
\]
\end{theorem}

No formality and no finiteness of the total rational homotopy are
required. A space is \emph{rationally elliptic} when its
total rational homotopy and total rational cohomology are finite
dimensional. Such a space satisfies Poincar\'e duality
\cite[Proposition~38.3]{FHT2001}, so we obtain:

\begin{corollary}\label{cor:elliptic-main}
For every simply connected rationally elliptic space $X$,
$\MTC(X)=\TC_0(X)=\TC_0^M(X)$.
\end{corollary}

The equality also holds for higher diagonals. Together with known
properties of the module invariant, it gives product additivity and
identifies all three invariants with $\HTC$ under duality. This weaker
invariant asks only that the models of the diagonal's Ganea fibrations
induce injections in cohomology. Section~\ref{sec:consequences} gives
these consequences and a non-formal elliptic example detected by a
nonzero Massey product.

Our second theorem concerns a map $f\colon X\to Y$ which is
\emph{rationally retractive}: some $g\colon Y_\Q\to X_\Q$ satisfies
$gf_\Q\simeq\id$. The diagonal, retracted by a projection, is the basic
example. The three invariants below are defined in
Section~\ref{sec:preliminaries}. Module sectional category
$\msecat$ asks for module retractions; sectional category $\secat$ asks
for local sections; and the relative category $\relcat$ of Doeraene and
El Haouari retains compatibility with the map from $X$
\cite{DEH2013}. For the diagonal, relative category is monoidal
topological complexity \cite{CGV2014}.

\begin{theorem}[The general comparison]\label{thm:main}
Let $f\colon X\to Y$ be a rationally retractive map between simply
connected spaces of finite rational type. Then
\[
 \msecat(f_\Q)\le\secat(f_\Q)\le\relcat(f_\Q)\le\msecat(f_\Q)+1.
\]
The three invariants are simultaneously finite or infinite.
If $\msecat(f_\Q)=0$, all three vanish.
\end{theorem}

The constant one is optimal for rationally retractive maps:
\cite[Theorem~1.1]{Parent2026} gives a map with $\msecat=2$ and
$\secat=3$, hence $\relcat=3$ by our theorem. For the diagonal we obtain:

\begin{corollary}\label{cor:diagonal-bound}
For every simply connected space $X$ of finite rational type,
\[
 \MTC(X)\le\TC_0(X)\le\TC_0^M(X)\le\MTC(X)+1.
\]
\end{corollary}

Thus the module invariant determines rational topological complexity up
to one, and exactly under Poincar\'e duality. Whether a strict inequality
occurs for a diagonal remains open.

Two features of $\MTC$ already echoed the history of rational category:
it is additive on products \cite{CPV2018}, and under Poincar\'e duality
it agrees with $\HTC$ \cite[Corollary~4.1]{CKV2016}. The missing equality with
$\TC_0$ is what Theorem~\ref{thm:poincare-main} supplies. We view it as
the analogue, for topological complexity, of the equality
$\cat_0=e_0$ for Poincar\'e duality spaces \cite{FHL1998}.

\subsection{One obstruction, two targets}\label{subsec:intro-mechanism}

The two theorems come from the same obstruction. To describe it, choose
a split minimal Sullivan model of the rationally retractive map $f$:
\[
 \varphi\colon R=A\otimes\Lambda W\longrightarrow A,
 \qquad \varphi|_A=\id_A,\qquad\varphi(W)=0.
\]
Here $R$ models $Y$, $A$ models $X$, and the kernel is the differential
ideal $I=(W)$, with $I^m=A\otimes\Lambda^{\ge m}W$.
Section~\ref{sec:preliminaries} constructs this model and defines the
module notions used below.

Write $\Dcat(R)$ for the derived category of differential $R$-modules,
and let $\iota\colon I\hookrightarrow R$ be the inclusion. For $m\ge1$,
consider the class
\[
 o_m=[I^m\hookrightarrow I]\in\Hom_{\Dcat(R)}(I^m,I),
 \qquad
 \iota_*(o_m)=[I^m\hookrightarrow R].
\]
Module sectional category detects the image $\iota_*(o_m)$; relative
category detects $o_m$ itself. The difference is the target in which
the inclusion must become nullhomotopic.

This distinction has a concrete meaning. Choose a semifree resolution
$\varepsilon\colon(P,D_P)\to I^m$. Vanishing of $\iota_*(o_m)$ means that
there is an $R$-linear nullhomotopy $h\colon P\to R$, a map of degree
$-1$ satisfying $d_Rh+hD_P=\iota\varepsilon$. Vanishing of $o_m$ asks for
such a nullhomotopy with values in $I$. The composite
$\varphi h\colon P\to A$ is a cocycle, called the \emph{trace} of $h$.
The question is whether $h$ can be replaced by an $R$-linear
nullhomotopy with zero trace. Although $\varphi$ has an algebra section,
that section need not be $R$-linear; simply subtracting the trace does
not solve the problem. Section~\ref{subsec:quotient-definitions}
expresses this obstruction through the exact sequence associated with
$I\to R\to A$.

There are two ways to reach the smaller target. The \emph{exponent
shift} compares resolutions of $I^{m+1}$ and $I^m$ by a map with
coefficients in $I$. Under Poincar\'e duality, evaluation supplies that
factor of $I$ without changing the exponent. Together with the
quotient and relative-category criteria, these mechanisms give four
relations, for $n\ge0$ and $m\ge1$:
\begin{equation}\label{eq:intro-scheme}
\begin{array}{rcl@{\qquad}l}
 \msecat(f_\Q)\le n &\Longleftrightarrow& \iota_*(o_{n+1})=0
   &\text{quotient criterion (Carrasquel-Vera)}\\[6pt]
 \relcat(f_\Q)\le n &\Longleftrightarrow& o_{n+1}=0
   &\begin{array}{@{}l@{}}
      \text{relative-category criterion}\\[-2pt]
      \text{Theorem~\ref{thm:relative-category}}
    \end{array}\\[6pt]
 \iota_*(o_m)=0 &\Longrightarrow& o_{m+1}=0
   &\begin{array}{@{}l@{}}
      \text{exponent shift}\\[-2pt]
      \text{Theorem~\ref{thm:shift}}
    \end{array}\\[6pt]
 \iota_*(o_m)=0 &\Longrightarrow& o_m=0
   &\begin{array}{@{}l@{}}
      \text{duality theorem}\\[-2pt]
      \text{Theorem~\ref{thm:duality}, under its hypotheses}
    \end{array}
\end{array}
\end{equation}
The first three lines give the bound in Theorem~\ref{thm:main}; the
last line removes the possible gap. The last assertion of
Theorem~\ref{thm:main} is Lemma~\ref{lem:zero-threshold}.

The second relation contains the essential multiplicative step. An
ideal-valued module nullhomotopy must extend to a retraction of algebras
compatible with $\varphi$. The conversion theorem supplies
this extension for $m\ge2$, adapting Hess's construction while keeping
every new generator value in $I$. Together with
Lemma~\ref{lem:quotient-to-ganea}, this proves the relative-category
criterion. The general
comparison and the duality argument provide the required nullhomotopies;
the construction of the algebra retraction occupies
Sections~\ref{sec:toolbox}--\ref{sec:filtered-model}.

For the diagonal of a Poincar\'e duality space, a compatible comparison
with the finite Poincar\'e models of Lambrechts and Stanley
\cite{LS2008} verifies the duality hypotheses. This gives the passage from the last line of
\eqref{eq:intro-scheme} to Theorem~\ref{thm:poincare-main}.

\subsection{Earlier work}

Quotient models for rational topological complexity were introduced by
Jessup, Murillo and the author \cite{JMP2012}. Carrasquel-Vera develops
quotient upper bounds for sectional category and its module approximation
\cite[Definition~9 and Proposition~12]{CV2015}, then proves
the necessary and sufficient quotient criteria for rationally retractive
maps \cite{CV2017}. In particular,
\cite[Theorem~3.6]{CV2017} characterizes relative category through
algebra retractions compatible, up to homotopy, with the map to $A$.
Theorem~\ref{thm:relative-category} replaces that multiplicative condition
by the vanishing of $o_{n+1}$. When $A=\Q$, the ideal-valued condition
recovers Hess's theorem (Corollary~\ref{cor:hess-revisited}).

Our duality argument supplies the ideal-valued nullhomotopy needed to
pass from the module invariants to algebras. The resulting equality
complements computations and bounds
for elliptic spaces \cite{HRV2023,HRV2025,Hamoun2025}. Higher topological
complexity is due to Rudyak \cite{Rudyak2010}; related growth results
appear in \cite{KY2020}.

The preprint \cite{Parent2026} supplies the strict retractive example
and the bounded-target argument used among the applications below;
the ideal criterion and the general comparison and duality theorems
are proved here.

\subsection{Organization}

The first part presents the results and their consequences.
Sections~\ref{sec:preliminaries}--\ref{sec:ideal} prove the theorems using
the controlled quotient model
(Proposition~\ref{prop:filtered-sullivan-model}), the module-to-algebra conversion theorem
(Theorem~\ref{thm:ideal-nullhomotopy}) and the controlled resolutions
and comparison maps (Propositions~\ref{prop:controlled-resolution}
and~\ref{prop:power-comparison}). Section~\ref{sec:consequences} then
gives the non-formal example, product and growth formulas, and degree
criteria. The dependency diagram at the end of that section marks the
passage from the results to the technical construction.
The second part constructs these inputs.
Section~\ref{sec:toolbox} isolates the form of differential for which
Hess's extension works. Section~\ref{sec:controlled-resolutions} builds
the resolutions and comparison homotopies with controlled coefficients.
Section~\ref{sec:filtered-model} uses the binary comparison of
Corollary~\ref{cor:binary-comparison} to assemble a Sullivan model and
prove Theorem~\ref{thm:ideal-nullhomotopy}.

The two-sphere accompanies the proof. Its diagonal model makes the
obstruction concrete; its comparison maps exhibit the extra ideal
factor; and its primitives distinguish values in $I$ from values in
$I^2$. Its finite Poincar\'e model also shows why powers of the kernel
ideal must not be transported through arbitrary quasi-isomorphisms.
Example~\ref{ex:sphere-comparison} collects its model, resolution and
comparison maps. Appendix~\ref{app:sphere-resolutions} checks the
formulas, and Appendix~\ref{app:universal-linear-resolution} proves the
natural splittings used in the resolutions.

\section{Three invariants and one obstruction class}\label{sec:preliminaries}

We now make the obstruction from the introduction precise. We begin
with the geometric invariants and the models that compute them. The
constructions are standard; the point of view is new. The quotient criteria of Carrasquel-Vera then
turn the module invariant into a nullhomotopy problem. The crucial distinction is between making an
inclusion nullhomotopic in the ambient algebra and making it
nullhomotopic in its kernel ideal.

\subsection{Conventions}\label{subsec:conventions}

\subsubsection*{Spaces}
All spaces have the homotopy type of simply connected CW complexes.
They are of \emph{finite rational type}: $H^k(X;\Q)$ is finite
dimensional for every $k$. For simply connected spaces, this is
equivalent to requiring $\pi_k(X)\otimes\Q$ to be finite dimensional
for every $k$, or the minimal Sullivan model to have finitely many
generators in each degree. This is a finiteness condition on the
rational homotopy type, degree by degree. We use the quoted rational
results with this convention; Appendix~\ref{app:finite-type-literature}
records the model-theoretic justification.
We write $X_\Q$ and $f_\Q$ for rationalization. Every numerical invariant
below is normalized, so that the value $n$ corresponds to $n+1$ local
domains, and takes values in $\mathbb N\cup\{\infty\}$.

\subsubsection*{Sullivan algebras}
We work over $\Q$, with the conventions of F\'elix, Halperin and Thomas
\cite{FHT2001}. A commutative differential graded algebra (CDGA) has a
differential of degree $+1$. For a graded vector space $V$, $\Lambda V$
denotes the free graded-commutative algebra on $V$: polynomial on even
generators, exterior on odd ones. We write $\Lambda^jV$ for the words of
length $j$ and $\Lambda^{\ge j}V$ for their span over all lengths at least~$j$;
cohomological degree is indicated by a superscript, $(\Lambda V)^k$. A
\emph{Sullivan algebra} $(\Lambda V,d)$ is a free CDGA with a well-ordered
homogeneous basis of $V$ such that the differential of each basis element
involves only earlier basis elements. All our absolute minimal Sullivan
models satisfy $V=V^{\ge2}$ and $\dim V^k<\infty$ for every $k$; in this
simply connected setting,
\emph{minimality} means $d(V)\subset\Lambda^{\ge2}V$, and the Sullivan
condition is then automatic when the basis is ordered by degree. A
\emph{relative Sullivan algebra} over a CDGA $R$ is an inclusion
$R\hookrightarrow(R\otimes\Lambda Z,D)$ with the analogous condition on a
well-ordered basis of $Z$; a \emph{relative Sullivan model} of a morphism
$R\to C$ is a factorization $R\hookrightarrow R\otimes\Lambda Z\xrightarrow{\simeq}C$
through such an inclusion. A morphism $R\to C$ of CDGAs \emph{admits a
homotopy retraction} if its relative Sullivan model admits a retraction
$R\otimes\Lambda Z\to R$ of CDGAs which is the identity on $R$; this does
not depend on the model.

\subsubsection*{Modules}
Let $R$ be a CDGA. A \emph{semifree} differential $R$-module is a graded
free module $P=R\otimes G$ with an exhaustive increasing filtration of a
basis of $G$ such that the differential of a basis element involves only
basis elements of lower filtration; the basis elements are called the
\emph{cells} of $P$. Semifree modules are cofibrant: every module has
a semifree resolution $P\xrightarrow{\simeq}M$, and maps from a semifree
module lift up to homotopy along quasi-isomorphisms. More precisely, a
quasi-isomorphism $M\to N$ induces a bijection between homotopy classes
of morphisms $P\to M$ and $P\to N$. A homogeneous $R$-linear map $u$ of
degree $k$ satisfies $u(rp)=(-1)^{k|r|}ru(p)$; the differential of the Hom
complex is $\partial u=Du-(-1)^{|u|}uD$. A \emph{nullhomotopy} of a
cochain map $v\colon P\to M$ is an $R$-linear map $h$ of degree $-1$
with $\partial h=v$. Tensor products and permutations carry the Koszul signs. We write $\Dcat(R)$ for the derived category of
differential $R$-modules, $M[k]^i=M^{i+k}$ for shifts, and $M^\vee$ for the
$\Q$-linear dual with its usual differential. The suspension $sP$ of a
module has $|sp|=|p|-1$ and $s(rp)=(-1)^{|r|}r(sp)$. For a morphism
$\varepsilon\colon P\to R$ of differential $R$-modules we use the
\emph{cone}
\begin{equation}\label{eq:general-module-cone}
 Q=R\oplus sP,\qquad D_Q|_R=d_R,\qquad D_Q(sp)=\varepsilon(p)-sDp .
\end{equation}

\subsection{Sectional category, relative category, topological complexity}
\label{subsec:invariant-definitions}

Sectional category asks how many local homotopy sections are needed.
Relative category imposes an additional compatibility with the original
map. For the diagonal of a space, these become ordinary and monoidal
topological complexity, respectively.

\begin{definition}[Sectional category]\label{def:secat}
The sectional category $\secat(f)$ of a map $f\colon X\to Y$ is the least
$n\ge0$ such that $Y$ has an open cover $U_0,\dots,U_n$ with maps
$s_i\colon U_i\to X$ satisfying $fs_i\simeq\operatorname{incl}_{U_i}$. If
no finite cover has this property, $\secat(f)=\infty$.
\end{definition}

The covering definition has a global reformulation due to Schwarz.
Replace $f$ by a fibration over $Y$ and form its $(n+1)$-fold fibrewise
join
\[
 g_n\colon G_n(f)\longrightarrow Y,
\]
the \emph{$n$th Ganea map of $f$}. The construction comes with a canonical
map $\alpha_n\colon X\to G_n(f)$ such that $g_n\alpha_n\simeq f$, and
$\secat(f)\le n$ if and only if $g_n$ has a homotopy section
\cite{Schwarz1966,DEH2013}. Relative category asks for a section which
remembers where it comes from.

\begin{definition}[Relative category]\label{def:relcat}
The relative category $\relcat(f)$ is the least $n\ge0$ for which there
is a map $s\colon Y\to G_n(f)$ with
\[
 g_ns\simeq\id_Y\qquad\text{and}\qquad sf\simeq\alpha_n .
\]
It is infinite if no such $n$ exists.
\end{definition}

Forgetting the second condition gives $\secat(f)\le\relcat(f)$, and
Doeraene and El Haouari prove $\relcat(f)\le\secat(f)+1$
\cite[Theorem~18]{DEH2013}.

\begin{definition}[Topological complexity]\label{def:tc}
The topological complexity $\TC(X)$ is the sectional category of the
endpoint fibration $X^{[0,1]}\to X\times X$, $\gamma\mapsto(\gamma(0),\gamma(1))$.
Thus $\TC(X)\le n$ means that $X\times X$ is covered by $n+1$ open sets,
on each of which a path between the two points of a pair can be chosen
continuously. The \emph{monoidal} topological complexity $\TCM(X)$ requires
in addition that each open set contain the diagonal and that the path
chosen from a point to itself be constant \cite{IwaseSakai2010}.
\end{definition}

The endpoint fibration is a fibration replacement of the diagonal
$\Delta\colon X\to X\times X$, so $\TC(X)=\secat(\Delta)$. For the CW
complexes considered here, $\TCM(X)=\relcat(\Delta)$
\cite[Theorem~12]{CGV2014}. We write
\[
 \TC_0(X)=\TC(X_\Q),\qquad\TC_0^M(X)=\TCM(X_\Q),
\]
and, for $r\ge2$, $\TC_{r,0}(X)=\secat((\Delta_r)_\Q)$ for the diagonal
$\Delta_r\colon X\to X^r$; this is Rudyak's higher topological complexity
of $X_\Q$ \cite{Rudyak2010}.

\begin{definition}[Rationally retractive maps]\label{def:rationally-retractive}
A map $f\colon X\to Y$ is \emph{rationally retractive} if there is a map
$g\colon Y_\Q\to X_\Q$ with $gf_\Q\simeq\id_{X_\Q}$.
\end{definition}

Either projection retracts a diagonal $\Delta_r$; the inclusion of a
point is retracted by the constant map; a section of a fibration is
retracted by its projection. Throughout the paper, $f$ is the map whose
sectional invariants are studied, and $g$ a chosen homotopy retraction.

\subsection{Module invariants}\label{subsec:module-definitions}

A homotopy section of a Ganea map corresponds to an algebra retraction
of its relative Sullivan model. Forgetting the multiplicative condition
and asking only for a module retraction gives the lower bound used
throughout this paper.

\begin{definition}[Module sectional category]\label{def:msecat}
Let $R$ be a Sullivan model of $Y_\Q$ and let
$j_n\colon R\hookrightarrow(R\otimes\Lambda Z_n,D)$ be a relative Sullivan
model of the Ganea map $g_n\colon G_n(f_\Q)\to Y_\Q$. The \emph{module
sectional category} $\msecat(f_\Q)$ is the least $n\ge0$ such that $j_n$
admits a retraction of differential $R$-modules, or equivalently a left
inverse in $\Dcat(R)$; it is infinite if no such $n$ exists
\cite[Definition~5.1 and Theorem~5.4]{FGKV2006}. The \emph{module topological complexity}
is $\MTC(X)=\msecat(\Delta_\Q)$, and $\MTC_r(X)=\msecat((\Delta_r)_\Q)$.
\end{definition}

This definition is independent of the relative Sullivan model:
comparisons under $R$ preserve homotopy retractions. Since these models
are semifree as differential $R$-modules, having a module homotopy
retraction is equivalent to having a left inverse in $\Dcat(R)$;
such a retraction fixes $R$ strictly, since an $R$-linear homotopy
of degree $-1$ on $R$ is determined by its value on $1$, which lies in
$R^{-1}=0$. When
$q=\msecat(f_\Q)<\infty$, we call $q$ the \emph{module threshold}.
The quotient criterion below will test this threshold at exponent $q+1$.

A retraction of CDGAs is in particular a retraction of modules. Hence
\begin{equation}\label{eq:three-inequalities}
 \msecat(f_\Q)\le\secat(f_\Q)\le\relcat(f_\Q),
\end{equation}
and in particular $\MTC(X)\le\TC_0(X)\le\TC_0^M(X)$. For the inclusion of
a point, $\msecat(*\to X_\Q)$ is the module category $\Mcat(X)$ of
Halperin and Lemaire \cite{HL1988}; this identification follows from
\cite[Proposition~5.6]{FGKV2006}. Moreover,
$\secat(*\to X_\Q)=\cat_0(X)$. We follow
\cite{FGKV2006} in writing $\msecat$ and $\MTC$; the papers
\cite{CV2017,CPV2018} write $\operatorname{msecat}$ and
$\operatorname{mTC}$ for the same invariants.

A still weaker invariant only looks at cohomology: $\operatorname{Hsecat}(f_\Q)$
is the least $n$ for which $g_n$ induces an injection in rational
cohomology, and $\HTC(X)=\operatorname{Hsecat}(\Delta_\Q)$ \cite{CKV2016}.
A module retraction induces such an injection, so
$\HTC(X)\le\MTC(X)$.

The quotient criterion below replaces a Ganea model by $R\to R/J$ for
an ideal $J$. To use it, we need to turn the existence of a module
retraction into a question about the ideal itself. The next lemma does
this in two equivalent ways: an inclusion becomes nullhomotopic, or its
module cone retracts onto $R$. The cone will later be the first layer of
an algebra model.

\begin{lemma}\label{lem:derived-retraction}
Let $J$ be a differential ideal of a CDGA $R$, let
$\varepsilon\colon P\xrightarrow{\simeq}J$ be a semifree resolution, and
regard $\varepsilon$ as a map to $R$. The following are equivalent:
\begin{enumerate}[label=\textup{(\roman*)}]
\item the projection $R\to R/J$ has a left inverse in $\Dcat(R)$;
\item the inclusion $J\hookrightarrow R$ is zero in $\Dcat(R)$;
\item there is an $R$-linear map $h\colon P\to R$ of degree $-1$ with
 $d_Rh+hD_P=\varepsilon$;
\item the cone $Q=R\oplus sP$ of \eqref{eq:general-module-cone}, which is a
 module model of $R/J$, admits a retraction $\rho_0\colon Q\to R$ of
 differential $R$-modules with $\rho_0|_R=\id_R$.
\end{enumerate}
When they hold, $\rho_0$ and $h$ determine each other by
$\rho_0(sp)=h(p)$.
\end{lemma}

\begin{proof}
The short exact sequence $0\to J\to R\to R/J\to0$ gives a distinguished
triangle in $\Dcat(R)$, and the identity of $R$ extends over $R\to R/J$
exactly when its composite with $J\to R$ vanishes; this is
(i)$\Leftrightarrow$(ii). The quasi-isomorphism
$\varepsilon\colon P\to J$ is invertible in $\Dcat(R)$, so the inclusion
$J\to R$ is zero there if and only if its composite with $\varepsilon$
is zero. Since $P$ is semifree, the latter condition is equivalent to
an $R$-linear nullhomotopy, proving (ii)$\Leftrightarrow$(iii).
Finally, the kernel of the surjection $Q\to R/J$, $sp\mapsto0$, is
$J\oplus sP$, the cone of the quasi-isomorphism $P\to J$. It is acyclic,
so $Q$ is a module model of $R/J$. An
$R$-linear map $\rho_0\colon Q\to R$ with $\rho_0|_R=\id_R$ is determined by
the degree $-1$ map $h=\rho_0\circ s\colon P\to R$, and it commutes with
the differentials exactly when
$d_Rh(p)=\rho_0(D_Q(sp))=\varepsilon(p)-h(D_Pp)$ for
all $p$, which is (iii).
\end{proof}

\subsection{The split model of a retractive map}\label{subsec:split-model}

We next choose a model adapted to the retraction of $f$. The generators
coming from $X$ generate a subalgebra $A$ of the model $R$ of $Y$; the
remaining generators generate the kernel ideal of the model of $f$. In
these coordinates, its $m$th power is spanned over $A$ by words of length
at least $m$ in the remaining generators.

\begin{lemma}[Split minimal model]\label{lem:split-minimal-model}
Let $f\colon X\to Y$ be a rationally retractive map between simply
connected spaces of finite rational type. Then $f_\Q$ has a model
\begin{equation}\label{eq:retractive-minimal-model}
 \varphi\colon(R,d_R)=(A\otimes\Lambda W,d_R)\longrightarrow(A,d),
 \qquad\varphi|_A=\id_A,\qquad\varphi(W)=0,
\end{equation}
where $R$ models $Y$ and $A$ models $X$; both are minimal Sullivan
algebras with generators of degree at least two and finite dimensional in
each degree. The inclusion $A\hookrightarrow R$ is a relative Sullivan
algebra, it models $g$, and
\[
 d_R(W)\subset I:=\ker\varphi=(W),\qquad
 I^m=A\otimes\Lambda^{\ge m}W\quad(m\ge1).
\]
\end{lemma}

\begin{proof}
Choose morphisms of minimal Sullivan models $\varphi\colon R\to A$ and
$\psi\colon A\to R$ representing $f_\Q$ and $g$. The composite
$\varphi\psi$ is homotopic to the identity of $A$, hence a
quasi-isomorphism, hence an isomorphism, since a quasi-isomorphism between
minimal Sullivan algebras is an isomorphism
\cite[Theorem~14.11]{FHT2001}, applied over the base $\Q$.
Replacing $\psi$ by $\psi(\varphi\psi)^{-1}$ we may assume
$\varphi\psi=\id_A$.

Write $A=\Lambda V$ and let $Q(\varphi)$ be the induced map on
indecomposables. The image of $V$ under $Q(\psi)$ is a complement of
$\ker Q(\varphi)$. Choose a homogeneous basis of this kernel, lift its
elements to $\widetilde w\in R$, and put
$w=\widetilde w-\psi\varphi(\widetilde w)$. Each $w$ lies in $\ker\varphi$
and still represents the chosen indecomposable class, because
$\varphi(\widetilde w)$ is decomposable. Let $W$ be the span of these
elements. The elements $\psi(V)$ together with the chosen basis of $W$
freely generate $R$: the induced map on indecomposables is an
isomorphism, and since all generators have degree at least two, the
change of generators is inverted by induction on degree. Identifying $A$
with $\psi(A)$ gives \eqref{eq:retractive-minimal-model}, and
$I=\ker\varphi$ is the ideal generated by $W$, with the stated powers.

Minimality is preserved by a change of generators, so every relative
generator occurring in a decomposable term of $d_Rw$ has degree less than
$|w|$; ordering $W$ by degree makes $A\hookrightarrow R$ a relative
Sullivan algebra. Since $\varphi$ is a cochain map which kills $W$,
$d_R(W)\subset\ker\varphi=I$.
\end{proof}

We fix such a split model and write $\iota_A\colon A\hookrightarrow R$
for its inclusion. The \emph{relative length} of a monomial is the number
of its factors from $W$; its filtration is given by the powers $I^m$.
Factors from $A$ do not contribute to this count, so it differs from
word length in all the generators of $R$.
Decomposing $d_R$ according to relative length gives
\begin{equation}\label{eq:relative-differential-splitting}
 d_R=D_{\mathrm{lin}}+d_+ ,
\end{equation}
where $D_{\mathrm{lin}}$ agrees with $d$ on $A$ and keeps, on a relative
generator $w$, the terms of $d_Rw$ that are linear in $W$, while $d_+$
vanishes on $A$ and strictly increases relative length. Both are
derivations.

\begin{example}[The diagonal]\label{ex:diagonal-model}
Let $M=(\Lambda V,d)$ be the minimal model of $X$. The diagonal
$\Delta\colon X\to X\times X$ is modelled by the multiplication
$\mu\colon M\otimes M\to M$. Writing $v_L=v\otimes1$, $v_R=1\otimes v$ and
$\bar v=v_R-v_L$ for $v\in V$, we have
\begin{equation}\label{eq:diagonal-coordinates}
 M\otimes M=M_L\otimes\Lambda\bar V,\qquad
 \ker\mu=(\bar V),\qquad(\ker\mu)^m=M_L\otimes\Lambda^{\ge m}\bar V .
\end{equation}
This is the split model with $A=M_L$ and $W=\bar V$, and
\eqref{eq:relative-differential-splitting} is obtained by expanding
$d_R\bar v=d(v_L+\bar v)-d(v_L)$. For the higher diagonal $\Delta_r$ one keeps the first copy of $V$ and replaces each
generator of the other $r-1$ copies by its difference with the first copy.
\end{example}

\subsection{The quotient criteria and the obstruction class}
\label{subsec:quotient-definitions}

For a rationally retractive map, the Ganea models can be replaced by the
quotients of $R$ by the powers of the kernel ideal. This is the theorem of
Carrasquel-Vera, which we state in our setting.

\begin{theorem}[The quotient criterion, Carrasquel-Vera]\label{thm:quotient-criteria}
Let $\varphi\colon R\to A$ be the split model of a rationally retractive
map $f$, and $I=\ker\varphi$. Then:
\begin{enumerate}[label=\textup{(\roman*)}]
\item $\secat(f_\Q)$ is the least $n\ge0$ such that $R\to R/I^{n+1}$
 admits a homotopy retraction of CDGAs;
\item $\msecat(f_\Q)$ is the least $n\ge0$ such that $R\to R/I^{n+1}$
 has a left inverse in $\Dcat(R)$.
\end{enumerate}
If no such $n$ exists, the corresponding invariant is infinite.
\end{theorem}

This is \cite[Theorem~3.3\textup{(i)--(ii)}]{CV2017}. Its ``$s$-model'' is a
surjective CDGA model which retracts a relative Sullivan inclusion;
Lemma~\ref{lem:split-minimal-model} supplies this structure. The module
homotopy retraction in that theorem is the derived left inverse of
Definition~\ref{def:msecat}. For the diagonal, part (i) is the case
$r=2$, with a minimal Sullivan model, of the higher-topological-complexity
formula in \cite[Theorem~1.2]{CV2017}.

By Lemma~\ref{lem:derived-retraction}, the module criterion asks whether
$I^{n+1}\hookrightarrow R$ is nullhomotopic. But this inclusion already
factors through $I$. Keeping that smaller target gives the obstruction
class introduced in Section~\ref{subsec:intro-mechanism}.

\begin{definition}[Obstruction class]\label{def:obstruction-class}
For $m\ge1$, the \emph{obstruction class} is
\begin{equation}\label{eq:obstruction-class}
 o_m=[I^m\hookrightarrow I]\in\Hom_{\Dcat(R)}(I^m,I).
\end{equation}
We denote the inclusion of the kernel ideal by
$\iota\colon I\hookrightarrow R$, and composition with it by $\iota_*$.
\end{definition}

By Lemma~\ref{lem:derived-retraction}, Theorem~\ref{thm:quotient-criteria}(ii)
reads
\begin{equation}\label{eq:msecat-derived}
 \msecat(f_\Q)\le n
 \quad\Longleftrightarrow\quad
 \iota_*(o_{n+1})=[I^{n+1}\hookrightarrow R]=0\quad\text{in }\Dcat(R),
\end{equation}
where we have used that the condition is monotone in $n$: the inclusion
of $I^{n+2}$ factors through that of $I^{n+1}$. The next section proves
the corresponding criterion for relative category:
\[
 \relcat(f_\Q)\le n\quad\Longleftrightarrow\quad o_{n+1}=0
 \qquad\text{(Theorem~\ref{thm:relative-category}).}
\]
Thus the module invariant detects the image of the obstruction under
$\iota_*$, whereas relative category detects the obstruction itself.
Once the image vanishes, the remaining question is whether the
nullhomotopy can be chosen with values in $I$.

\emph{Why the splitting does not remove the obstruction.}
Suppose that $\iota_*(o_m)=0$, and choose a semifree resolution
$\varepsilon:P\to I^m$ with an $R$-linear nullhomotopy $h:P\to R$ of
$\iota\varepsilon$. Since $R=A\oplus I$ as cochain complexes, one
might try to subtract the $A$-component of $h$. Recall that
$\iota_A:A\hookrightarrow R$ is the splitting.
The formula $h'=h-\iota_A\varphi h$ gives values in $I$ and still
satisfies $d_Rh'+h'D_P=\iota\varepsilon$: the cochain map
$\iota_A\varphi$ kills $\iota\varepsilon$, so the correction has zero
differential as a graded $\Q$-linear map. However, $\iota_A$ is $A$-linear
and need not be $R$-linear for the quotient action on $A$. Consequently,
$h'$ need not be $R$-linear either. For homogeneous $r\in R$ and $p\in P$,
its defect is
\[
 h'(rp)-(-1)^{|r|}rh'(p)
   =(-1)^{|r|}(r-\iota_A\varphi(r))\iota_A\varphi h(p).
\]
The difficulty is therefore to remove the $A$-component of $h$ while
preserving $R$-linearity. The trace records exactly that component, and
an exact sequence tells us which corrections are allowed.

The distinguished triangle
$I\xrightarrow{\iota}R\xrightarrow{\varphi}A\to I[1]$ gives the exact
sequence
\begin{equation}\label{eq:trace-sequence}
 \Hom_{\Dcat(R)}(I^m,R[-1])\xrightarrow{\varphi_*}
 \Hom_{\Dcat(R)}(I^m,A[-1])\xrightarrow{\ \delta\ }
 \Hom_{\Dcat(R)}(I^m,I)\xrightarrow{\ \iota_*\ }
 \Hom_{\Dcat(R)}(I^m,R),
\end{equation}
in which $R$ acts on $A$ through $\varphi$ and $\delta$ is the
connecting morphism, computed by lifting a cocycle and applying the Hom
differential $\partial$. Concretely, let $\varepsilon\colon P\to I^m$ be a
semifree resolution. The Hom complexes $\Hom_R(P,T)$, $T=I,R,A$, form a
short exact sequence of complexes because $P$ is graded free, and
\eqref{eq:trace-sequence} is a piece of its long exact cohomology
sequence.

\begin{definition}[Trace of a nullhomotopy]\label{def:trace}
Let $\varepsilon\colon P\to I^m$ be a semifree resolution and let
$h\colon P\to R$ be an $R$-linear nullhomotopy of $\iota\varepsilon$.
The \emph{trace} of $h$ is the composite $\varphi h\colon P\to A$,
regarded as a cocycle of degree $-1$ in $\Hom_R(P,A)$.
\end{definition}

It is a cocycle because $\partial(\varphi h)=\varphi\iota\varepsilon=0$.
With the stated convention, $\delta[\varphi h]=o_m$: lifting
$\varphi h$ to $h$ and applying the Hom differential gives
$\partial h=\iota\varepsilon$, which represents $o_m$ in
$\Hom_R(P,I)$.
Consequently, when $\iota_*(o_m)=0$, exactness gives
\[
 o_m=0\quad\Longleftrightarrow\quad
 [\varphi h]\in\operatorname{im}\varphi_*
 \quad\Longleftrightarrow\quad
 \text{some nullhomotopy takes its values in }I.
\]
The trace itself depends on the chosen nullhomotopy. If $h$ is replaced
by another nullhomotopy $h+z$, then $z$ is a cocycle of degree $-1$ in
$\Hom_R(P,R)$, and the trace class changes by $\varphi_*[z]$.
Its image in $\coker(\varphi_*)$ is therefore independent of $h$.
This is the residual obstruction after the module class has vanished.

The exact sequence also produces the desired correction. If
$[\varphi h]\in\operatorname{im}\varphi_*$, write
$\varphi h=\varphi z+\partial k$, where $z\colon P\to R$ is a cocycle
of degree $-1$ and $k\colon P\to A$ has degree $-2$. Since $P$ is
graded free, lift $k$ to a graded $R$-linear map
$\widetilde k\colon P\to R$. Then
$h-z-\partial\widetilde k$ is a nullhomotopy with zero trace, hence
with values in $I$. The converse follows by choosing that nullhomotopy.

The diagram below puts the obstruction and its trace together. The
map $\iota_*$ forgets that the nullhomotopy must take values in $I$;
the connecting map $\delta$ recovers the obstruction from the trace.
The labels under the groups recall the vanishing criteria at exponent
$m=n+1$.
\[
\begin{gathered}
 I\xrightarrow{\iota}R\xrightarrow{\varphi}A\longrightarrow I[1],
 \\[4pt]
 \begin{tikzcd}[ampersand replacement=\&,column sep=1.6em,row sep=.6em]
 \Hom_{\Dcat(R)}(I^m,A[-1])
    \arrow[r,"\delta"]
 \& \Hom_{\Dcat(R)}(I^m,I)\arrow[r,"\iota_*"]
 \& \Hom_{\Dcat(R)}(I^m,R) \\
 \text{trace modulo }\operatorname{im}\varphi_*
 \& \begin{gathered}o_m=0\\ \relcat(f_\Q)\le m-1\end{gathered}
 \& \begin{gathered}\iota_*(o_m)=0\\ \msecat(f_\Q)\le m-1\end{gathered}
 \end{tikzcd}
\end{gathered}
\]

Before proceeding, we dispose of the zero case. Here vanishing of the
module invariant already forces the map to be a rational equivalence.

\begin{lemma}[The zero threshold]\label{lem:zero-threshold}
For a rationally retractive map, the following are equivalent:
$\msecat(f_\Q)=0$; the ideal $I$ is acyclic; $f_\Q$ is a homotopy
equivalence. They imply $\secat(f_\Q)=\relcat(f_\Q)=0$.
\end{lemma}

\begin{proof}
If $\msecat(f_\Q)=0$, then $\iota\colon I\to R$ is zero in $\Dcat(R)$ by
\eqref{eq:msecat-derived}; since $R=A\oplus I$ as cochain complexes,
$H(\iota)$ is injective, so $H(I)=0$. The exact sequence
$0\to I\to R\to A\to0$ shows that $H(I)=0$ exactly when $\varphi$ is a
quasi-isomorphism, that is, when $f_\Q$ is a homotopy equivalence. A
homotopy equivalence has relative category zero, and
\eqref{eq:three-inequalities} gives the rest.
\end{proof}

\subsection{The running example}\label{subsec:running-example}

The diagonal of the rational two-sphere will accompany the argument.
The minimal model is $M=(\Lambda(x_2,a_3),d)$, with $dx=0$ and $da=x^2$.
In its split model, detailed in Example~\ref{ex:sphere-comparison},
write $\bar x=x_R-x_L$ and $\bar a=a_R-a_L$. The cocycle
$\bar x^2\in I^2$ represents $-2[x_Lx_R]\ne0$ in $H(R)$, so
$\iota_*(o_2)\ne0$ and $\MTC(S^2)\ge2$. Since $S^2$ is formal with
zero-divisors cup-length two, $\TC_0(S^2)=2$ \cite{LM2007}; hence
$\MTC(S^2)=2$ as well. The term $\bar x^2$ in
$d\bar a=2x_L\bar x+\bar x^2$ shows that the diagonal differential
still has a nonlinear part in the relative variables.
Example~\ref{ex:sphere-comparison} collects the resolution and comparison
maps; Example~\ref{ex:sphere-ideal-valued} exhibits the ideal-valued
nullhomotopy at exponent three.

\subsection{Notation introduced so far}

The table collects the objects used in this section. The resolutions
and multiplicative models will be introduced with their constructions.

\begin{center}\small
\begin{tabular}{@{}ll@{}}
\toprule
Symbol & Meaning \\
\midrule
$f\colon X\to Y$, $g$ & map under study and rational homotopy retraction, $gf_\Q\simeq\id$\\
$\varphi\colon R=A\otimes\Lambda W\to A$ & split model; $R$ models $Y$ and $A$ models $X$\\
$\iota_A\colon A\hookrightarrow R$ & inclusion splitting $\varphi$ as a CDGA map\\
$I=(W)$, $\iota\colon I\hookrightarrow R$ & kernel ideal and its inclusion; $I^m=A\otimes\Lambda^{\ge m}W$\\
$o_m=[I^m\hookrightarrow I]$ & obstruction class in $\Hom_{\Dcat(R)}(I^m,I)$\\
$h\colon P\to R$, $\varphi h\colon P\to A$ & $R$-linear nullhomotopy and its trace\\
$[\varphi h]\in\Hom_{\Dcat(R)}(I^m,A[-1])$ & trace class; its image modulo $\operatorname{im}\varphi_*$ is independent of $h$\\
$\rho_0\colon Q\to R$, $Q=R\oplus sP$ & module retraction of the cone, with $\rho_0(sp)=h(p)$\\
$g_n\colon G_n(f)\to Y$, $\alpha_n\colon X\to G_n(f)$ & Ganea map and its canonical compatibility map\\
$q=\msecat(f_\Q)$ & module threshold, when finite\\
$\Dcat(R)$ & derived category of differential $R$-modules\\
$D_{\mathrm{lin}}$, $d_+$ & linear and length-increasing parts of $d_R$\\
\bottomrule
\end{tabular}
\end{center}

The notation for the construction is collected in
Table~\ref{tab:technical-notation}; its gradings are listed in
Table~\ref{tab:construction-gradings}.

\section{Relative category and the kernel ideal}
\label{sec:relative-category}

We keep the split model $\varphi\colon R=A\otimes\Lambda W\to A$ with
kernel $I=(W)$. The module invariant asks whether an ideal power becomes
null in $R$. Relative category asks the same question with target $I$.

This section and the next use the controlled quotient model
(Proposition~\ref{prop:filtered-sullivan-model}), the conversion theorem
(Theorem~\ref{thm:ideal-nullhomotopy}), and the controlled resolutions
and comparisons (Propositions~\ref{prop:controlled-resolution}
and~\ref{prop:power-comparison}).
The reader willing to accept these inputs can read the present section
and the next without consulting
Sections~\ref{sec:toolbox}--\ref{sec:filtered-model}.

\begin{theorem}[The relative-category criterion]\label{thm:relative-category}
Let $f\colon X\to Y$ be a rationally retractive map between simply
connected spaces of finite rational type, with split minimal
model $\varphi\colon R\to A$ and kernel $I$. For every $n\ge0$,
\begin{equation}\label{eq:relative-category-ideal-equivalence}
 \relcat(f_\Q)\le n
 \quad\Longleftrightarrow\quad
 o_{n+1}=[I^{n+1}\hookrightarrow I]=0\quad\text{in }\Dcat(R).
\end{equation}
\end{theorem}

The proof connects three descriptions of the same compatibility:
a relative section, a quotient-model retraction compatible with
$\varphi$, and a nullhomotopy with zero trace. We first isolate the
passage from the quotient model back to the Ganea map. Its auxiliary
notation is confined to the following lemma.

\begin{lemma}[From the quotient model to the Ganea map]
\label{lem:quotient-to-ganea}
Let $m\ge2$, let $\pi\colon R\to R/I^m$ be the quotient map, and let
$\overline\varphi\colon R/I^m\to A$ be induced by $\varphi$.
Suppose a relative Sullivan model
\[
 R\hookrightarrow E\xrightarrow{\ \zeta\ \simeq\ }R/I^m
\]
with $\zeta$ surjective admits a CDGA retraction $\rho\colon E\to R$
satisfying $\varphi\rho=\overline\varphi\zeta$. Then
$\relcat(f_\Q)\le m-1$.
\end{lemma}

\begin{proof}
The quotient keeps only the product ideal $I^m$, whereas a Ganea model
starts with the tensor ideal $I^{\otimes m}$. Multiplication supplies the
comparison between them. Use the construction of
\cite[Section~5]{CGV2014}, and choose a relative Sullivan model
\[
 R^{\otimes m}\hookrightarrow
 (R^{\otimes m}\otimes\Lambda V_{\mathrm G},D)
 \xrightarrow{\ \xi\ \simeq\ }
 R^{\otimes m}/I^{\otimes m},\qquad \xi(V_{\mathrm G})=0 .
\]
Here $I^{\otimes m}=I\otimes\cdots\otimes I$ is a differential ideal of
$R^{\otimes m}$. Pushout along multiplication
$\mu_m\colon R^{\otimes m}\to R$ gives
$\mathcal G=(R\otimes\Lambda V_{\mathrm G},\overline D)$, a relative Sullivan model
of the Ganea map $g_{m-1}$. Since
\[
 D(V_{\mathrm G})\subset I^{\otimes m}+R^{\otimes m}\otimes\Lambda^+V_{\mathrm G},
 \qquad \mu_m(I^{\otimes m})\subset I^m,
\]
the assignments $\nu|_R=\pi$ and $\nu(V_{\mathrm G})=0$ define a CDGA map
$\nu\colon\mathcal G\to R/I^m$.

Relative Sullivan lifting through the surjective quasi-isomorphism
$\zeta$ gives a map $\widetilde\nu\colon\mathcal G\to E$ under $R$
with $\zeta\widetilde\nu=\nu$. The composite
$\rho^{\mathrm G}=\rho\widetilde\nu\colon\mathcal G\to R$ is the
identity on $R$ and satisfies
\[
 \varphi\rho^{\mathrm G}
   =\overline\varphi\zeta\widetilde\nu
   =\overline\varphi\nu .
\]
The last map is $\varphi$ on $R$ and zero on $V_{\mathrm G}$: it is the model of
$\alpha_{m-1}$ in that construction. Thus $\rho^{\mathrm G}(V_{\mathrm G})\subset I$.
By \cite[Proposition~29]{CGV2014}, this is precisely the compatibility
which gives a relative section at stage $m-1$.
\end{proof}

\begin{proof}[Proof of Theorem~\ref{thm:relative-category}]
\emph{The zero threshold.} The class $o_1$ is the identity of $I$, which
is zero in $\Dcat(R)$ exactly when $I$ is acyclic. By
Lemma~\ref{lem:zero-threshold} and $\msecat\le\relcat$, this happens
exactly when $\relcat(f_\Q)=0$.

\emph{The model for $n\ge1$.} Put $m=n+1$. We use the controlled
quotient model of Proposition~\ref{prop:filtered-sullivan-model}, whose
construction is deferred to Section~\ref{sec:filtered-model}:
\begin{equation}\label{eq:controlled-model-recalled}
 R\lhook\joinrel\longrightarrow E_m=(R\otimes\Lambda Z,D_E)
 \xrightarrow{\ \zeta_m\ \simeq\ }R/I^m,\qquad\zeta_m(Z)=0.
\end{equation}
It contains the module cone $Q_m=R\oplus sP_m$ of a semifree resolution
$\varepsilon_m\colon(P_m,D_m)\to I^m$, with $P_m=R\otimes G_m$;
the restriction of $\zeta_m$ to $Q_m$ is its cone projection. We also
write $\varepsilon_m$ for the composite into $R$, so
$D_E(sp)=\varepsilon_m(p)-sD_mp$. Let
$\overline\varphi\colon R/I^m\to A$ be induced by $\varphi$, and write
$\varphi_m=\overline\varphi\zeta_m\colon E_m\to A$; this map extends
$\varphi$ and kills $Z$. For the converse we will use the conversion
Theorem~\ref{thm:ideal-nullhomotopy}: its needed conclusion is a CDGA
retraction of this model sending $Z$ into $I$.

\emph{From a relative section to an ideal-valued nullhomotopy.}
Suppose $\relcat(f_\Q)\le n$. The least-stage criterion of
\cite[Theorem~3.6]{CV2017} provides a compatible retraction
at an exponent $m'\le m$. To pass to exponent $m$, lift the quotient map
$R/I^m\to R/I^{m'}$ to a map $E_m\to E_{m'}$ under $R$; the target
factorization may be chosen surjective. The two augmentations to $A$
commute with this comparison, because
$\overline\varphi_{m'}\circ(R/I^m\to R/I^{m'})=\overline\varphi_m$.
Composing the retraction and its relative homotopy with the comparison
therefore gives a CDGA retraction $\rho\colon E_m\to R$ such that
\begin{equation}\label{eq:relative-category-criterion}
 \varphi\rho\simeq\varphi_m\quad\text{relative to }R.
\end{equation}
Let $H\colon E_m\to A\otimes\Lambda(u,du)$ be the resulting homotopy
from $\varphi_m$ to $\varphi\rho$, constant at $\varphi$ on $R$.
The relative condition is what gives
$H(rz)=(\varphi(r)\otimes1)H(z)$ for $r\in R$ and $z\in E_m$.
Integration along $u$, with the polynomial forms written on the right,
sends $au^j$ to $0$ and $au^jdu$ to $(-1)^{|a|}a/(j+1)$.
It therefore gives a graded $R$-linear map
$\mathcal I\colon E_m\to A$ of degree $-1$, vanishing on $R$, with
\[
 d_A\mathcal I+\mathcal I D_E=\varphi\rho-\varphi_m.
\]
Define $h(p)=\rho(sp)$ and $k(p)=\mathcal I(sp)$, of degrees $-1$ and
$-2$. On the cone the two identities become
\[
 \partial h=\varepsilon_m,\qquad \partial k=\varphi h.
\]
Thus $h$ is a nullhomotopy of $\varepsilon_m$ and $k$ is a nullhomotopy
of its trace. Since $P_m$ is graded free and $\varphi$ is surjective, lift $k$
to a graded $R$-linear map $\widetilde k\colon P_m\to R$ with
$\varphi\widetilde k=k$. Then
\[
 h_I=h-\partial\widetilde k,\qquad
 \partial h_I=\varepsilon_m,\qquad
 \varphi h_I=\varphi h-\partial k=0.
\]
This is an $I$-valued nullhomotopy, so $o_m=0$.

\emph{From an ideal-valued nullhomotopy to a relative section.}
Suppose $o_m=0$ and represent it by an $R$-linear nullhomotopy
$h\colon P_m\to I$ of $\varepsilon_m$. The conversion
Theorem~\ref{thm:ideal-nullhomotopy} gives a CDGA retraction
$\rho\colon E_m\to R$ with $\rho(Z)\subset I$. Hence
$\varphi\rho=\varphi_m=\overline\varphi\zeta_m$, since these maps
agree on $R$ and kill $Z$. Lemma~\ref{lem:quotient-to-ganea} now gives
$\relcat(f_\Q)\le m-1=n$.
\end{proof}

\subsection{The point inclusion and Hess's theorem}

For the inclusion of a point the trace has nowhere to take its values.
Indeed, the split model has $A=\Q$, while the controlled resolutions of
Proposition~\ref{prop:controlled-resolution} have all their cells in
degree at least $2m$ by \eqref{eq:resolution-degree-bound}. A map of
degree $-1$ from these cells to $\Q$ vanishes. The relative criterion
therefore recovers the equality between module and rational category.

\begin{corollary}[Hess's theorem revisited]\label{cor:hess-revisited}
Let $X$ be a simply connected space of finite rational type.
Then $\relcat(*\to X_\Q)=\cat_0(X)=\Mcat(X)$.
\end{corollary}

\begin{proof}
Take $f\colon *\to X$, retracted by the constant map. The split model is
the augmentation $\varphi\colon R=\Lambda W\to\Q$ of the minimal model of
$X$, with $I=\Lambda^+W$ and $I^{m}=\Lambda^{\ge m}W$. Here
$\Hom_R(P_m,\Q)=\Hom_\Q(G_m,\Q)$ is the graded dual of the generating
space of $P_m$, which is concentrated in degrees at least $2m\ge2$;
hence it has no component of degree $-1$, and
$\Hom_{\Dcat(R)}(I^m,\Q[-1])=0$. By \eqref{eq:trace-sequence},
$o_m=0$ if and only if $\iota_*(o_m)=0$. Theorem~\ref{thm:relative-category} and
\eqref{eq:msecat-derived} now give $\relcat(*\to X_\Q)=\msecat(*\to X_\Q)$,
and \eqref{eq:three-inequalities} places $\secat(*\to X_\Q)=\cat_0(X)$
between them. Finally $\msecat(*\to X_\Q)=\Mcat(X)$
\cite[Proposition~5.6]{FGKV2006}.
\end{proof}

The conversion theorem follows Hess's construction, and this corollary
records its absolute case in the relative formulation. For a general
source the trace can be nonzero; a bound on its possible degrees can
still force it to vanish. Corollary~\ref{cor:degree-criterion} develops
that argument alongside the other consequences of the comparison
theorems.

\section{Two ways into the ideal}
\label{sec:ideal}

The obstruction $o_m=[I^m\hookrightarrow I]$ remembers the target
ideal; its image $\iota_*(o_m)=[I^m\hookrightarrow R]$ forgets it.
Theorem~\ref{thm:relative-category} and \eqref{eq:msecat-derived} therefore reduce the comparison
of the two invariants to one question: how can an $R$-valued nullhomotopy
be replaced by an ideal-valued one? We answer it in two ways. Comparing
consecutive resolutions supplies the missing factor of $I$ after
increasing $m$ by one. Poincar\'e duality supplies that factor through
evaluation, with no change of exponent. These are the mechanisms behind
the general bound and the main equality theorem.

\subsection{Shifting the exponent}\label{subsec:shift}

The controlled resolutions of Proposition~\ref{prop:controlled-resolution}
come with a comparison of consecutive powers. We use here only the
conclusion of Proposition~\ref{prop:power-comparison}: there is a
cochain map $j_{m,1}\colon P_{m+1}\to P_m$ satisfying
\[
 \varepsilon_mj_{m,1}=\varepsilon_{m+1},
 \qquad j_{m,1}(P_{m+1})\subset IP_m.
\]
The coefficient condition is what changes the target of a nullhomotopy.

\begin{theorem}[Exponent shift]\label{thm:shift}
For every $m\ge1$, if $\iota_*(o_m)=[I^m\hookrightarrow R]=0$ in
$\Dcat(R)$, then $o_{m+1}=[I^{m+1}\hookrightarrow I]=0$.
\end{theorem}

\begin{proof}
Since $P_m$ is semifree, the hypothesis is represented by an $R$-linear
map $h\colon P_m\to R$ of degree $-1$ with $d_Rh+hD_m=\varepsilon_m$.
Put $h'=hj_{m,1}\colon P_{m+1}\to R$. Then
$d_Rh'+h'D_{m+1}=\varepsilon_mj_{m,1}=\varepsilon_{m+1}$, and by graded
$R$-linearity of $h$,
\[
 h'(P_{m+1})=h(j_{m,1}(P_{m+1}))\subset h(IP_m)\subset I\cdot h(P_m)\subset I .
\]
So $h'$ is a nullhomotopy of $\varepsilon_{m+1}$ with values in $I$.
\end{proof}

\begin{example}[The two-sphere: the complete comparison data]\label{ex:sphere-comparison}
The minimal model of $S^2$ is $M=(\Lambda(x_2,a_3),d)$, with
$dx=0$ and $da=x^2$. In the diagonal model put
\[
 \begin{gathered}
 R=M_L\otimes M_R=M_L\otimes\Lambda(\bar x,\bar a),\qquad
 I=(\bar x,\bar a),\\
 \bar x=x_R-x_L,\qquad \bar a=a_R-a_L.
 \end{gathered}
\]
Write $x,a$ for the left-hand generators and $\lambda=2x+\bar x$.
Then $d\bar x=0$, $d\bar a=\bar x\lambda$, and
$I^m=\bar x^{m-1}I$. For each $m\geq1$, let $P_m^{S^2}$ be free
over $R$ on $e_j^{(m)},f_j^{(m)}$, $j\geq0$, of degrees
$2m+2j$ and $2m+2j+1$, respectively. Suppressing $(m)$ and setting
$e_{-1}=f_{-1}=0$, its differential and augmentation are
\begin{align}
 De_j&=\bar x f_{j-1}-\bar a e_{j-1},&
 Df_j&=\lambda e_j-\bar a f_{j-1},
 \label{eq:sphere-resolution-differential}\\
 \varepsilon_m(e_j)&=\delta_{j0}\bar x^m,&
 \varepsilon_m(f_j)&=\delta_{j0}\bar x^{m-1}\bar a.
 \label{eq:sphere-resolution-augmentation}
\end{align}
The differential matrices are independent of $m$, so the comparison
lifting $I^{m+1}\hookrightarrow I^m$ is
\begin{equation}\label{eq:sphere-consecutive-comparison}
 j_{m,1}(e_j^{(m+1)})=\bar x e_j^{(m)},\qquad
 j_{m,1}(f_j^{(m+1)})=\bar x f_j^{(m)}.
\end{equation}
Its coefficients lie in $I$, as the exponent-shift argument requires.
The degree-$-1$ binary comparison is given by graded $R$-linear extension of
\begin{align}
 \sigma_m(e_i,e_j)&=0, &
 \sigma_m(e_i,f_j)&=
 \bar x^{m-1}\binom{i+j+1}{i}e_{i+j+1},\notag\\
 \sigma_m(f_j,e_i)&=-\sigma_m(e_i,f_j), &
 \sigma_m(f_i,f_j)&=
 -\bar x^{m-1}\binom{i+j+2}{i+1}f_{i+j+1}.
 \label{eq:sphere-binary-comparison}
\end{align}
Its image lies in $I^{m-1}\ker\varepsilon_m$, and
\[
 D\sigma_m+\sigma_mD_\otimes:
 p\otimes q\longmapsto\varepsilon_m(p)q
 -(-1)^{|p||q|}\varepsilon_m(q)p.
\]
Appendix~\ref{app:sphere-resolutions} verifies the resolution and these
identities. This example displays the coefficient estimates; its
associative comparisons require no higher corrections, as the ternary
calculation in Section~\ref{sec:filtered-model} will show.

Since $\MTC(S^2)=2$, an $R$-valued nullhomotopy of $\varepsilon_3$
exists. The comparison produces an $I$-valued nullhomotopy of
$\varepsilon_4$. Duality keeps the exponent at three;
Example~\ref{ex:sphere-ideal-valued} gives that homotopy.
\end{example}

\begin{proof}[Proof of Theorem~\ref{thm:main}]
Put $q=\msecat(f_\Q)$. The inequalities $\msecat\le\secat\le\relcat$ are
\eqref{eq:three-inequalities}; they settle the case $q=\infty$, and
Lemma~\ref{lem:zero-threshold} settles the case $q=0$. For
$1\le q<\infty$, \eqref{eq:msecat-derived} gives $\iota_*(o_{q+1})=0$,
Theorem~\ref{thm:shift} gives $o_{q+2}=0$, and Theorem~\ref{thm:relative-category} with $n=q+1$ gives
$\relcat(f_\Q)\le q+1$.
\end{proof}

\begin{proof}[Proof of Corollary~\ref{cor:diagonal-bound}]
Either projection retracts $\Delta\colon X\to X\times X$, and
rationalization preserves finite products of simply connected spaces, so $\Delta_\Q$ is the diagonal of $X_\Q$. By definition
$\secat(\Delta_\Q)=\TC_0(X)$ and $\msecat(\Delta_\Q)=\MTC(X)$, and
$\relcat(\Delta_\Q)=\TC_0^M(X)$ by \cite[Theorem~12]{CGV2014}, since a
rational CW model is locally equiconnected. Theorem~\ref{thm:main} gives
the chain.
\end{proof}

The diagonal also inherits the exact criterion of Theorem~\ref{thm:relative-category}: for the
minimal model $M$ of $X$, $R=M\otimes M$ and $I=\ker\mu$,
\begin{equation}\label{eq:monoidal-criterion}
 \TC_0^M(X)\le n\quad\Longleftrightarrow\quad
 [I^{n+1}\hookrightarrow I]=0\ \text{ in }\Dcat(R).
\end{equation}

\subsection{Poincar\'e duality}\label{subsec:duality}

To keep the same power $I^m$, we shall use a pairing which already
takes its values in $I$. Write $\DD_RL=\RHom_R(L,R)$ for the derived
dual of a differential $R$-module $L$. Derived evaluation and its
adjoint are the canonical maps
\[
 \ev_L\colon L\Lt_R\DD_RL\longrightarrow R,
 \qquad \eta_L\colon L\longrightarrow\DD_R\DD_RL.
\]
On homogeneous representatives they are given by
$\ev_L(x\otimes u)=(-1)^{|x||u|}u(x)$ and
$\eta_L(x)(u)=(-1)^{|x||u|}u(x)$. The argument needs this natural
biduality map to be an isomorphism, and evaluation to factor through the
ideal. Both properties come from finite Poincar\'e duality.

\begin{definition}[Differential Poincar\'e duality algebra]\label{def:differential-pd}
A finite-dimensional connected CDGA $S$ is a \emph{differential
Poincar\'e duality algebra of formal dimension $N$} if it has an
orientation $\epsilon\colon S^N\to\Q$ with
$\epsilon(dS^{N-1})=0$ such that
every pairing $S^k\otimes S^{N-k}\to\Q$, $x\otimes y\mapsto\epsilon(xy)$,
is nondegenerate \cite[Definition~2.2]{LS2008}.
\end{definition}

Lambrechts and Stanley construct such a model for every simply
connected space of finite rational type whose rational cohomology satisfies
Poincar\'e duality \cite[Theorem~1.1]{LS2008}. The property we need is
stronger than duality of the underlying vector space: evaluation on any
differential ideal stays inside that ideal. The double-annihilator
identity is what ensures this compatibility.

\begin{lemma}[Evaluation preserves the ideal]\label{lem:pd-evaluation}
Let $S$ be a finite differential Poincar\'e duality algebra and $J$ a
differential ideal of $S$. Then the biduality morphism
$J\to\DD_S\DD_SJ$ is an isomorphism in $\Dcat(S)$, and derived
evaluation $J\Lt_S\DD_SJ\to S$ factors through $J\hookrightarrow S$.
\end{lemma}

\begin{proof}
\emph{The double annihilator.} Extend $\epsilon$ by zero outside degree $N$. Let $\Ann(J)=\{a\in S:aJ=0\}$; it is a
differential ideal because $J$ is. Nondegeneracy of the orientation
pairing identifies $\Ann(J)$ with the orthogonal of $J$: if
$\epsilon(aJ)=0$, then $\epsilon(ajb)=0$ for all $j\in J$ and $b\in S$
because $jb\in J$, hence $aj=0$ for all $j$. Taking orthogonals twice
gives
\begin{equation}\label{eq:pd-double-annihilator}
 \Ann(\Ann(J))=J .
\end{equation}

\emph{Evaluation preserves $J$.} Let $u\colon J\to S$ be a homogeneous $S$-linear map,
$a\in\Ann(J)$ and $x\in J$. Graded linearity gives
$a\,u(x)=(-1)^{|a||u|}u(ax)=0$, so $u(x)\in\Ann(\Ann(J))=J$. Hence
ordinary evaluation factors as a cochain map
$J\otimes_S\Hom_S(J,S)\to J\hookrightarrow S$.

\emph{The derived dual and natural biduality.} The orientation gives an isomorphism of differential
$S$-modules $S\cong S^\vee[-N]$, $x\mapsto\epsilon(x\,\cdot)$, with
the shifted module structure. Explicitly,
$d_{S^\vee}\phi=-(-1)^{|\phi|}\phi d_S$ and
$(r\phi)(y)=(-1)^{|r||\phi|}\phi(ry)$; the shift $[-N]$ multiplies
the differential by $(-1)^N$ and the action of $r$ by $(-1)^{N|r|}$.
For $\phi_x=\epsilon(x\,\cdot)$, these conventions give
$d_{S^\vee[-N]}\phi_x(y)=(-1)^{|x|+1}\epsilon(x\,dy)=\epsilon(dx\,y)$.
The coinduction adjunction then gives,
for every differential $S$-module $L$,
\begin{equation}\label{eq:pd-dual-identification}
 \Hom_S(L,S)\cong\Hom_S(L,S^\vee)[-N]\cong L^\vee[-N].
\end{equation}
Concretely, the composite sends $u$ to $\epsilon\circ u$, viewed in the
shifted dual. Since $\Q$-linear duality preserves acyclicity,
$\Hom_S(-,S)$ sends acyclic modules to acyclic complexes. Thus $S$ is
homotopically injective, and ordinary Hom into $S$ computes $\DD_S$.

If $\eta_J(x)=0$, then
$(-1)^{|x||u|}\epsilon(u(x))=0$ for every homogeneous $u\colon J\to S$.
By \eqref{eq:pd-dual-identification}, these functionals exhaust
$J^\vee$, so $x=0$. Applying the same identification twice shows that
$J$ and $\Hom_S(\Hom_S(J,S),S)$ have the same finite dimension.
Consequently the natural cochain map $\eta_J$ is an isomorphism. This is
where finite dimensionality is used.

\emph{Derived evaluation.} Put $C_0=\Hom_S(J,S)$ and
choose semifree replacements $q\colon P\to J$ and $v\colon P'\to C_0$.
Since $S$ is homotopically injective, $C_0\to\Hom_S(P,S)$ is a
quasi-isomorphism, so $C_0$ represents $\DD_SJ$, and the cochain map
$P\otimes_SP'\to J$, $p\otimes u\mapsto(-1)^{|p||u|}v(u)(q(p))$, followed
by $J\hookrightarrow S$, represents derived evaluation. This is the
required factorization.
\end{proof}

Return to the split minimal model and its controlled resolutions.
The iterated comparison of Remark~\ref{rem:higher-power-comparisons} is
\[
 j_{1,m-1}=j_{1,1}j_{2,1}\cdots j_{m-1,1}\colon P_m\longrightarrow P_1,
 \qquad j_{1,m-1}(P_m)\subset I^{m-1}P_1;
\]
for $m=1$ it is the identity. It represents $I^m\hookrightarrow I$.
The comparison supplies $m-1$ factors of the ideal; evaluation supplies
the last. Biduality will then recover the inclusion $I^m\hookrightarrow
I$ from this pairing.

\begin{theorem}[The duality theorem]\label{thm:duality}
Suppose that $\eta_I\colon I\to\DD_R\DD_RI$ is an isomorphism in
$\Dcat(R)$ and that derived evaluation factors as
\begin{equation}\label{eq:pd-ideal-evaluation}
 I\Lt_R\DD_RI\xrightarrow{\ \widetilde{\ev}\ }I\lhook\joinrel\longrightarrow R .
\end{equation}
Then, for every $m\ge1$, $\iota_*(o_m)=0$ implies $o_m=0$.
\end{theorem}

\begin{proof}
Choose a semifree resolution $P_C\to\Hom_R(P_1,R)$ with an
$\mathbb N$-indexed semifree filtration, so that $P_C$ represents
$\DD_RI$. The controlled resolution $P_1$ has such a filtration by cell
degree. Filtering $P_1\otimes_RP_C$ by the sum of the two filtration
indices makes it semifree, so it computes $I\Lt_R\DD_RI$, and
its semifreeness lets us represent $\widetilde{\ev}$ by an actual
cochain map
$e\colon P_1\otimes_RP_C\to I$ whose composite with $\iota$ represents
$\ev$. By $R$-linearity and the coefficient bound on $j_{1,m-1}$,
\[
 e\bigl(j_{1,m-1}\otimes1\bigr)(P_m\otimes_RP_C)\subset I^{m-1}\cdot I=I^m ,
\]
and since $I^m$ is a differential ideal this defines a cochain map
$t_m\colon P_m\otimes_RP_C\to I^m$. As before, $\varepsilon_m$ denotes the augmentation into $I^m$ and,
when composed with inclusion, into $R$. Write
$\operatorname{incl}_m:I^m\hookrightarrow I$, so that
$o_m=[\operatorname{incl}_m]$. Naturality of evaluation identifies
$\ev\circ(\operatorname{incl}_m\Lt_R1)$ with the composite
\[
 P_m\otimes_RP_C\xrightarrow{\ t_m\ }I^m\lhook\joinrel\longrightarrow R
\]
in $\Dcat(R)$. If $\iota_*(o_m)=0$, the last inclusion is zero in
$\Dcat(R)$, so this pairing vanishes. Under the derived tensor--Hom
adjunction its adjoint is $\eta_I\circ \operatorname{incl}_m\colon I^m\to\DD_R\DD_RI$.
Since $\eta_I$ is an isomorphism, $\operatorname{incl}_m=0$ in $\Dcat(R)$, that is, $o_m=0$.
\end{proof}

\begin{corollary}[An evaluation criterion for equality]\label{cor:duality-equality}
Let $f$ be a rationally retractive map whose split model satisfies the
hypotheses of Theorem~\ref{thm:duality}. Then
$\msecat(f_\Q)=\secat(f_\Q)=\relcat(f_\Q)$.
\end{corollary}

\begin{proof}
If $q=\msecat(f_\Q)<\infty$, then $\iota_*(o_{q+1})=0$ by
\eqref{eq:msecat-derived}, $o_{q+1}=0$ by Theorem~\ref{thm:duality}, and
$\relcat(f_\Q)\le q$ by Theorem~\ref{thm:relative-category}; the case $q=0$ is included. If
$q=\infty$, \eqref{eq:three-inequalities} gives the conclusion.
\end{proof}

The finite Poincar\'e duality model is not a minimal Sullivan algebra,
whereas the comparison maps live on the split minimal model. The
following lemma carries the two hypotheses of Theorem~\ref{thm:duality} from one to the
other.

\begin{lemma}[Transport along quasi-isomorphisms]\label{lem:pd-transport}
Consider a strictly commutative square of CDGA morphisms with surjective
rows and CDGA quasi-isomorphisms $\alpha$, $\beta$ as columns,
\[
\begin{tikzcd}[column sep=3.5em,row sep=1.6em]
 R\arrow[r]\arrow[d,"\alpha"'] & A\arrow[d,"\beta"]\\
 S\arrow[r] & B\rlap{,}
\end{tikzcd}
\]
and let $I$ and $J$ be the kernels of the rows. If $\eta_J$ is an
isomorphism and derived evaluation of $J$ factors through
$J\hookrightarrow S$, then $\eta_I$ is an isomorphism and derived
evaluation of $I$ factors through $I\hookrightarrow R$.
\end{lemma}

\begin{proof}
The short exact sequences of kernels show that $I\to J$ is a
quasi-isomorphism. Extension and restriction of scalars along $\alpha$
are inverse derived equivalences \cite[Theorem~4.3]{SchwedeShipley2000};
the tensor hypothesis holds because semifree modules are K-flat.
Write $\alpha_!=S\Lt_R-$. The induced isomorphism $\alpha_!I\cong J$
identifies $\alpha_!(I\hookrightarrow R)$ with $J\hookrightarrow S$.
This equivalence is symmetric monoidal, since
\[
 \alpha_!(L\Lt_RL')\cong(\alpha_!L)\Lt_S(\alpha_!L'),
 \qquad \alpha_!R\cong S.
\]

A symmetric monoidal equivalence preserves the internal Hom and its
evaluation: applying the equivalence to the tensor--Hom adjunction
identifies the objects representing that adjunction. In particular, the
adjoint of $\alpha_!(\ev_L)$ is a natural isomorphism
\[
 \chi_L\colon\alpha_!\DD_RL\xrightarrow{\ \cong\ }\DD_S\alpha_!L.
\]
It carries evaluation to evaluation and preserves the symmetry of the
tensor factors. It therefore carries the adjoint biduality map to the
canonical biduality map over $S$.
Taking $L=I$ and using $\alpha_!I\cong J$ shows that $\eta_I$ is an
isomorphism whenever $\eta_J$ is.

The factorization of $\ev_J$ through $J$ now gives a morphism
$\alpha_!(I\Lt_R\DD_RI)\to\alpha_!I$. Fullness lifts it to
$w\colon I\Lt_R\DD_RI\to I$, and faithfulness gives
$\ev_I=\iota w$. This is the required ideal-valued factorization.
\end{proof}

\begin{remark}[What is transported, and what is not]\label{rem:powers-not-transported}
The lemma transports the kernel, its biduality and its evaluation, which
are derived constructions; it says nothing about the powers of the
kernel, and indeed these are not preserved. For the two-sphere, the
finite Poincar\'e duality model is $B=\Q[x]/(x^2)$ with zero
differential, and the kernel $J$ of $\mu_B\colon B\otimes B\to B$ is
spanned by $\bar x=x_R-x_L$ and $x_Lx_R$, so that $J^3=0$. In the
minimal model, $I^3=\bar x^2I$, and multiplication by $\bar x^2$ is an injective
map of degree $4$. The split exact sequence gives
$H(I)\cong\ker(H(R)\to H(A))$, which is one dimensional in each
of degrees $2$ and $4$, and zero otherwise. Hence $H^k(I^3)\cong H^{k-4}(I)$, which is
nonzero in degrees $6$ and $8$, whereas $H(J^3)=0$. The quasi-isomorphism $I\to J$ induced by the square
does not induce a quasi-isomorphism $I^3\to J^3$. This is why the proof
keeps the powers $I^m$, their resolutions and the comparison maps in the
split minimal model, and only borrows the duality data from the finite
model.
\end{remark}

\begin{proposition}[Compatible finite Poincar\'e models]\label{prop:pd-compatible-map}
Let $f$ be a rationally retractive map between simply connected spaces of
finite rational type whose split minimal model
$\varphi\colon R\to A$ fits into a strictly commutative square
\[
\begin{tikzcd}[column sep=3.5em,row sep=1.6em]
 R\arrow[r,"\varphi"]\arrow[d,"\simeq"'] & A\arrow[d,"\simeq"]\\
 S\arrow[r] & B\rlap{,}
\end{tikzcd}
\]
with surjective rows and CDGA quasi-isomorphisms as columns, in which
$S$ is a finite differential Poincar\'e
duality algebra. Then $\msecat(f_\Q)=\secat(f_\Q)=\relcat(f_\Q)$.
\end{proposition}

\begin{proof}
Lemma~\ref{lem:pd-evaluation} gives the two properties for the kernel of
$S\to B$, Lemma~\ref{lem:pd-transport} transports them to $I$, and
Corollary~\ref{cor:duality-equality} concludes.
\end{proof}

\begin{proof}[Proof of Theorem~\ref{thm:poincare-main} and Corollary~\ref{cor:elliptic-main}]
Let $B$ be a finite differential Poincar\'e duality algebra modelling $X$
\cite[Theorem~1.1]{LS2008}, and let $\alpha\colon M=(\Lambda V,d)\to B$ be
a quasi-isomorphism from the minimal model. Such a direct morphism
represents the quasi-isomorphism zigzag of that theorem, since $M$ is
cofibrant and CDGAs are fibrant. For every $r\ge2$ the
iterated multiplications give a strictly commutative square
\[
\begin{tikzcd}[column sep=3.5em,row sep=1.6em]
 R=M^{\otimes r}\arrow[r,"\mu^{(r)}_M"]\arrow[d,"\alpha^{\otimes r}"'] & M\arrow[d,"\alpha"]\\
 S=B^{\otimes r}\arrow[r,"\mu^{(r)}_B"] & B\rlap{,}
\end{tikzcd}
\]
with surjective rows and CDGA quasi-isomorphisms as columns. Tensor products
preserve differential Poincar\'e duality: for oriented algebras of
formal dimensions $N,N'$, use $\epsilon\otimes\epsilon'$ in degree
$N+N'$. Its pairing matrix is, up to Koszul signs, the tensor product
of two invertible pairing matrices, and it vanishes on differentials
because both orientations do. Iterating this argument makes $S$ a
finite differential Poincar\'e duality algebra.
In the coordinates of Example~\ref{ex:diagonal-model}, the top row is the
split minimal model of $\Delta_r$. Proposition~\ref{prop:pd-compatible-map}
gives
\[
 \msecat((\Delta_r)_\Q)=\secat((\Delta_r)_\Q)=\relcat((\Delta_r)_\Q).
\]
For $r=2$ these are $\MTC(X)$, $\TC_0(X)$ and $\TC_0^M(X)$, which proves
Theorem~\ref{thm:poincare-main}. Corollary~\ref{cor:elliptic-main}
follows since the cohomology of a simply connected elliptic space is a
Poincar\'e duality algebra \cite[Proposition~38.3]{FHT2001}.
\end{proof}

\begin{example}[The two-sphere, concluded]\label{ex:sphere-ideal-valued}
At the module threshold $q=\MTC(S^2)=2$, the equality theorem asserts
$o_3=0$. On the cells introduced in
Example~\ref{ex:sphere-comparison}, omit the superscript $(3)$ and set
\[
 h(e_0)=\omega=4\bar x a+(\bar x-2x)\bar a,
 \qquad h(f_0)=4a\bar a,
 \qquad h(e_j)=h(f_j)=0\quad(j\ge1).
\]
These values define a graded $R$-linear map of degree $-1$ into $I$.
The first equation is $d\omega=\bar x^3$; the value on $f_0$ and the
remaining module equations are checked in
Appendix~\ref{app:sphere-resolutions}.

The distinction between $I$ and $I^2$ is already visible on $e_0$.
One has $\omega\in I\setminus I^2$, and
\[
 (I^2)^5=\Q\,\bar x\bar a,
 \qquad d(\bar x\bar a)=2x\bar x^2+\bar x^3.
\]
The two monomials on the right are linearly independent, so
$\bar x^3$ has no primitive in $I^2$. An ideal-valued nullhomotopy
therefore exists at exponent three, but none can have all its values
in $I^2$.
\end{example}

The next section draws the consequences of the two comparison
mechanisms. Their technical inputs are constructed in
Sections~\ref{sec:toolbox}--\ref{sec:filtered-model}.

\section{Examples, products, and degree criteria}\label{sec:consequences}
\label{sec:products}

The comparison theorems now yield three kinds of consequences: equalities
for Poincar\'e duality spaces, product and growth formulas, and degree
criteria for general retractive maps. The nonformal example below
shows why the duality hypothesis is useful beyond the formal case.

\subsection{Duality beyond the formal case}

\begin{example}[A nonformal elliptic space]\label{ex:nonformal-pd}
Let $E$ be the rational realization of
\[
 \mathcal M=(\Lambda(x_2,y_2,a_3,b_3,c_3),d),\qquad
 da=x^2,\quad db=xy,\quad dc=y^2,\quad dx=dy=0.
\]
The value $\TC_0(E)=5$ follows from the bounds of Hamoun, Rami and
Vandembroucq \cite{HRV2023}. The new conclusion is
$\MTC(E)=\TC_0^M(E)=5$.
The algebra is finite free over $\Q[x,y]$, so its cohomology is a
finite module over $\Q[x,y]/(x^2,xy,y^2)$ and hence finite dimensional.
Thus $E$ is elliptic and has Poincar\'e duality cohomology
\cite[Proposition~38.3]{FHT2001}, of formal dimension
$3+3+3-(2-1)-(2-1)=7$ \cite[Theorem~32.6]{FHT2001}.

Nevertheless $E$ is not formal. The Massey product
$\langle[x],[x],[y]\rangle$ is represented, up to an overall sign, by
$xb-ay$. Its differential is zero, and it is not a boundary because
all degree-four cochains are closed polynomials in $x,y$.
Its indeterminacy is $[x]H^3+H^3[y]=0$, since the differentials of
$a,b,c$ are linearly independent. This nonzero Massey product obstructs
formality.

With $(x_1,x_2,y_1,y_2,y_3)=(x,y,a,c,b)$, this model belongs to the
family of \cite[\S5.2, Example~1]{HRV2023}.
Since $\mathcal M$ is elliptic with quadratic differential,
$\cat_0(E)=3$, the number of odd generators
\cite[\S1, following Theorem~1.2]{HRV2023}.
Its homotopy Euler characteristic is
$\chi_\pi(E)=\dim_\Q\pi_{\mathrm{even}}(E)-\dim_\Q\pi_{\mathrm{odd}}(E)=2-3=-1$.
The model is pure: its even generators are closed, and the differential
of each odd generator is a polynomial in the even generators.
The extension $(\Lambda Z,d)=(\Lambda(x,y,a,c),d)\hookrightarrow\mathcal M$
has the same even generators as $\mathcal M$, and $(\Lambda Z,d)$ is an
$F_0$-model: it is elliptic with equally many even and odd generators.
These are the extension hypotheses of
\cite[Theorem~4.2]{HRV2023}, which gives
\[
 \TC_0(E)\le2\cat_0(E)+\chi_\pi(E)=5.
\]
For the lower bound, put
\[
 C=\bigl(\Q[x,y]/(x^2,y^2)\otimes\Lambda(a,b,c),\bar d\bigr),
 \qquad \bar da=\bar dc=0,\quad\bar db=xy.
\]
The invariant $L=L(\mathcal M,\{x,y\})$ of
\cite[\S5.2]{HRV2023} is the maximal length of a
nonzero product of cohomology classes in $H(C)$ with odd wordlength
in $x,y$. Here $[x][yb]=[xyb]\ne0$: every boundary in $C$ is free of
$b$, because only $b$ has a nonzero differential. Every product of three
such classes vanishes since $x^2=y^2=0$. Thus $L=2$, and
\cite[Theorem~5.1]{HRV2023} gives
$\TC_0(E)\ge\dim\langle a,b,c\rangle+L=5$.
The known value $\TC_0(E)=5$ and
Corollary~\ref{cor:elliptic-main} therefore give
\[
 \MTC(E)=\TC_0(E)=\TC_0^M(E)=5.
\]
By contrast, $\operatorname{zcl}H(E;\Q)=3$. Indeed, write $x,y$ for
their cohomology classes and put $u=[xb-ay]$, $v=[xc-yb]$,
$\omega=xv=-yu$. The class $\omega$ is nonzero: its representative
$x^2c-xyb$ has an $xyb$ term, while the boundaries
$d(ab),d(ac),d(bc)$ have none. Poincar\'e duality gives the basis
$1,x,y,u,v,\omega$, whose only nonzero positive-degree products are
$xv=\omega$ and $yu=-\omega$. With $\bar z=1\otimes z-z\otimes1$,
\[
 \bar x^{\,2}\bar v=2(\omega\otimes x-x\otimes\omega)\ne0.
\]
In a product of four such generators, each surviving term must pair
one degree-two class with one degree-five class in each tensor factor.
The resulting terms in $\Q(\omega\otimes\omega)$ cancel because the
degree-five classes are odd. Thus $\operatorname{zcl}H(E;\Q)=3<5$.
\end{example}

\begin{corollary}[Higher diagonals]\label{cor:higher-pd}
Under the hypotheses of Theorem~\ref{thm:poincare-main}, for every
$r\ge2$,
\[
 \MTC_r(X)=\TC_{r,0}(X)=\relcat((\Delta_r)_\Q).
\]
\end{corollary}

\begin{proof}
The proof of Theorem~\ref{thm:poincare-main} establishes the equality
for every iterated diagonal; the first two
terms are $\MTC_r(X)$ and $\TC_{r,0}(X)$ by definition, the higher
evaluation fibration being a fibration replacement of $\Delta_r$
\cite{Rudyak2010,CV2017}.
\end{proof}

\begin{corollary}[The cohomological approximation]\label{cor:pd-cohomological-equality}
Under the hypotheses of Theorem~\ref{thm:poincare-main},
$\HTC(X)=\MTC(X)=\TC_0(X)=\TC_0^M(X)$.
\end{corollary}

\begin{proof}
Carrasquel-Vera, Kahl and Vandembroucq prove $\HTC(X)=\MTC(X)$ under
Poincar\'e duality \cite[Corollary~4.1]{CKV2016}. Apply
Theorem~\ref{thm:poincare-main}.
\end{proof}

\subsection{Products under Poincar\'e duality}

Poincar\'e duality gives exact additivity on products; the one-power
bound determines the growth of arbitrary Cartesian powers up to a
bounded error. The common input is the additivity of module sectional
category proved by Carrasquel, Parent and Vandembroucq
\cite[Theorem~1]{CPV2018}. Applying it to fibration replacements of
the rationalized maps, which admit homotopy retractions, gives
\begin{equation}\label{eq:module-additivity}
 \msecat\bigl((f\times f')_\Q\bigr)=\msecat(f_\Q)+\msecat(f'_\Q)
\end{equation}
for rationally retractive maps $f$ and $f'$ with finite module
sectional categories.

\begin{corollary}[Product additivity under Poincar\'e duality]\label{cor:PD-products}
Let $X$ and $Y$ be simply connected spaces of finite rational type whose rational cohomology algebras satisfy Poincar\'e duality. Then
\[
 \TC_0(X\times Y)=\TC_0(X)+\TC_0(Y),
\]
and the same identity holds for $\MTC$, for $\TC_0^M$, and for every
$\TC_{r,0}$, $r\ge2$. In particular
$\TC_0(X^n)=\TC_0^M(X^n)=\MTC(X^n)=n\TC_0(X)$ for every $n\ge1$.
\end{corollary}

\begin{proof}
The tensor-product argument in the proof of
Theorem~\ref{thm:poincare-main} shows that $X\times Y$ also has
Poincar\'e duality cohomology. These spaces have finite rational
higher topological complexity: if $M\to B$ is a quasi-isomorphism
to a finite connected CDGA model supplied by \cite[Theorem~1.1]{LS2008},
the image in $B^{\otimes r}$ of the multiplication kernel
of $M^{\otimes r}$ lies in the nilpotent positive-degree ideal.
Thus $M^{\otimes r}\to B^{\otimes r}$ factors through the quotient
by a sufficiently high power of that kernel. Inverting this
quasi-isomorphism in the homotopy category gives a left inverse of
the quotient map. Represent it on a relative Sullivan model; homotopy
extension along its inclusion straightens its restriction to
$M^{\otimes r}$ to the identity. Thus the quotient map admits a
homotopy retraction, and Theorem~\ref{thm:quotient-criteria} gives
$\TC_{r,0}<\infty$.
Then
$\TC_0(X\times Y)=\MTC(X\times Y)=\MTC(X)+\MTC(Y)=\TC_0(X)+\TC_0(Y)$, and
likewise for the monoidal invariant; for higher complexities use
Corollary~\ref{cor:higher-pd} and \cite[Corollary~2]{CPV2018}. Induction
gives the Cartesian powers.
\end{proof}

\subsection{Stable growth for retractive maps}

For general retractive maps, module additivity combines with the
one-power bound to control every Cartesian power. To identify when
sectional category itself is additive, we also need the following
kernel-power argument of Carrasquel-Vera
\cite[Proposition~6 and Corollary~7\textup{(i)}]{CV2016}.

\begin{lemma}[Product subadditivity]\label{lem:secat-subadditivity}
For rationally retractive maps $f$ and $f'$,
\begin{equation}\label{eq:secat-subadditivity}
 \secat\bigl((f\times f')_\Q\bigr)\le\secat(f_\Q)+\secat(f'_\Q).
\end{equation}
\end{lemma}

\begin{proof}
Only the case $a=\secat(f_\Q)$ and $b=\secat(f'_\Q)$ finite needs proof.
Let $\varphi:R\to A$ and $\psi:S\to B$ be split models with kernels
$I$ and $J$. Their tensor product is a split model of $(f\times f')_\Q$
with kernel $I_\times=I\otimes S+R\otimes J$. Every term in the
expansion of $I_\times^{a+b+1}$ has at least $a+1$ factors in the first
summand or at least $b+1$ in the second. Hence
\[
 I_\times^{a+b+1}\subset I^{a+1}\otimes S+R\otimes J^{b+1},
\]
and the tensor product of the two quotient maps factors as
\[
 R\otimes S\longrightarrow (R\otimes S)/I_\times^{a+b+1}
 \longrightarrow (R/I^{a+1})\otimes(S/J^{b+1}).
\]
By Theorem~\ref{thm:quotient-criteria}, each quotient map has a homotopy
retraction. Choose relative Sullivan models with surjective
quasi-isomorphisms for the first map and for the displayed composite,
using the tensor product of the individual models for the composite.
Sullivan lifting gives a comparison from the first model to the second
under $R\otimes S$. Composing it with the tensor product of the two
algebra retractions gives a retraction of the first model.
The same quotient criterion proves
\eqref{eq:secat-subadditivity}.
\end{proof}

\begin{corollary}[Stable growth]\label{thm:power-consequences}
Let $f$ be rationally retractive with $q=\msecat(f_\Q)<\infty$, and put
$c_n=\secat((f^{\times n})_\Q)$. Then $nq\le c_n\le nq+1$ for every
$n\ge1$, and
\[
 \lim_{n\to\infty}\frac{c_n}{n}=q,\qquad
 q=\left\lfloor\frac{c_n}{n}\right\rfloor\quad(n\ge2).
\]
Equality $c_n=nc_1$ holds for some $n\ge2$ if and only if $c_1=q$,
in which case it holds for every $n\ge1$. If $c_1=q+1$, then
$nc_1-c_n\in\{n-1,n\}$ for every $n\ge2$.
\end{corollary}

\begin{proof}
The Cartesian powers remain rationally retractive, and
\eqref{eq:module-additivity} gives $\msecat((f^{\times n})_\Q)=nq$.
Theorem~\ref{thm:main} gives the displayed bounds, limit and integer
part. If $c_1=q$, subadditivity gives $nq\le c_n\le nc_1=nq$.
Conversely, $c_n=nc_1$ for $n\ge2$ implies $0\le n(c_1-q)\le1$,
so the integer $c_1-q$ is zero. The last assertion follows by writing
$c_n=nq+\epsilon_n$, with $\epsilon_n\in\{0,1\}$.
\end{proof}

\begin{corollary}[Strict subadditivity for retractive maps]\label{cor:strict-retractive-products}
Let $p$ be the rationally retractive map between simply connected
spaces of finite rational type constructed in
\cite[Theorem~1.1]{Parent2026}. Its rationalization satisfies
$\msecat(p_\Q)=2$ and $\secat(p_\Q)=3$. Therefore
$2n\le\secat((p^{\times n})_\Q)\le2n+1<n\secat(p_\Q)$ for every
$n\ge2$; in particular $\secat((p\times p)_\Q)\in\{4,5\}$.
\end{corollary}

Strict subadditivity of rational sectional category was observed in
\cite[Remark~11]{CV2016}; the corollary gives it for all Cartesian powers
of a retractive example. For the diagonal of $X$, $f^{\times n}$ is the
diagonal of $X^n$, and the corollary reads as follows.

\begin{corollary}[Stable topological complexity]\label{cor:TC-power-consequences}
Let $X$ be simply connected of finite rational type with
$q=\MTC(X)<\infty$. Then $nq\le\TC_0(X^n)\le nq+1$ for every $n\ge1$,
\[
 \lim_{n\to\infty}\frac{\TC_0(X^n)}{n}=\MTC(X),\qquad
 \MTC(X)=\Bigl\lfloor\frac{\TC_0(X^n)}{n}\Bigr\rfloor\quad(n\ge2),
\]
and $\TC_0(X)=\MTC(X)$ if and only if $\TC_0(X\times X)=2\TC_0(X)$.
\end{corollary}

\subsection{Two degree criteria}\label{subsec:degree-criteria}

A bound on the source $X$ can force the trace of a nullhomotopy to
vanish. This gives a degree criterion for equality of all three
invariants; a separate target-degree criterion is recalled afterwards.

\begin{corollary}[A source-degree criterion]\label{cor:degree-criterion}
Let $f:X\to Y$ be rationally retractive, with split minimal model
\eqref{eq:retractive-minimal-model}. Suppose $W^{<c}=0$ for some $c\ge2$, that $H^{>d}(X;\Q)=0$, and that
$q=\msecat(f_\Q)<\infty$. If $c(q+1)>d+1$, then
\[
 \msecat(f_\Q)=\secat(f_\Q)=\relcat(f_\Q).
\]
In particular, if $H^{>d}(X;\Q)=0$ and $d\le2\MTC(X)<\infty$, then
$\MTC(X)=\TC_0(X)=\TC_0^M(X)$.
\end{corollary}

\begin{proof}
If $q=0$ the three invariants vanish by Lemma~\ref{lem:zero-threshold}.
Let $q\ge1$ and $m=q+1$. Choose a complement $T_d$ of the cocycles in
$A^d$ and put $\overline A=A/(T_d\oplus A^{>d})$. The subspace
$T_d\oplus A^{>d}$ is a differential ideal: $dT_d\subset A^{d+1}$, and
multiplication by a positive-degree element raises degree above $d$.
It is acyclic because $T_d$ contains no nonzero cocycle, every boundary
in degree $d+1$ has a primitive in $T_d$, and $H^{>d}(A)=0$.
Thus $A\to\overline A$ is a surjective quasi-isomorphism of CDGAs onto
an algebra concentrated in degrees at most $d$. By
Proposition~\ref{prop:controlled-resolution}, the cells of $P_m$ have
degree at least $cm>d+1$, so $\Hom_R(P_m,\overline A)$ vanishes in degree
$-1$. As $P_m$ is semifree, $\Hom_R(P_m,A)\to\Hom_R(P_m,\overline A)$ is a
quasi-isomorphism, and $\Hom_{\Dcat(R)}(I^m,A[-1])=0$. By
\eqref{eq:msecat-derived} and \eqref{eq:trace-sequence}, $o_m=0$, and
Theorem~\ref{thm:relative-category} gives $\relcat(f_\Q)\le q$. For the diagonal, $c=2$ and the
condition reads $2(q+1)>d+1$, that is, $d\le 2q$.
\end{proof}

\begin{remark}[A target-degree criterion]\label{cor:target-cohomology-equality}
The bounded-target argument of \cite[Lemma~A.1]{Parent2026} gives
another criterion. For a rationally retractive $f\colon X\to Y$ with
split minimal model \eqref{eq:retractive-minimal-model}, suppose
$W^{<c}=0$ for $c\ge2$, $H^{>D}(Y;\Q)=0$ and
$1\le q=\msecat(f_\Q)<\infty$. If $2(c(q+1)-1)>D$, then
$\secat(f_\Q)=\msecat(f_\Q)$. Indeed, the quotient model can be built
with new generators of degree at least $b=c(q+1)-1$. In a truncation
of $R$ above degree $D$, all product defects of a module retraction
vanish because their degrees are at least $2b>D$. Sullivan lifting
then gives an algebra retraction over $R$.
\end{remark}

\subsection{What remains open}\label{subsec:open}

The general bound leaves only one possible unit of difference, but the
vanishing of that difference is still a substantive question. We end
with three problems suggested by the two mechanisms of
Section~\ref{sec:ideal}.

\begin{enumerate}[label=\textup{(\arabic*)},itemsep=3pt plus 3pt]
\item Is $\MTC(X)=\TC_0(X)$ for every simply connected space $X$ of
 finite rational type? A stronger question asks whether
 $\MTC(X)=\TC_0^M(X)$ always holds. When $q=\MTC(X)<\infty$, the stronger
 equality is equivalent, by Theorem~\ref{thm:relative-category}, to
 $o_{q+1}=[I^{q+1}\hookrightarrow I]=0$ for the diagonal. Thus it asks
 whether an $R$-valued nullhomotopy at the module threshold can always be
 corrected to have zero trace. This would imply the ordinary equality;
 the converse does not follow from the criteria proved here. The strict
 retractive example of \cite{Parent2026} does not settle either diagonal
 question.
\item Which hypotheses on a rationally retractive map, beyond a
 compatible finite Poincar\'e duality model, ensure both natural
 biduality for $I$ and a factorization of its evaluation through $I$?
 These are precisely the two hypotheses used by Theorem~\ref{thm:duality}.
\item Is $\secat((p\times p)_\Q)$ equal to $4$ or to $5$ for the map $p$ of
 Corollary~\ref{cor:strict-retractive-products}? Equivalently, is
 $2\secat(p_\Q)-\secat((p\times p)_\Q)$ equal to two or to one?
\end{enumerate}

% BEGIN PROOF ROADMAP
\Needspace*{30\baselineskip}
\subsection*{The construction behind the results}\label{subsec:proof-roadmap}

The results and their consequences are now in place. The diagram shows
which constructions remain to be supplied in
Sections~\ref{sec:toolbox}--\ref{sec:filtered-model}, and how they enter
the proofs above. Arrows indicate proof dependencies.

\begin{center}
\begin{tikzpicture}[x=1pt,y=1pt,>=stealth,
 roadmap/.style={draw=black!45,rounded corners=2pt,fill=black!2,
   align=center,text width=128pt,inner sep=5pt,font=\small},
 outcome/.style={roadmap,text width=204pt,draw=black!70}]
\node[roadmap,text width=210pt] (res) at (0,0)
 {Controlled resolutions and comparisons\\
  Propositions~\ref{prop:controlled-resolution}--\ref{prop:power-comparison};
  Corollary~\ref{cor:binary-comparison}};
\node[roadmap,text width=210pt] (conv) at (0,-62)
 {Hess extension and multiplicative model\\
  Sections~\ref{sec:toolbox}--\ref{sec:filtered-model};
  Theorem~\ref{thm:ideal-nullhomotopy}};
\node[roadmap] (shift) at (-152,-139)
 {Exponent shift\\Theorem~\ref{thm:shift}\\[2pt]
  $\iota_*(o_m)=0\Rightarrow o_{m+1}=0$};
\node[roadmap] (criterion) at (0,-139)
 {Relative-category criterion\\Theorem~\ref{thm:relative-category}\\[2pt]
  $o_{n+1}=0\iff\relcat(f_\Q)\le n$};
\node[roadmap] (duality) at (152,-139)
 {Duality\\Theorem~\ref{thm:duality}\\
  under its hypotheses\\[2pt]
  $\iota_*(o_m)=0\Rightarrow o_m=0$};
\node[outcome] (bound) at (-115,-229)
 {General comparison: Theorem~\ref{thm:main}\\[2pt]
  $\relcat(f_\Q)\le\msecat(f_\Q)+1$};
\node[outcome] (equality) at (115,-229)
 {Poincar\'e equality: Theorem~\ref{thm:poincare-main}\\[2pt]
  $\MTC(X)=\TC_0(X)=\TC_0^M(X)$};
\draw[->] (res) -- (conv);
\draw[->] (conv) -- (criterion);
\draw[->] (res.west) -| (shift.north);
\draw[->] (res.east) -| (duality.north);
\draw[->] (shift.south) -- ([xshift=-40pt]bound.north);
\draw[->] (criterion.south west) -- ([xshift=40pt]bound.north);
\draw[->] (criterion.south east) -- ([xshift=-40pt]equality.north);
\draw[->] (duality.south) -- ([xshift=40pt]equality.north);
\end{tikzpicture}
\end{center}

The general comparison also uses the quotient criterion
(Theorem~\ref{thm:quotient-criteria}) and the zero case
(Lemma~\ref{lem:zero-threshold}); the duality conclusion uses the
compatible Poincar\'e models of Proposition~\ref{prop:pd-compatible-map}.
% END PROOF ROADMAP

\section{The model required by the extension argument}
\label{sec:toolbox}

An ideal-valued nullhomotopy retracts a module cone. To obtain the
algebra retraction required by relative category, we must extend that
cone by generators and products while keeping their new values in the
ideal. Three questions organize the construction. Which algebra model
allows the retraction to extend? How can a resolution retain the
coefficient information needed to keep values in the ideal? How can
the cone be enlarged without losing that bound? This section answers
the first question by isolating a form of differential for which
Hess's extension argument works. Section~\ref{sec:controlled-resolutions}
builds resolutions and comparison homotopies with controlled
coefficients. Section~\ref{sec:filtered-model} enlarges the cone to a
Sullivan model of the prescribed form and proves
Theorem~\ref{thm:ideal-nullhomotopy}.
\label{subsec:overview}

\subsection{Extending a module retraction multiplicatively}
\label{sec:relative-conversion}

Suppose the algebra model has already been built. To extend a map
across a generator $z$, one must find a primitive of the previously
defined value of $Dz$. Acyclicity will provide that primitive only if we
can keep both elements inside the same acyclic complex. We index the layers of new generators by a nonnegative integer $p$,
called their \emph{level}. The family of acyclic complexes $N_p$ below
will keep the extension at each level inside the ideal.

The construction adapts Hess's square-zero extension and
weighted-kernel argument \cite{Hess1991}; see also
\cite[Theorem~29.9 and Section~29\textup{(f)}]{FHT2001}. We use a differential ideal $\mathfrak a$ in the
next two lemmas to separate this general argument from its eventual
application to $\mathfrak a=I$.

\begin{lemma}\label{lem:square-zero-product}
Let $R\longrightarrow Q\xrightarrow{\rho_0}R$ be a retraction of differential
$R$-modules, where the first map is $R$-linear and identifies $R$ with a subcomplex
of $Q$. The formula
\begin{equation}\label{eq:square-zero-product}
 u*v=\rho_0(u)v+u \rho_0(v)-\rho_0(u)\rho_0(v)
\end{equation}
defines a CDGA structure on $Q$ extending that of $R$. For this
structure, $\rho_0$ is a CDGA retraction and $\ker \rho_0$ is a square-zero ideal.
Here the right action is $ua=(-1)^{|u||a|}au$ for homogeneous
$u\in Q$ and $a\in R$.
\end{lemma}

\begin{proof}
Write $u=\rho_0(u)+u_0$ and $v=\rho_0(v)+v_0$, with $u_0,v_0\in\ker \rho_0$.
The formula retains the product in $R$ and its two module actions, and
sets $u_0*v_0=0$. Thus it is precisely the differential square-zero
extension $R\ltimes\ker \rho_0$, which proves all the assertions.
\end{proof}

The product has made $\rho_0$ multiplicative, but it need not make the
projection $Q\to R/\mathfrak a^m$ multiplicative. The next lemma compensates for
this discrepancy by changing its kernel according to the level. The
individual kernels need not be subalgebras; their products move to the
level specified in \eqref{eq:weighted-product}.

The differential in the eventual algebra model will have four kinds of
terms: quadratic products, same-level linear terms, lower-level linear
terms with coefficients in the ideal, and scalar terms. The next lemma
controls products and the last two kinds by explicit inclusions;
$R$-linearity controls the same-level terms.

\begin{lemma}\label{lem:weighted-acyclic-subcomplexes}
Let $R$ be a CDGA over $\Q$, let $Q$ carry the product of
Lemma~\ref{lem:square-zero-product}, and
write $j_Q\colon R\to Q$ for the inclusion. Let $\mathfrak a$ be a differential
ideal of $R$, let $m\geq1$, and write $\pi\colon R\to R/\mathfrak a^m$.
Suppose that $\zeta_Q\colon Q\to R/\mathfrak a^m$ is a surjective
quasi-isomorphism of $R$-modules restricting to $\pi$ on $R$, and that
$\rho_0(N)\subset \mathfrak a$, where $N=\ker\zeta_Q$. For $p\geq0$, set
\begin{equation}\label{eq:weighted-acyclic-subcomplexes}
 \Psi_p=\operatorname{id}_Q+p\,j_Q\rho_0,
 \qquad
 N_p=\Psi_p^{-1}(N)=\ker(\zeta_Q+p\,\pi \rho_0).
\end{equation}
Then $N_p$ is an acyclic $R$-subcomplex of $Q$, $\rho_0(N_p)\subset \mathfrak a$, and
\begin{align}
 N_i*N_j&\subset N_{i+j+1},\label{eq:weighted-product}\\
 \mathfrak a^{m-1}N_q&\subset N_p\quad(p,q\geq0),\label{eq:weighted-linear}\\
 j_Q(\mathfrak a^m)&\subset N_p\quad(p\geq0).\label{eq:weighted-scalar}
\end{align}
\end{lemma}

\begin{proof}
\emph{Acyclicity and values in the ideal.}
The correction must preserve acyclicity. Since $j_Q\rho_0$ is an idempotent
$R$-linear cochain map, it gives an invertible change of coordinates:
\[
 \Psi_p^{-1}=\operatorname{id}_Q-\frac{p}{p+1}j_Q\rho_0.
\]
The identity $\zeta_Qj_Q=\pi$ gives
$\zeta_Q\Psi_p=\zeta_Q+p\pi\rho_0$, proving the two descriptions of
$N_p$ in \eqref{eq:weighted-acyclic-subcomplexes}.
The restriction $\Psi_p:N_p\to N$ is an isomorphism of $R$-complexes,
so $N_p$ is acyclic. For $u\in N_p$,
\[
 (p+1)\rho_0(u)=\rho_0(\Psi_pu)\in \mathfrak a,
\]
which also gives $\rho_0(N_p)\subset \mathfrak a$.

\emph{Compatibility with the four differential terms.}
For $u\in N_i$ and $v\in N_j$, $R$-linearity of $\zeta_Q$ and
\eqref{eq:square-zero-product} give
\[
 \zeta_Q(u*v)=-(i+j+1)\pi\bigl(\rho_0(u)\rho_0(v)\bigr),
 \qquad \rho_0(u*v)=\rho_0(u)\rho_0(v).
\]
This proves \eqref{eq:weighted-product}. For $a\in \mathfrak a^{m-1}$ and
$u\in N_q$, we have
\begin{equation}\label{eq:hess-only-extra-term}
 (\zeta_Q+p\,\pi \rho_0)(au)
       =(p-q)\pi\bigl(a\,\rho_0(u)\bigr)=0,
\end{equation}
because $\rho_0(u)\in \mathfrak a$. This proves \eqref{eq:weighted-linear}.
For $a\in \mathfrak a^m$, the same expression on $j_Q(a)$ is
$(p+1)\pi(a)=0$, giving \eqref{eq:weighted-scalar}. When $m=1$,
$\pi\rho_0$ vanishes on $N$ and all the complexes $N_p$ equal $N$;
the formulas still apply.
\end{proof}

\Needspace*{4\baselineskip}
\begin{proposition}\label{prop:relative-hess-conversion}
Let $R$ be a CDGA over $\mathbb Q$, let $\mathfrak a$ be a differential
ideal, and let $m\geq1$. Write $\pi\colon R\to R/\mathfrak a^m$ for the quotient.
Suppose that
\[
 R\longrightarrow E=(R\otimes\Lambda Z,D)
       \xrightarrow{\zeta\,\simeq}R/\mathfrak a^m
\]
is a relative Sullivan model with $Z=\bigoplus_{p\geq0}Z_p$.
Choose a well-ordered homogeneous basis of each $Z_p$. Suppose that
these orders and the differential satisfy the following properties:
\begin{enumerate}
\item $Q=R\oplus(R\otimes Z_0)$ is a subcomplex, $\zeta(Z_0)=0$, and
      $\zeta|_Q$ is a surjective quasi-isomorphism;
\item for $p\geq1$,
\begin{equation}\label{eq:general-hess-shape}
 DZ_p\subset
 \sum_{i+j=p-1}R\otimes Z_iZ_j
 +R\otimes Z_p^{\mathrm{earlier}}
 +\mathfrak a^{m-1}\otimes Z_{<p}+\mathfrak a^m,
\end{equation}
where ``earlier'' refers to the chosen order within $Z_p$;
\item there is an $R$-linear cochain retraction $\rho_0\colon Q\to R$ such that
\[
 \rho_0\bigl(\ker(\zeta|_Q)\bigr)\subset \mathfrak a.
\]
\end{enumerate}
Then the inclusion $R\longrightarrow E$ admits a CDGA retraction
$\rho$ with $\rho(Z)\subset \mathfrak a$.
\end{proposition}

\begin{proof}
We construct the map into $Q$ before composing with $\rho_0$. This makes
the problem of extending a retraction a problem of choosing primitives.
Equip $Q$ with the product of Lemma~\ref{lem:square-zero-product} and
write $\mathcal Q$ for the resulting CDGA. The restriction $\zeta_Q=\zeta|_Q$
satisfies the hypotheses of Lemma~\ref{lem:weighted-acyclic-subcomplexes}.
Let $N_p$ be the resulting acyclic $R$-subcomplexes. We construct a
CDGA map $\Phi\colon E\to\mathcal Q$, identical on $R$, such that
\[
 \Phi(Z_p)\subset N_p.
\]
\emph{Initial layer.} On $Z_0$ use its given inclusion into $Q$.
The equation $\zeta(Z_0)=0$ places this image in $N_0$; the subcomplex
condition ensures that this choice already respects the differential.

\emph{Extension across one generator.}
Order the remaining generators first by level $p$, then by their
chosen order within $Z_p$. This is well founded, and every generator
in \eqref{eq:general-hess-shape} is earlier than the one being defined.
Thus $\Phi(Dz)$ is already defined when $z\in Z_p$ is reached.
For a quadratic term $az_iz_j$, multiplicativity gives
$\Phi(az_iz_j)=a(\Phi(z_i)*\Phi(z_j))\in N_p$ by
\eqref{eq:weighted-product} and the $R$-linearity of $N_p$. The same-level linear terms also land in $N_p$, since it is an
$R$-submodule. Equations
\eqref{eq:weighted-linear} and \eqref{eq:weighted-scalar} handle the
remaining terms. Since $\Phi$ is already a cochain map on the earlier
generators and $D^2z=0$, the element $\Phi(Dz)$ is a cycle in $N_p$.
Its acyclicity supplies $\Phi(z)\in N_p$ with
$d\Phi(z)=\Phi(Dz)$.

\emph{Conclusion.} The induction defines $\Phi$ on every generator.
Since $\rho_0$ is multiplicative for the chosen product on $Q$, the composite
$\rho=\rho_0\Phi$ is a CDGA retraction. Moreover,
$\rho(Z_p)\subset \rho_0(N_p)\subset \mathfrak a$.
\end{proof}

The extension is now reduced to a concrete requirement: a model of
$R/I^m$ whose differential has shape \eqref{eq:general-hess-shape},
with $\mathfrak a=I$. An ideal-valued nullhomotopy will retract its initial cone
layer. The next section constructs the resolution underlying that cone;
Section~\ref{sec:filtered-model} supplies the multiplicative model.

\section{Controlled resolutions and comparison homotopies}
\label{sec:controlled-resolutions}

An arbitrary resolution makes ideal inclusions comparable, but does not
force the comparison coefficients into $I$. We therefore seek a
resolution in which the weight of each free generator is prescribed.
The surplus weight in a comparison must then occur in its coefficients.
We first construct this resolution for the length-preserving
differential and then transfer the remaining terms.

The construction runs from the linear resolution
(Lemma~\ref{lem:natural-linear-resolution}) to the controlled resolution
(Proposition~\ref{prop:controlled-resolution}), the Hom contraction
(Proposition~\ref{prop:controlled-hom}), and the binary comparison
(Corollary~\ref{cor:binary-comparison}). Section~\ref{sec:filtered-model}
then passes through the curved comparison
(Lemma~\ref{lem:curved-comparison}) and the Sullivan model
(Proposition~\ref{prop:filtered-sullivan-model}) to the conversion theorem
(Theorem~\ref{thm:ideal-nullhomotopy}).

Fix the split minimal Sullivan model of Lemma~\ref{lem:split-minimal-model}:
\[
 A=(\Lambda V,d),\qquad
 (R,d_R)=(A\otimes\Lambda W,d_R),\qquad
 \varphi\colon R\longrightarrow A,\qquad I=\ker\varphi=(W).
\]
Both $A$ and $R$ are simply connected and of finite type; the
inclusion of $A$ and the projection $\varphi$ are CDGA morphisms.
In particular, $W^{\leq1}=0$ and $d_R(I)\subset I$.
For each $m\geq1$ we construct a semifree resolution
$\varepsilon_m:P_m\to I^m$. We first resolve the part of the
differential preserving relative length and then restore the rest.

The \emph{conormal module} $C=I/I^2$ carries the differential induced
by $d_R$. Its derivation extension is $D_{\mathrm{lin}}$ on
$\Lambda_A C=\operatorname{gr}_I R$. The construction uses the gradings
in Table~\ref{tab:construction-gradings}; the last three rows also fix
the conventions for the multiplicative model in
Section~\ref{sec:filtered-model}.

\begin{table}[htbp]
\centering\small
\begin{tabular}{@{}p{.23\textwidth}p{.70\textwidth}@{}}
\toprule
Grading or filtration & Meaning\\
\midrule
Relative length & Number of factors from $W$; also called diagonal length for a diagonal model.\\
Generator weight & Relative length in a linear-resolution generator; $G_{m,j}$ has weight $m+j$.\\
Total weight & Sum of the generator and coefficient weights, or of the weights in all bar entries.\\
Coefficient order $k$ & Membership of a coefficient in $I^k$; this is a filtration bound.\\
Resolution index $j$ & Homological index; total cohomological degree is internal degree minus $j$.\\
Index difference $\nu$ & $j_{\mathrm{target}}-j_{\mathrm{source}}$ for a matrix term of a comparison.\\
Harrison arity $n$ & Number of cogenerators in a Lie word.\\
Level $p$ & Index of the cobar generating layer $Z_p$, with $p=n-1$.\\
Leaf arity & $p+1$ on $Z_p$, extended additively to products of cobar generators.\\
\bottomrule
\end{tabular}
\par\vspace{8pt}
\caption{Gradings and filtrations used in the construction.}
\label{tab:construction-gradings}
\end{table}

\subsection{The transfer formula}
The perturbation lemma restores a differential after passage to a
smaller resolution. We use it along a
normalized deformation retract; see \cite[Main Perturbation Lemma~2.4]{Crainic2004},
with the opposite convention for the contracting homotopy. Suppose
$i:P\to B$ and $p:B\to P$ are cochain maps and $\mathsf h:B\to B$ has degree
$-1$, with
\[
 pi=1,\qquad d_B\mathsf h+\mathsf h d_B=1-ip,\qquad \mathsf h i=p\mathsf h=\mathsf h^2=0.
\]
Here we first regard the complexes as rational complexes. Let $\eta$
have degree one, with $(d_B+\eta)^2=0$, and suppose the series below
are finite on each element. In the module application, $\eta$ includes
a derivation on coefficients; the transferred Leibniz rule will be
checked separately. Put
\[
 \Omega=(1+\eta \mathsf h)^{-1}\eta
   =\sum_{a\geq0}(-1)^a\eta(\mathsf h\eta)^a.
\]
The transferred differential and the new retract are
\begin{equation}\label{eq:perturbation-formulas}
 \begin{gathered}
 d'_P=d_P+p\Omega i,\qquad i'=i-\mathsf h\Omega i,\qquad p'=p-p\Omega\mathsf h,\\
 \mathsf h'=\mathsf h-\mathsf h\Omega\mathsf h.
 \end{gathered}
\end{equation}
The relation $\Omega=\eta-\eta\mathsf h\Omega=\eta-\Omega\mathsf h\eta$ verifies the formulas directly:
substitution, followed by the retract identities, gives $(d'_P)^2=0$,
the cochain identities for $i',p'$, and
\[
 p'i'=1,\qquad (d_B+\eta)\mathsf h'+\mathsf h'(d_B+\eta)=1-i'p'.
\]
Each application below proves that the sums terminate on every input,
so these formulas define maps of ordinary complexes. The analogous
formula for a contraction to zero will be used in
\eqref{eq:zero-contraction-perturbation}.

\subsection{Separating the linear differential}

The first approximation keeps the part of $d_R$ that preserves relative
length, including its coefficients in the base algebra $A$.
Recall $d_R=D_{\mathrm{lin}}+d_+$ from
\eqref{eq:relative-differential-splitting}. The associated graded
algebra explains both why $D_{\mathrm{lin}}$ squares to zero and why its
coefficients belong to the base $A$.

\begin{lemma}\label{lem:conormal-differential}
The operator $D_{\mathrm{lin}}$ is a differential.  With
$C=I/I^2$, the differential graded algebra
$(R,D_{\mathrm{lin}})$ is naturally
$\Lambda_{A}C$, where $C$ carries its induced differential.
In the $A$-basis $W$ of $C=A\otimes W$, the matrix of its conormal
differential has entries in $A^{\geq2}$. This differential satisfies
the module Leibniz rule over $d$; its matrix describes its values on
that basis.
\end{lemma}

\begin{proof}
The differential $d_R$ preserves powers of $I$.  Under the
identification
$\operatorname{gr}_I R=\Lambda_{A}(I/I^2)$, its associated
graded differential is exactly $D_{\mathrm{lin}}$.
Consequently $D_{\mathrm{lin}}^2=0$. To check its matrix
coefficients, write $D_{\mathrm{lin}}w=\sum_{w'}a_{w,w'}w'$.
Minimality of $R=\Lambda(V\oplus W)$ excludes a scalar
coefficient: every monomial in $d_Rw$ is decomposable.
A term linear in $W$ must therefore contain a positive-degree
factor from $A$. Since $A^1=0$, each nonzero $a_{w,w'}$ has
degree at least two.
\end{proof}

\begin{remark}\label{ex:nonconstant-conormal}
For $A=\Q[a_2]$ and $R=A\otimes\Lambda(u_3,w_4)$ with
$du=0$ and $dw=au$, the conormal differential is $D_{\mathrm{lin}}w=au$.
The universal splittings retain this $A$-linear differential.
\end{remark}

The length-preserving part can therefore be retained throughout the
linear construction. Only $d_+$ needs perturbation: each of its terms
adds at least one relative factor.

\subsection{The linear resolution and its grading}

Temporarily forget the internal differential and put
$R_{\mathrm{lin}}=\Lambda_A C$, with augmentation ideal
$\mathfrak n=(C)$. We seek generators of a single weight at each
resolution index: this will turn comparisons of total weight into
bounds on coefficients. Here coefficient order is measured by the
powers of $\mathfrak n$, which correspond to the powers of $I$ after
the full differential is restored.

\Needspace*{4\baselineskip}
\begin{lemma}\label{lem:natural-linear-resolution}
For every $m\geq1$, the ideal $\mathfrak n^m$ has a free linear
resolution
\begin{equation}\label{eq:linear-resolution}
 P^{\mathrm{lin}}_{m,j}
       =R_{\mathrm{lin}}\otimes_{A}G_{m,j}(C),
 \qquad \operatorname{weight}G_{m,j}(C)=m+j.
\end{equation}
Its differential $\delta_m$ lowers $j$ by one and has coefficients
of weight one. The resolution and its bar inclusion are canonical.
After choosing equivariant complements, the
retraction and homotopy are natural in $C$.
The kernel of the augmentation has an $A$-linear contraction
$H_{m,0}$ which preserves total weight, raises $j$ by one, and
satisfies $H_{m,0}^2=0$.

If $C$ carries a differential over the differential graded algebra
$A$, these maps are compatible with it, with the totalization signs
specified in \eqref{eq:appendix-totalization}.
\end{lemma}

\begin{proof}
\emph{1. Why the generators have weight $m+j$.}
First work over $\mathbb Q$. Tensoring the usual resolutions for
polynomial and exterior generators shows that a free
graded-commutative algebra is Koszul. For infinitely many generators,
take the filtered union of the finite constructions; every element
belongs to one of them.

To pass from the algebra to its ideal powers, suppose $\mathfrak n^a$
already has a linear resolution and use
\[
 0\longrightarrow\mathfrak n^{a+1}
 \longrightarrow\mathfrak n^a
 \longrightarrow (R_{\mathrm{lin}})_a\otimes\mathbb Q\longrightarrow0,
\]
where $(R_{\mathrm{lin}})_a$ is the weight-$a$ component and $R_{\mathrm{lin}}$ acts on the last
term through its augmentation. The exact segment
\[
 \operatorname{Tor}_{j+1}^{R_{\mathrm{lin}}}
       ((R_{\mathrm{lin}})_a\otimes\mathbb Q,\mathbb Q)
 \longrightarrow
 \operatorname{Tor}_{j}^{R_{\mathrm{lin}}}
       (\mathfrak n^{a+1},\mathbb Q)
 \longrightarrow
 \operatorname{Tor}_{j}^{R_{\mathrm{lin}}}
       (\mathfrak n^a,\mathbb Q)
\]
has outer terms in weights $a+j+1$ and $a+j$, respectively. It allows $\operatorname{Tor}_j^{R_{\mathrm{lin}}}(\mathfrak n^{a+1},\mathbb Q)$
only in weights $a+j$ and $a+j+1$.  Since the algebra is generated in weight one and its weights are
bounded below, a minimal resolution of a module generated in weights
at least $a+1$ has its generators
at index $j$ in weights at least $a+1+j$.  The first possibility
is therefore excluded.  The initial case follows from
$0\to\mathfrak n\to R_{\mathrm{lin}}\to\mathbb Q\to0$ and the Koszul property.

For the underlying graded algebra over $A$, this concentration passes
to $A\otimes\Lambda W$ by flat base change from $\Q$. Writing
$\mathfrak n_W=\Lambda^{\geq1}W$, the resulting identification is
\[
 \operatorname{Tor}^{A\otimes\Lambda W}_j
       (A\otimes\mathfrak n_W^m,A)
 \cong A\otimes
       \operatorname{Tor}^{\Lambda W}_j(\mathfrak n_W^m,\Q).
\]
The next step chooses the maps so that the internal differential over
$A$ can also be restored.

\emph{2. Why the maps respect the internal differential.}
Existence of a linear resolution does not yet provide compatible
splittings. Over $A$ we evaluate a universal complex; semisimplicity of
the rational symmetric-group representations makes this evaluation exact.
Appendix~\ref{app:universal-linear-resolution} represents the
bar generators by the ordered-block-partition modules
\eqref{eq:appendix-representation-module}, splits their complexes over
$\mathbb Q[\mathfrak S_N]$, and only then evaluates on $C$. This gives
the maps $i,p,\mathsf h$ of \eqref{eq:appendix-bar-retract}, with
\begin{equation}\label{eq:normalized-retract}
             pi=1,\qquad \mathsf h i=p\mathsf h=\mathsf h^2=0.
\end{equation}
The inclusion consists of the top cycles, and the first bar face
restricts to $\delta_m$. All representation maps preserve total weight;
in particular, $i,p,\mathsf h$ send each weight-homogeneous component
to the component of the same weight. The bar comparison $u$ used in
\S\ref{sec:power-comparison}, induced by inclusion of the last ideal
factor, has this same homogeneity. Since the part of positive weight
in $R=A\otimes\Lambda W$ is exactly $I$, a coefficient of weight at
least one belongs to $I$.
The homotopies raise $j$ by one. A separate splitting of the augmented
linear complex gives $H_{m,0}$. This last map is only $A$-linear; it is
not the $R_{\mathrm{lin}}$-linear homotopy $\mathsf h$ on the bar complex. Evaluation and the
sign convention \eqref{eq:appendix-totalization} make the degree-zero
maps commute with the internal differential and the homotopies
anticommute with it. Thus the construction applies to the differential
$A$-module $C$, including Remark~\ref{ex:nonconstant-conormal}.
\end{proof}

\begin{remark}
Over $\mathbb Q$, linearity of augmentation-ideal powers for a
free graded-commutative algebra on finitely many generators is a
special case of the general Koszul-algebra construction in
\cite[Theorem~5.3 and Corollary~5.5]{FaberEtAl2021}. A related
use of homological perturbation to obtain resolutions of power
quotients in the polynomial case appears in
\cite[Section~3.6]{MillerRahmati2024}. The preceding proof also covers infinitely many generators.
Appendix~\ref{app:universal-linear-resolution} constructs the natural
relative maps and verifies their compatibility with the internal
differential.

\end{remark}

Each $G_{m,j}(C)$ in
\eqref{eq:linear-resolution} includes the totalization shift
by $-j$.  Its underlying graded $A$-module is obtained by
extension of scalars from a graded vector space, which we also
denote by $G_{m,j}$ when no confusion is possible.  Since
$W^{\leq1}=0$, its free generators satisfy
\begin{equation}\label{eq:resolution-degree-bound}
                   |g|\geq2m+j
                   \qquad(g\in G_{m,j}).
\end{equation}
More generally, if $W^{<c}=0$ for $c\geq2$, the same weight
calculation gives $|g|\geq cm+(c-1)j$.
In particular, only finitely many indices occur in a fixed
cohomological degree.  Each $G_{m,j}$ is also finite dimensional
in every degree: it is a summand of a finite sum of tensor
powers of $W$, shifted by $-j$, and these tensor powers
have finite-dimensional homogeneous components because $W=W^{\ge2}$
and each $W^k$ is finite dimensional.

Restoring $d_+$ preserves the coefficient bound because a perturbation term
adds more total weight than its homotopies can put into the target
generator. The difference leaves at least one factor in the
coefficient. The next proposition records this estimate together with
the augmented comparison maps that will be used for ideal inclusions.

\begin{proposition}\label{prop:controlled-resolution}
For every $m\geq1$ there is a surjective quasi-isomorphism
\[
 \varepsilon_m:(P_m,D_m)\xrightarrow{\simeq}(I^m,d_R),
 \qquad P_m=R\otimes G_m,
 \qquad G_m=\bigoplus_{j\geq0}G_{m,j},
\]
where $P_m$ is semifree, with the degree bound
\eqref{eq:resolution-degree-bound}. More generally, if $W^{<c}=0$ for
$c\geq2$, its generators $g\in G_{m,j}$ satisfy
$|g|\geq cm+(c-1)j$; in particular, all its cells have degree at least
$cm$. Its differential has a decomposition
\begin{equation}\label{eq:controlled-differential}
             D_m=\delta_m+D_{\mathrm{lin},m}+\Delta_m.
\end{equation}
Here $\delta_m$ lowers $j$ by one and has coefficients of weight one,
in $A\otimes W$, while $D_{\mathrm{lin},m}$ preserves $j$ and is induced by the
conormal differential. The perturbation $\Delta_m$ does not decrease
$j$ and satisfies, for homogeneous $a\in R$ and $p\in P_m$,
\begin{equation}\label{eq:perturbation-leibniz}
 \Delta_m(ap)=d_+(a)p+(-1)^{|a|}a\Delta_m(p).
\end{equation}
Its coefficient estimate is
\begin{equation}\label{eq:delta-positive}
             \Delta_m(I^kP_m)\subset I^{k+1}P_m
             \qquad(k\geq0).
\end{equation}
All matrix coefficients of $D_m$ on free generators have
cohomological degree at least two.

Let $\mathcal B_m$ be the normalized bar resolution of
$(I^m,d_R)$ relative to $A$, with bar factors in $I\cong R/A$.
There are $R$-linear cochain maps
\begin{equation}\label{eq:complete-bar-retract}
 i_m^+:P_m\longrightarrow\mathcal B_m,
 \qquad p_m^+:\mathcal B_m\longrightarrow P_m,
 \qquad p_m^+i_m^+=1,
\end{equation}
which strictly preserve the augmentations:
\begin{equation}\label{eq:complete-bar-augmentation}
 \varepsilon_{\mathrm{bar},m}i_m^+=\varepsilon_m,
 \qquad \varepsilon_m p_m^+=\varepsilon_{\mathrm{bar},m}.
\end{equation}
The proof constructs $i_m^+$ and $p_m^+$ by series that terminate on
each input.

\end{proposition}

\begin{proof}
\emph{1. Define the perturbation where it is explicit.}
Keep $D_{\mathrm{lin}}$ in the internal differential of the retract from
Lemma~\ref{lem:natural-linear-resolution}. On the bar resolution, apply
$d_+$ to every tensor factor, including the free coefficient and the
last ideal factor, with the tensor differential signs. Because $d_+$ is
an $A$-linear derivation vanishing on $A$, these operations respect the
multiplication faces. They preserve $j$ and raise total weight by at
least one. The full internal differential is induced by $d_R$ and
squares to zero; it commutes with the untotalized multiplication
faces and anticommutes with their totalized operators. Hence this is a perturbation to which
\eqref{eq:perturbation-formulas} applies. For the normalized retract
$(i,p,\mathsf h)$ it gives
\begin{equation}\label{eq:complete-bar-map-series}
 \begin{aligned}
 i_m^+&=\sum_{a\geq0}(-1)^a(\mathsf h d_+)^a i,\\
 p_m^+&=\sum_{a\geq0}(-1)^a p(d_+\mathsf h)^a.
 \end{aligned}
\end{equation}
\emph{2. Verify the Leibniz rule and then the coefficient estimate.}
Put $\mathsf i_a=(\mathsf h d_+)^a i$ and
$\mathsf p_a=p(d_+\mathsf h)^a$. The side conditions give
$\mathsf h\mathsf i_a=0$ and $\mathsf p_a\mathsf h=0$ for every $a$.
When a coefficient is moved past $d_+$ in
$\mathsf i_{a+1}=\mathsf h d_+\mathsf i_a$, its extra derivation term
is killed by $\mathsf h\mathsf i_a=0$; the analogous term in
$\mathsf p_{a+1}=\mathsf p_a d_+\mathsf h$ is killed by
$\mathsf p_a\mathsf h=0$. The comparison maps are therefore
$R$-linear. The perturbation identities give their cochain and
retract identities.

The transferred perturbation has summands
\begin{equation}\label{eq:transferred-perturbation}
 \Delta_{m,a}=(-1)^a p\,d_+(\mathsf h d_+)^a i,
 \qquad a\ge0.
\end{equation}
Here $p\mathsf i_0=pi=1$ and $p\mathsf i_a=0$ for $a\ge1$,
the latter identity following from $p\mathsf h=0$.
Consequently the term $a=0$ has derivation $d_+$ on coefficients,
whereas all higher summands are $R$-linear. This proves
\eqref{eq:perturbation-leibniz}; in particular, $D_m$ satisfies the
module Leibniz rule over $d_R$.

For a free generator at index $j$, the $a$th summand raises the index
by $a$ and total weight by at least $a+1$. Its target generator has
weight $m+j+a$, leaving at least one factor of $I$ in the coefficient.
Thus $\Delta_m(G_m)\subset IP_m$. Together with
\eqref{eq:perturbation-leibniz} and $d_+(I^k)\subset I^{k+1}$,
this gives \eqref{eq:delta-positive} for every $k\ge0$, not only on
free generators.

\emph{3. Retain the augmentation and resolve the ideal.}
The augmentation sees only index zero. Consequently every correction
in $i_m^+-i$ or $p_m^+-p$ disappears after augmentation. Define
$\varepsilon_m=\varepsilon_{\mathrm{bar},m}i_m^+
=\varepsilon_{\mathrm{bar},m}i$: it multiplies the weight-$m$
generators at index zero and vanishes at higher indices. The
unperturbed augmented identities now give both equations in
\eqref{eq:complete-bar-augmentation}.
The augmented bar contraction uses the split projection
$R\to I$ and insertion of the unit, both compatible with
$d_R$. It therefore contracts the full augmented complex
over $A$, with the totalization signs. Thus
$\varepsilon_{\mathrm{bar},m}$ is a quasi-isomorphism, and
the deformation retract makes $\varepsilon_m$ a
quasi-isomorphism as well.
By its definition above, $\varepsilon_m$ is the same graded map as
in the linear resolution, so its graded kernel $K_m=\ker\varepsilon_m$ is unchanged. The
augmentation is a cochain map for both the linear and full
differentials, so $\varepsilon_m\Delta_m=d_+\varepsilon_m$.
It follows that $K_m$ is also invariant under $\Delta_m$.

\emph{4. Verify that the formulas define a semifree module.}
The expansions terminate on each fixed input. In
\eqref{eq:complete-bar-map-series} and
\eqref{eq:transferred-perturbation}, every extra homotopy raises $j$.
The bound \eqref{eq:resolution-degree-bound} limits that index in each
degree. On the bar side, the $j$ nonempty factors and the last ideal
factor have internal degree at least $2j+2m$; totalization gives the
same bound $2m+j$. The geometric sums therefore act on the ordinary
modules, and the finite perturbation identities used above are valid.

Finally, the values of $\delta_m$ and $\Delta_m$ on free generators
have coefficients in $I$, and the conormal matrix entries lie
in $A^{\geq2}$ by Lemma~\ref{lem:conormal-differential}.
Thus every generator differential involves generators of
strictly smaller cohomological degree.  Ordering free
generators by that degree makes $P_m$ semifree. We have obtained a
resolution with cells in degree at least $2m$, whose nonlinear
corrections carry an additional coefficient in $I$.
\end{proof}

\begin{table}[htbp]
\centering\small
\begin{tabular}{@{}lccc@{}}
\toprule
Operator & Change in $j$ & \shortstack{Change in\\total weight}
         & \shortstack{Coefficient order\\on generators}\\
\midrule
$D_{\mathrm{lin},m}$ & $0$ & $0$ & $0$\\
$\delta_m$ & $-1$ & $0$ & $+1$\\
$\mathsf h$ on the bar & $+1$ & $0$ & ---\\
$d_+$ on the bar & $0$ & $\geq1$ & ---\\
$H_{m,0}$ & $+1$ & $0$ & $-1$\\
$\Delta_{m,a}$ & $a$ & $\geq a+1$ & $\geq1$\\
\bottomrule
\end{tabular}
\par\vspace{8pt}
\caption{Weight estimates for the resolution construction. Total weight
includes the coefficient; linear-resolution generators at index $j$
have weight $m+j$. The coefficient column describes generator values.
The $\Delta_{m,a}$ estimate extends to all $I^kP_m$ using the Leibniz
rule proved in Proposition~\ref{prop:controlled-resolution}.
The $-1$ entry for $H_{m,0}$ is exact on a nonzero homogeneous component
and means a loss of at most one ideal power in general.}\label{tab:resolution-weights}
\end{table}

In the sphere resolution $P_m^{S^2}$ of
Example~\ref{ex:sphere-comparison}, the terms involving index $j-1$
form $\delta_m$, the term $2xe_j$ in $Df_j$ is conormal, and
$\bar x e_j$ is the perturbation: it keeps the index and supplies
one factor of $I$.

\subsection{Comparing consecutive ideal powers}
\label{sec:power-comparison}

The prescribed generator weights force the lifted inclusion of
consecutive ideal powers to have coefficients in $I$. This is the
comparison needed for the exponent shift.

\begin{proposition}\label{prop:power-comparison}
For every $m\geq1$ there is a morphism of differential
graded $R$-modules
\[
                 j_{m,1}:P_{m+1}\longrightarrow P_m
\]
lifting $I^{m+1}\hookrightarrow I^m$ and satisfying
\begin{equation}\label{eq:power-comparison-filtration}
                 j_{m,1}(P_{m+1})\subset IP_m.
\end{equation}
\end{proposition}

\begin{proof}
The bar comparison is canonical; the issue is its coefficient order
after transfer. Let $u:\mathcal B_{m+1}\to\mathcal B_m$ be inclusion
of the last ideal factor. It is an augmented $R$-linear cochain map,
and we transfer it by
\[
                  j_{m,1}=p_m^+u i_{m+1}^+.
\]
This is an $R$-linear cochain map, and
\eqref{eq:complete-bar-augmentation} shows that it strictly
lifts the inclusion of ideals.

The map $u$ itself need not have coefficients in $I$: the last ideal
factor can carry all the surplus weight. The fixed generator weights
of $P_m$ change this. Start with a source generator of index $j$ and
weight $m+1+j$.
A term using $a$ homotopies in $i_{m+1}^+$ and $b$ in $p_m^+$
has target index $j+a+b$ and total weight at least
$m+1+j+a+b$. The map $u$ preserves index and total weight.
Since the target generator has weight $m+j+a+b$, every
coefficient has weight at least one. This proves
\eqref{eq:power-comparison-filtration}. All sums are finite
on each element by Proposition~\ref{prop:controlled-resolution}.
\end{proof}

For $P_m^{S^2}$, the comparison is multiplication by $\bar x$ on each
cell, as recorded in \eqref{eq:sphere-consecutive-comparison}.

\begin{remark}\label{rem:higher-power-comparisons}
Set $j_{m,0}=1$ and, for $t>1$, compose the consecutive
comparisons:
\[
 j_{m,t}=j_{m,1}j_{m+1,1}\cdots j_{m+t-1,1}:
 P_{m+t}\longrightarrow P_m.
\]
It lifts $I^{m+t}\hookrightarrow I^m$ and has image in
$I^tP_m$.  The general comparison theorem uses the consecutive case; the
Poincar\'e duality argument uses $j_{1,m-1}:P_m\to I^{m-1}P_1$.
\end{remark}

To build the multiplicative model, we next need primitives of
comparison defects. The Hom contraction below supplies them with a
controlled coefficient loss.

\subsection{Solving comparison equations without losing control}\label{sec:hom-contraction}

We now need primitives of comparison defects, with a bound on their
coefficients. The target $K_m=\ker\varepsilon_m$ is acyclic, but a
choice of primitive in an acyclic complex need not preserve powers of
$I$. The contraction $H_{m,0}$ offers a quantitative starting point:
it preserves total weight while raising the resolution index, so it
loses one unit of coefficient order.

There is a linearity issue. The map $H_{m,0}$ is $A$-linear, not
$R$-linear, so postcomposition with it need not send an $R$-linear map
to an $R$-linear map. Instead we apply it to the values on a specified
free generating space and then extend $R$-linearly. We must correct
this generatorwise operator for the source differential and for the
nonlinear part of the target differential. We give the operators here;
Lemma~\ref{lem:pointwise-hom-finiteness} in the appendix proves their
pointwise finiteness, and the next proposition establishes their
contraction identity and coefficient bound.

For a free module $\mathcal P$, write $F^k\operatorname{Hom}_R(\mathcal P,P_m)$
for maps whose values on free generators have coefficients in
$I^k$.  Equivalently, their images are contained in $I^kP_m$.
The differential on Hom is
\[
             \partial f=D_mf-(-1)^{|f|}fD_{\mathcal P}.
\]

For $\mathcal P=P_{b_1}\otimes_R\cdots\otimes_RP_{b_n}$ we use the
sum of the factor indices. The tensor differential decomposes as
\[
 D_{\mathcal P}
   =\delta_{\mathcal P}+D_{\mathrm{lin},\mathcal P}
      +\Delta_{\mathcal P}.
\]
Each operator is the sum of its insertions in the factors, with sign
$(-1)^{|p_1|+\cdots+|p_{i-1}|}$ for an insertion in the $i$th factor
of $p_1\otimes\cdots\otimes p_n$.
The operator $\delta_{\mathcal P}$ is graded $R$-linear, lowers the
index by one, and has coefficients of weight one. The operator
$D_{\mathrm{lin},\mathcal P}$ preserves the index and has coefficients
in $A$ on free generators. The Leibniz rule over $d_+$ makes the
signed sum defining $\Delta_{\mathcal P}$ balanced over $R$; this
operator satisfies the same Leibniz rule, does not decrease the
index, and has coefficients in $I$ on free generators. These
observations specify all three source operators used below.

For $\mathcal P=P_{b_1}\otimes_R\cdots\otimes_RP_{b_n}=R\otimes G_{\mathcal P}$,
define $\kappa$ on generator values by
\[
 (\kappa f)(g)=H_{m,0}(f(g))\qquad(g\in G_{\mathcal P}),
\]
followed by graded $R$-linear extension, and put
\[
 \mathcal S(f)=-(-1)^{|f|}f\delta_{\mathcal P},\qquad
 \partial_+(f)=\Delta_mf-(-1)^{|f|}f\Delta_{\mathcal P}.
\]
The two successive corrections are
\begin{align}
 \mathcal H_0&=\sum_{a\geq0}(-1)^a\kappa(\mathcal S\kappa)^a,
                                  \label{eq:koszul-hom-contraction}\\
 \mathcal H_{\mathcal P,m}&=\sum_{a\geq0}(-1)^a
                  \mathcal H_0(\partial_+\mathcal H_0)^a
                                  \label{eq:full-hom-contraction}
\end{align}
Lemma~\ref{lem:pointwise-hom-finiteness} proves pointwise finiteness of
these series and of their geometric inverses. We use that fact here to
establish the contraction identity and its coefficient estimate.

\begin{proposition}\label{prop:controlled-hom}
Let $m,b_1,\ldots,b_n\geq1$ and
$\mathcal P=P_{b_1}\otimes_R\cdots\otimes_RP_{b_n}$, with $n\geq1$.
There is a contraction $\mathcal H_{\mathcal P,m}$ of
$\operatorname{Hom}_R(\mathcal P,K_m)$ satisfying
\begin{equation}\label{eq:hom-filtration}
 \partial\mathcal H_{\mathcal P,m}+\mathcal H_{\mathcal P,m}\partial=1,
 \qquad
 \mathcal H_{\mathcal P,m}(F^k)\subset F^{k-1}\quad(k\geq1).
\end{equation}
Its defining series is finite on each source generator.
For $\mathcal P=P_b^{\otimes_R n}$ the contraction may be chosen
$\mathfrak S_n$-equivariant.
\end{proposition}

\begin{proof}
We now check the contraction identity while keeping track of the
one-power loss. Begin with the target Koszul differential
$\partial_0f=\delta_mf$. On a free
source generator the identity for $\kappa$ is exactly
$\delta_mH_{m,0}+H_{m,0}\delta_m=1$. Extending the values
$R$-linearly gives that identity on the Hom complex, although
$H_{m,0}$ itself is only $A$-linear. We now add the remaining
parts of the differential in three stages.

We use the following perturbation identity for contractions
to zero. Suppose $d\mathsf h+\mathsf h d=1$, set $D=d+\eta$ with $D^2=0$,
and assume $\Xi=1+\eta \mathsf h$ is invertible. Direct expansion gives
$D\Xi=\Xi d$, hence $\Xi^{-1}D=d\Xi^{-1}$. Therefore
\begin{equation}\label{eq:zero-contraction-perturbation}
 \mathsf h'=\mathsf h\Xi^{-1},\qquad
 D\mathsf h'+\mathsf h'D=(D\mathsf h+\mathsf h d)\Xi^{-1}=\Xi\Xi^{-1}=1.
\end{equation}
The inverses used here are the pointwise finite geometric
sums of Lemma~\ref{lem:pointwise-hom-finiteness}.

\emph{1. Source Koszul differential.}
The identities $\delta_{\mathcal P}^2=0$ and $\delta_m^2=0$ give
$\mathcal S^2=0$ and $\partial_0\mathcal S+\mathcal S\partial_0=0$.
Thus $\partial_0+\mathcal S$ is a differential. Adding $\mathcal S$
gives the contraction \eqref{eq:koszul-hom-contraction} by
\eqref{eq:zero-contraction-perturbation}. Each $\mathcal S$ supplies a coefficient in $I$,
whereas each $\kappa$ loses at most one power.  Every term
therefore loses at most one power in total.

\emph{2. Internal differential: where naturality is used.}
On $\Hom_R(\mathcal P,K_m)$, the internal Hom differential is
\[
 \partial_{\mathrm{lin}}f
   =D_{\mathrm{lin},m}f-(-1)^{|f|}fD_{\mathrm{lin},\mathcal P}.
\]
The coefficient derivations cancel, so this is an operator on
$R$-linear maps. Its matrices need not be rational, as
Remark~\ref{ex:nonconstant-conormal} shows. The universal construction
in Lemma~\ref{lem:natural-linear-resolution} makes the target homotopy
anticommute with it over $A$. On the source side, its matrix entries
belong to $A$, so $A$-linearity of $H_{m,0}$ is sufficient. On free
generators these two facts give
\[
       [\partial_{\mathrm{lin}},\kappa]_{\mathrm{gr}}=0.
\]
The source Koszul term has the corresponding graded
commutation relation because
\[
 \delta_{\mathcal P}D_{\mathrm{lin},\mathcal P}
   +D_{\mathrm{lin},\mathcal P}\delta_{\mathcal P}=0.
\]
Indeed, each $\delta_{b_i}+D_{\mathrm{lin},b_i}$ is the differential
of the linear resolution in Lemma~\ref{lem:natural-linear-resolution},
and its tensor differential has the same identity. This uses the
linear resolution itself, before adding $\Delta_{\mathcal P}$.
Thus every term of $\mathcal H_0$
anticommutes with the conormal differential, and
$\mathcal H_0$ remains a contraction when it is added.

\emph{3. Nonlinear differential.} Add $\partial_+$.  The coefficient derivations
cancel in its definition, so it acts on $R$-linear maps.
It raises coefficient order by at least one, by
\eqref{eq:delta-positive}. Equation~\eqref{eq:zero-contraction-perturbation}
gives the contraction \eqref{eq:full-hom-contraction}; its
finiteness is exactly
Lemma~\ref{lem:pointwise-hom-finiteness}.  Each new loss from
$\mathcal H_0$ is compensated by the gain from $\partial_+$,
which proves \eqref{eq:hom-filtration}.

For equal tensor factors, averaging over the finite group
$\mathfrak S_n$ preserves the contraction identity and the coefficient
bound. Each evaluation is still a finite sum. This gives the
permutation-equivariant operator needed for higher comparisons.
\end{proof}

\Needspace*{4\baselineskip}
\begin{corollary}\label{cor:binary-comparison}
There is an $R$-linear map of degree $-1$
$\sigma_m:P_m\otimes_RP_m\to K_m$ such that, for homogeneous
$p,q\in P_m$,
\begin{align}
 D_m\sigma_m+\sigma_mD_{\otimes}
       &=\mathcal E_m,
 &\mathcal E_m(p,q)
       &=\varepsilon_m(p)q
               -(-1)^{|p||q|}\varepsilon_m(q)p,
                  \label{eq:binary-comparison}\\
 \sigma_m(q,p)&=-(-1)^{|p||q|}\sigma_m(p,q),
 &\sigma_m(P_m\otimes_RP_m)&\subset I^{m-1}P_m.
                  \label{eq:binary-filtration}
\end{align}
\end{corollary}

The first equation compares the maps $p\otimes q\mapsto
\varepsilon_m(p)q$ and $p\otimes q\mapsto
(-1)^{|p||q|}\varepsilon_m(q)p$. They have the same augmentation, so their
difference is a closed map with values in $K_m$. The controlled
contraction therefore applies.

\begin{proof}
Graded commutativity makes the two terms $R$-balanced, hence defines
$\mathcal E_m$ on $P_m\otimes_RP_m$ as an $R$-linear map. Since the
augmentation is a cochain map, $\partial\mathcal E_m=0$, and graded
commutativity also gives $\varepsilon_m\mathcal E_m=0$. Each of the
two terms contains an augmentation value in $I^m$. Set
$\sigma_m=\mathcal H_{P_m^{\otimes2},m}(\mathcal E_m)$ using the
equivariant contraction. Its identity gives
$\partial\sigma_m=\mathcal E_m$, its one-power loss gives
\eqref{eq:binary-filtration}, and equivariance preserves
antisymmetry. Thus the first comparison has exactly the bound required
by Hess's induction.
\end{proof}

For $P_m^{S^2}$, formula~\eqref{eq:sphere-binary-comparison} makes the
loss visible: the binary defect has coefficients in $I^m$, whereas its
primitive has the factor $\bar x^{m-1}$.

The three maps constructed in this section can be read together as
follows; the image bounds refer to coefficients in the target module.
\begin{center}
\small
\begin{tabular}{@{}lccc@{}}
\toprule
Map & Source and target & Degree & Image bound\\
\midrule
$\varepsilon_m$ & $P_m\to R$ & $0$ & $I^m$\\
$j_{m,1}$ & $P_{m+1}\to P_m$ & $0$ & $IP_m$\\
$\sigma_m$ & $P_m\otimes_RP_m\to K_m$ & $-1$ & $K_m\cap I^{m-1}P_m$\\
\bottomrule
\end{tabular}
\end{center}

We now have the two estimates used by the multiplicative construction:
the resolution perturbation gains a coefficient in $I$, and the Hom
contraction loses at most one such factor. The binary comparison is the
first application; the next section treats its higher compatibilities.

\section{Coherent comparisons and the multiplicative model}
\label{sec:filtered-model}

We have the module resolution and a way to solve equations with a
controlled loss of coefficients. We now organize those equations into an algebra model. We start with
the cone, examine its binary and ternary operations, and then encode
all arities at once. The resulting construction has four obligations:
define its differential, prove that its square is zero, exhibit a
Sullivan order, and prove that its projection to $R/I^m$ is a
quasi-isomorphism. The coefficient estimate must survive all four.

Fix $m\geq2$, retain $R=A\otimes\Lambda W$ and $I=(W)$, and write
$\varepsilon:(P,D_P)\to I^m$ for the resolution of
Proposition~\ref{prop:controlled-resolution}. Put $P=R\otimes G$,
$K=\ker\varepsilon$, and regard $\varepsilon$ also as a map to $R$.
The cone $Q=R\oplus sP$ of \eqref{eq:general-module-cone} is a
module model of $R/I^m$. The goal is Proposition~\ref{prop:filtered-sullivan-model}:
a relative Sullivan model containing $Q$ with explicit bounds on
all additional coefficients. The Harrison construction organizes
its higher comparison maps.

\subsection{The binary product and the ternary defect}

The binary comparison gives a product on the cone that respects the
differential. Its failure of associativity is the first possible
higher defect, and can be computed before introducing coalgebras.
Write $\sigma=\sigma_m$, and extend the product of $R$ and its module
actions on $Q$ by
\begin{equation}\label{eq:cone-binary-product}
 \mu_Q(sp,sq)=(-1)^{|p|+1}s\sigma(p,q).
\end{equation}
The antisymmetry of $\sigma$ makes this product graded commutative.
Using $D_Q(sp)=\varepsilon(p)-sD_Pp$, the condition $\varepsilon\sigma=0$ follows from $\sigma(P^{\otimes2})\subset K$.
With this condition, the Leibniz identity for
\eqref{eq:cone-binary-product} is equivalent to
$\partial\sigma=\mathcal E_m$. Thus $\mu_Q$ is a map of complexes.
This preliminary product depends only on the resolution; it is used
here to expose the obstruction. It is distinct from the associative
square-zero product determined by a retraction in
Lemma~\ref{lem:square-zero-product}.

For homogeneous $p,q,r\in P$, the two parenthesizations give
\begin{align}
 \mu_Q(\mu_Q(sp,sq),sr)
   &=(-1)^{|q|+1}s\sigma(\sigma(p,q),r),\notag\\
 \mu_Q(sp,\mu_Q(sq,sr))
   &=(-1)^{|p|+|q|}s\sigma(p,\sigma(q,r)).
 \label{eq:ternary-parenthesizations}
\end{align}
The normalized ternary defect is consequently the degree-$-2$ map
\begin{equation}\label{eq:explicit-ternary-defect}
 \mathcal F_3(p,q,r)=
 \sigma(\sigma(p,q),r)+(-1)^{|p|}\sigma(p,\sigma(q,r)).
\end{equation}
To check the shifts, let $\operatorname{Ass}$ denote the first
parenthesization minus the second. The graded tensor convention gives
$(s\otimes s\otimes s)(p\otimes q\otimes r)
=(-1)^{|q|}sp\otimes sq\otimes sr$, so
$\operatorname{Ass}(s\otimes s\otimes s)=-s\mathcal F_3$.
Both parenthesizations are maps of complexes, and their difference
vanishes whenever an input lies in $R$. It descends to a cochain map
$(Q/R)^{\otimes_R3}\to Q$, with image in $sK\subset sP$.
Compose with the cochain projection
$Q\to Q/R\cong(sP,-sD_Ps^{-1})$, which is the identity on the
$sP$ summand. Removing the suspensions in the displayed identity then
gives $\partial\mathcal F_3=0$ in $\Hom_R(P^{\otimes_R3},K)$.

The coefficients are as important as this identity. Each summand of
$\mathcal F_3$ contains two occurrences of $\sigma$; $R$-linearity gives
\[
 \mathcal F_3(P^{\otimes_R3})\subset K\cap I^{2m-2}P.
\]
Hence $\tau_3=-\mathcal H_{P^{\otimes3},m}(\mathcal F_3)$ has degree
$-3$, satisfies $\partial\tau_3=-\mathcal F_3$, and takes values in
$K\cap I^{2m-3}P$. At the limiting value $m=2$, the binary defect is
in $I^2$, its primitive $\sigma$ is in $I$, the ternary defect is again
in $I^2$, and its primitive is again in $I$. There is no spare power in
this estimate. Higher arities repeat this two-corrections/one-primitive
calculation; their symmetry relations are what the Harrison language
organizes.

For the explicit choices on $P_m^{S^2}$, the ternary defect is zero.
For example, on $(e_0,f_0,f_0)$ its two terms are
$2\bar x^{2m-2}e_2$ and $-2\bar x^{2m-2}e_2$.
The cancellation holds on every triple because the product giving
$\sigma_1$ is transported from the
associative CDGA
\[
 \mathcal E_{\mathrm{sph}}=R\otimes\Lambda(u_1,v_2),\qquad
 Du=\bar x,\qquad Dv=\bar a-\lambda u
\]
through the cone identification \eqref{eq:sphere-cone-identification}.
The resolutions for general $m$ are even degree shifts of $P_1^{S^2}$, and
\eqref{eq:sphere-binary-comparison} gives
$\sigma_m=\bar x^{m-1}\sigma_1$. Thus
\[
 \mathcal F_{3,m}=\bar x^{2m-2}\mathcal F_{3,1}=0.
\]
The sphere illustrates coefficient control, the ideal target, and the
failure of a deeper target; these explicit choices do not require a
nonzero ternary correction. The general construction must allow for
such defects, which is the purpose of the higher comparisons.

\subsection{Harrison notation and higher comparisons}

The quotient of the cone is $Q/R=sP$. A bar construction suspends that
quotient once more, giving cogenerators $s^2P$. The scalar term
$\varepsilon$ in the cone differential becomes the curvature. We use
Harrison coalgebras to encode the product comparisons and their
symmetry relations in every arity.

We realize these coalgebras as Harrison quotients of tensor
coalgebras. If $T$ is a
free graded $R$-module, $\operatorname{coLie}_R(T)$ denotes the conilpotent
cofree graded Lie coalgebra cogenerated by $T$. It is the direct sum of
its arity components $\operatorname{coLie}_{R,n}(T)$, $n\geq1$.
We realize this Lie coalgebra as the quotient of the reduced tensor
coalgebra by the shuffle decomposables. If $q$ denotes the quotient
map, $\overline\Delta$ reduced deconcatenation, and $\chi$ the Koszul interchange
of the two output factors, its cobracket is defined by
\[
 \nabla q=(q\otimes q)(\overline\Delta-\chi\overline\Delta).
\]
In characteristic zero the Harrison idempotents $e_n$ split the
quotient, so each arity component is a natural direct summand of
$T^{\otimes_R n}$. In this split realization the transported
cobracket is $(e\otimes e)(\overline\Delta-\chi\overline\Delta)$, where $e=\bigoplus_n e_n$. The output projections are part of this
transported cobracket. All permutations of homogeneous factors carry
their Koszul signs.

Write $[u\mid v]$ for the class of $u\otimes v$ in arity
two. The quotient relation, its chosen representative, and the
cobracket are
\begin{align}
 [u\mid v]&=-(-1)^{|u||v|}[v\mid u],\notag\\
 e_2(u\otimes v)&=\tfrac12
       \bigl(u\otimes v-(-1)^{|u||v|}v\otimes u\bigr),\notag\\
 \nabla[u\mid v]&=u\otimes v-(-1)^{|u||v|}v\otimes u.
 \label{eq:harrison-arity-two}
\end{align}
The normalization is checked in Appendix~\ref{app:binary-cobar-signs}.
In all arities we use the fully antisymmetric cobracket
\[
 \nabla\colon\operatorname{coLie}_R(T)
       \longrightarrow\operatorname{coLie}_R(T)\otimes_R
                         \operatorname{coLie}_R(T).
\]
For a cofree conilpotent Lie coalgebra, a coderivation is determined
by its composites with the projection onto $T$; these composites will be called its arity components. An arity-$n$ map from this quotient to $T$ extends uniquely to a
coderivation whose other cogenerator components vanish. Its value on $[u\mid v]$ is the prescribed arity-two component.
These are the cofree-coalgebra facts of \cite[Proposition~6.3.17]{LodayVallette2012}.
For the free module $P=R\otimes G$, extension of scalars gives the
graded identification
\[
 \operatorname{coLie}_R(s^2P)
 \cong R\otimes\operatorname{coLie}_{\mathbb Q}(s^2G),
\]
so the same facts apply over $R$. In the tensor realization, the
extension is the signed sum of insertions; these preserve the shuffle
decomposables. In particular, a coderivation whose cogenerator
components vanish below arity $n$ vanishes on those arities.

Set
\[
 \mathcal L=\operatorname{coLie}_R(s^2P),\qquad
 \theta\colon \mathcal L\longrightarrow R,
 \qquad
 \theta(s^2p)=\varepsilon(p),
\]
and let $\theta$ vanish on arities at least two. Its degree is two.
The unary differential on $\mathcal L$ induced by $D_P$ and $d_R$ is denoted by
$b_1$; thus $b_1(s^2p)=s^2D_Pp$.

\Needspace*{4\baselineskip}
\begin{lemma}\label{lem:curved-comparison}
For $m\geq2$, there is a degree-one coderivation $b=b_1+b_2+\cdots$ of $\mathcal L$
such that $b_1$ is over $d_R$, every $b_n$ with $n\ge2$ is $R$-linear, and
\begin{equation}\label{eq:curved-identities}
 b^2=(\theta\otimes\operatorname{id})\nabla,
 \qquad d_R\theta=\theta b.
\end{equation}
The second identity is the cochain condition for the even map $\theta$
over $d_R$. For $n\geq2$, the arity component $b_n$ takes its values in $s^2K$.
After removing the suspensions, it is represented by a map
\[
 \tau_n\colon P^{\otimes_R n}\longrightarrow K,
 \qquad |\tau_n|=3-2n,
\]
which satisfies $\tau_n=\tau_ne_n$ for the corresponding Harrison
projector and has the coefficient bound
\begin{equation}\label{eq:higher-coefficient-bound}
 \tau_n(P^{\otimes_R n})\subset K\cap I^{m-1}P.
\end{equation}
For $n=2$, this map is the comparison $\sigma$ of
Corollary~\ref{cor:binary-comparison}.
\end{lemma}

\begin{proof}
We solve the curvature equation one arity at a time. At arity two the
equation is the binary comparison already constructed; at higher
arities its first failure is a closed map with two controlled
coefficients.

\emph{1. Binary and ternary arities.}
The antisymmetry of $\sigma$ defines $b_2$ by
\[
 b_2[s^2p\mid s^2q]=s^2\sigma(p,q).
\]
Its extension is an $R$-linear coderivation of degree one, whereas
$b_1$ is a coderivation over $d_R$. Together with \eqref{eq:harrison-arity-two}, the identity
\[
 D_P\sigma+\sigma D_{P\otimes P}
 =\bigl[p\otimes q\longmapsto
       \varepsilon(p)q-(-1)^{|p||q|}\varepsilon(q)p\bigr]
\]
is exactly the first identity in \eqref{eq:curved-identities} in arity
two. The second identity holds in arity one because $\varepsilon$ is a
cochain map, and in arity two because $\varepsilon\sigma=0$.
The binary coefficient bound is $I^{m-1}$.

In arity three its extension is explicitly
\begin{align*}
 b_2[s^2p\mid s^2q\mid s^2r]
  ={}&[s^2\sigma(p,q)\mid s^2r]\\
    &+(-1)^{|p|}[s^2p\mid s^2\sigma(q,r)].
\end{align*}
Graded antisymmetry makes this expression vanish on the shuffle
relations. Projecting its second application to the cogenerators gives
$s^2\mathcal F_3(p,q,r)$, exactly the defect in
\eqref{eq:explicit-ternary-defect}. Thus the first higher correction is
the $\tau_3$ constructed there, restricted to the Harrison quotient:
$\mathcal F_3e_3=\mathcal F_3$, and equivariance of the contraction gives
$\tau_3e_3=\tau_3$.
The following induction proves the same statements in all arities.

\emph{2. Why the next defect is closed.}
Suppose the curvature identity holds through arity $n-1$ for the
coderivation $b=b_1+\cdots+b_{n-1}$. The other identity,
$d_R\theta=\theta b$, holds in every arity: $\varepsilon$ is a cochain
map, and all higher components already chosen take values in $s^2K$. Write
\[
 \Theta_\theta=(\theta\otimes\operatorname{id})\nabla,
 \qquad \mathcal O=b^2-\Theta_\theta.
\]
The map $\Theta_\theta$ is $R$-linear of degree two. Applying
$\theta\otimes1\otimes1$ to the co-Jacobi identity and using
co-antisymmetry gives
\[
 \nabla\Theta_\theta
 = (\Theta_\theta\otimes1+1\otimes\Theta_\theta)\nabla;
\]
thus it is a coderivation. This is the coalgebra form of the Jacobi
identity for an inner adjoint operator. Since $b$ is over $d_R$ and
$d_R^2=0$, the graded Leibniz rule also makes $b^2$, and hence
$\mathcal O$, $R$-linear. The already constructed components make
$\mathcal O$ zero in arities below $n$.
The coderivation identity for $b$ gives
$[b,\Theta_\theta]=0$ when $d_R\theta-\theta b=0$; equivalently, commuting
$b$ past contraction of the cobracket by $\theta$ replaces $\theta$
by its differential $d_R\theta-\theta b$. Hence
\[
 [b,\mathcal O]=[b,b^2]-[b,\Theta_\theta]=0.
\]
Since $\mathcal O$ has no component below arity $n$, only the unary part of
$b$ contributes to the first nonzero component of this commutator.
Thus $\mathcal O_n$ is closed for the Hom differential induced by $b_1$. This is
the condition needed for the controlled contraction.

\emph{3. Choose a primitive and preserve its symmetries.}
For $n\geq3$, the cogenerator component of $\Theta_\theta$ is zero.
Consequently, $\mathcal O_n$ is the signed sum of the insertions of $b_j$ into
$b_i$, with $i,j\geq2$ and $i+j=n+1$. Each summand takes its values in
$s^2K$. The induction hypothesis is that every previously constructed
$b_j$, $j\ge2$, has coefficients in $I^{m-1}$; its initial case $b_2$
is Corollary~\ref{cor:binary-comparison}. Each insertion therefore has coefficients in
$I^{2m-2}$. Proposition~\ref{prop:controlled-hom} supplies a primitive
with coefficients in $I^{2m-3}$. The estimate
\[
 2m-3\geq m-1\qquad(m\geq2)
\]
keeps the same coefficient bound at every arity; for $m=2$ it is an
equality, as the ternary calculation showed. After removing the double
suspensions, precompose $\mathcal O_n$ with the Harrison projector $e_n$ and
choose the primitive $-\mathcal H(\mathcal O_ne_n)$. Its differential is
$-\mathcal O_ne_n$, so adding the corresponding $b_n$ cancels this defect.
It remains to check that the primitive itself respects the quotient. Since $\mathcal H$ commutes with
signed permutations, it commutes with precomposition by $e_n$, and
\[
 \mathcal H(\mathcal O_ne_n)=\mathcal H(\mathcal O_ne_n)e_n.
\]
The primitive therefore descends to the Harrison quotient. It defines
an arity component of an $R$-linear coderivation and still takes its
values in $s^2K$.
Thus $\theta b_n=0$, the second curved identity is preserved, and the
first holds through arity $n$.

The induction supplies compatible components in every arity, all with
the bound \eqref{eq:higher-coefficient-bound}. On an element of arity $N$, only
$b_1,\ldots,b_N$ can act. This defines the coderivation on the ordinary conilpotent coalgebra.
\end{proof}

\subsection{From comparisons to a Sullivan model}

The equations for the higher comparisons now become a single algebra
differential. We verify the four model obligations in turn.

\emph{Defining the differential.}
Let $t=s^{-1}$, so that $t$ raises degree by one. The graded-commutative
algebra freely generated over $R$ by $t\mathcal L$ is
\[
 E=\Lambda_R(t\mathcal L)=R\otimes\Lambda Z,
 \qquad
 Z=\bigoplus_{p\geq0}Z_p,
 \qquad
 Z_{n-1}=t\operatorname{coLie}_{\mathbb Q,n}(s^2G).
\]
In particular, $Z_0=sG$ and $R\otimes Z_0=sP$.
Give each generator in $Z_{n-1}$ leaf arity $n$ and extend this
additively to products. We call the index $p$ of $Z_p$ its
\emph{level}, so a generator of level $p$ has leaf arity $p+1$.
The grading conventions are collected in
Table~\ref{tab:construction-gradings}.
For $\nabla(c)=\sum c_1\otimes c_2$, define
\begin{equation}\label{eq:curved-cobar-differential}
 D_E(tc)=\theta(c)-tb(c)
       -\frac12\sum(-1)^{|c_1|}(tc_1)(tc_2),
 \qquad D_E|_R=d_R,
\end{equation}
and extend by the graded Leibniz rule.

\begin{lemma}\label{lem:cobar-square-zero}
The derivation $D_E$ satisfies $D_E^2=0$. On
$R\oplus(R\otimes Z_0)$ it is the cone differential
\eqref{eq:general-module-cone}.
\end{lemma}

\begin{proof}
\emph{The square-zero identity.}
The curvature equation was chosen to cancel the low-degree part of
this square. To see all cancellations, separate $D_E^2(tc)$ by the
number of cobar generators in a term. The scalar term is
$(d_R\theta-\theta b)(c)$. The linear term is
\[
 t\bigl(b^2-(\theta\otimes\operatorname{id})\nabla\bigr)(c).
\]
Both vanish by \eqref{eq:curved-identities}. The quadratic term vanishes
by the coderivation identity for $b$, and the cubic term by the
co-Jacobi identity for $\nabla$. These are all the terms.
In the linear calculation, the two scalar substitutions into the
quadratic expression combine because
\[
 (\theta\otimes\operatorname{id})\nabla
 =\frac12(\theta\otimes\operatorname{id}
                  -\operatorname{id}\otimes\theta)\nabla;
\]
the minus sign uses the even degree of $\theta$ and the antisymmetry of
the cobracket. In arity one the cobracket is zero, and
\eqref{eq:curved-cobar-differential} reads
$D_E(sp)=\varepsilon(p)-sD_Pp$.

The arity-two normalization is verified explicitly in
Appendix~\ref{app:binary-cobar-signs}.
\end{proof}

The differential algebra now contains the cone. The proposition below
provides a Sullivan order and proves that its cohomology is that of
$R/I^m$. Its filtered shape is the one used in Hess's
construction \cite{Hess1991}; see
\cite[Proposition~29.10 and Section~29\textup{(f)}]{FHT2001}.
Here we also need the coefficient bound of
Proposition~\ref{prop:relative-hess-conversion}.

\Needspace*{4\baselineskip}
\begin{proposition}\label{prop:filtered-sullivan-model}
For $m\geq2$, the quotient map has a factorization
\[
 R\lhook\joinrel\longrightarrow (R\otimes\Lambda Z,D_E)
       \xrightarrow{\zeta\,\simeq}R/I^m
\]
in which the first map is a relative Sullivan extension and $\zeta$
sends $Z$ to zero. With $Z=\bigoplus_{p\geq0}Z_p$ as above,
\begin{equation}\label{eq:filtered-sullivan-shape}
 \begin{aligned}
 D_E Z_p\subset{}&
 \sum_{i+j=p-1}R\otimes Z_iZ_j
 +R^{>0}\otimes Z_p^{\mathrm{earlier}}\\
 &+I^{m-1}\otimes Z_{<p}+\delta_{p0}I^m.
 \end{aligned}
\end{equation}
Here the Sullivan basis is the union of homogeneous bases of the
$Z_p$, ordered by degree, and ``earlier'' at fixed level means
smaller degree. The symbol $\delta_{p0}$ is the Kronecker delta:
the scalar term occurs only on $Z_0$. The quadratic term is absent
for $p=0$. The term arising from $b_k$, $k\geq2$, lowers $p$ by $k-1$
and has coefficient order at least $m-1$. For $k\ge3$, the
construction gives the stronger bound $I^{2m-3}$.
Furthermore, $Q=R\oplus(R\otimes Z_0)$ is a subcomplex, and
$\zeta|_Q$ is a surjective quasi-isomorphism.
\end{proposition}

\begin{proof}
\emph{The filtered shape and the Sullivan order.}
The differential and its square-zero identity were established in
\eqref{eq:curved-cobar-differential} and
Lemma~\ref{lem:cobar-square-zero}. To read its shape, consider its four
sources. The cobracket splits arity $n$ into two positive arities and
therefore gives $Z_iZ_j$ with $i+j=p-1$. The unary component $b_1$
keeps the level and has the positive-degree coefficients of $D_P$.
Each $b_k$ for $k\ge2$ lowers level by $k-1$ and has coefficients in
$I^{m-1}$. Finally, $\theta$ contributes a scalar in $I^m$, only at
level zero. These are the four terms of
\eqref{eq:filtered-sullivan-shape}.

Every generator of $G$ has degree at least $2m$. A generator of $Z_{n-1}$
represented by $g_1,\ldots,g_n$ therefore has degree
\begin{equation}\label{eq:arity-degree-bound}
 \sum_{i=1}^n|g_i|-2n+1\ \geq\ n(2m-2)+1.
\end{equation}
Since $m\ge2$, all these degrees are at least three. Moreover,
\eqref{eq:arity-degree-bound} bounds the arity in any fixed degree;
the degreewise finiteness of $G$ therefore makes $Z$ finite-dimensional
in every degree. Since $R^1=0$, the positive coefficients of every
linear term have degree at least two; the same holds for coefficients
in $I^{m-1}$. In a quadratic term of $D_Ez$, each factor has degree
at least three, so the other has degree at most $|z|-2$.
Every generator occurring in $D_Ez$ thus has strictly smaller degree
than $z$. Ordering the finite homogeneous bases by degree gives a
Sullivan order and makes $R\longrightarrow E$ a relative Sullivan
extension.

\emph{The quasi-isomorphism.}
The differential shape makes $\zeta$ a CDGA map, and its restriction to
$Q$ is the cone projection \eqref{eq:general-module-cone}. That
restriction is already a surjective quasi-isomorphism. Thus the
remaining question is whether adjoining the higher generators has
introduced cohomology. It suffices to prove that $E/Q$ is acyclic.

\emph{First fix the leaf arity.}
The duality between Com and Lie and the Koszulity of Com
\cite[Propositions~13.1.6--13.1.7]{LodayVallette2012} give the required
universal contraction. Applying the bar--cobar quasi-isomorphism of
\cite[Theorem~11.3.6 and Corollary~11.3.8]{LodayVallette2012}
to the nonunital square-zero algebra $sT$, for a graded rational
vector space $T$, and then adjoining a unit shows that the
commutative cobar complex of
$\operatorname{coLie}_{\mathbb Q}(s^2T)$ has cohomology only in
leaf arities zero and one. In our convention, the Koszul dual
cooperad of Com is the operadic suspension of the Lie cooperad,
whose coalgebras are represented here by $\operatorname{coLie}$.
Thus the Harrison bar construction suspends the augmentation ideal
$sT$ once more, giving cogenerators $s^2T$; the cobar construction
then uses $t=s^{-1}$ as in \eqref{eq:curved-cobar-differential}.
The internal differential and multiplication on that ideal are
zero; its bar coalgebra is therefore
$\operatorname{coLie}_{\mathbb Q}(s^2T)$ with zero differential.
Give $sT$ weight one. Because its multiplication is zero, the
bar--cobar differential and quasi-isomorphism preserve total weight,
which here is leaf arity. The unital target has only weights zero
and one, so every higher-arity part of the cobar complex is acyclic.
For an explicit universal complex, let $T_n$ have basis
$x_1,\ldots,x_n$, all of degree zero, and set
\begin{equation}\label{eq:universal-arity-complex}
 \mathcal C_n=
 \bigl[\Lambda\bigl(t\operatorname{coLie}_{\mathbb Q}(s^2T_n)\bigr)
       \bigr]_{(1,\ldots,1)}.
\end{equation}
The subscript denotes the component in which each $x_i$ occurs once.
Its differential is the cobracket term of
\eqref{eq:curved-cobar-differential}, with no curvature or internal
differential. Give $\mathcal C_n$ the right $\mathfrak S_n$-action by
inverse relabelling of the $x_i$. A monomial with $\ell$ cobar factors
has degree $\ell-2n$: the $n$ leaves contribute $s^{2n}$ and the
$\ell$ factors contribute $t^\ell$. Thus all suspension shifts are
included in $\mathcal C_n$. The Koszul fact above says that
$\mathcal C_n$ is acyclic for $n\geq2$; $\mathcal C_1$ is one copy of
$\mathbb Q$ in degree $-1$. In arity two, putting $y_i=ts^2x_i$ gives
\[
 \partial_{\mathcal C}\bigl(t[s^2x_1\mid s^2x_2]\bigr)=-y_1y_2.
\]
Thus $\mathcal C_2$ is an isomorphism between one-dimensional spaces
in degrees $-3$ and $-2$, fixing explicitly the normalization.

The arity-$n$ part of the relative arity-preserving cobar complex is
\begin{equation}\label{eq:relative-arity-evaluation}
 \mathcal C_n\otimes_{\mathbb Q[\mathfrak S_n]}P^{\otimes_R n},
\end{equation}
where $\mathfrak S_n$ acts on the tensor power by Koszul permutations.
The identification sends $c\otimes(g_1\otimes\cdots\otimes g_n)$
to the expression obtained by substituting $x_i\mapsto g_i$ in the
multilinear Harrison expression $c$, with the Koszul signs for
moving homogeneous inputs past each other and past the suspension
symbols. It extends over $R$ with the graded tensor action below.
Under this substitution the cobracket gives
$\partial_{\mathcal C}\otimes1$, and $b_1$, induced by $s^2D_P$,
gives the signed tensor differential. Signed permutations commute
with that differential, so the identification descends to the
$\mathfrak S_n$-coinvariants. For homogeneous $a\in R$, $c\in\mathcal C_n$,
and $p\in P^{\otimes_R n}$, use the graded tensor action
$a(c\otimes p)=(-1)^{|a||c|}c\otimes ap$. The differential is
\[
 \partial_{\mathcal C}c\otimes p
       +(-1)^{|c|}c\otimes D_{P^{\otimes n}}p.
\]
For $n=1$ this recovers the complex $Q/R=sP$ with differential
$-sD_P$; it checks the unary convention before passing to $E/Q$,
whose arities zero and one vanish.

For $n\ge2$, choose an equivariant contraction $h_n$ of $\mathcal C_n$,
using semisimplicity of $\mathbb Q[\mathfrak S_n]$. After evaluation,
$h_n\otimes1$ contracts \eqref{eq:relative-arity-evaluation}. The
displayed tensor signs make it anticommute with the internal
differential on $P$, so the contraction survives that differential.
This proves relative acyclicity at each fixed arity, including when
$D_P$ has nonconstant coefficients.

\emph{Then reassemble the arities.}
Filter the module complex $E/Q$ increasingly by total leaf arity. Its arities zero and one vanish. The internal and
cobracket terms of $D_E$ preserve leaf arity, and the curvature
and higher components lower it. Thus its associated graded complex
is the direct sum, for $n\geq2$, of precisely the relative acyclic
complexes just described. Each finite stage is therefore acyclic,
by induction on the filtration, and their exhaustive union $E/Q$
is acyclic because cohomology commutes with directed unions of
module complexes. There is also a useful degree bound: for a
monomial with $\ell$ cobar factors and total leaf arity $n$,
\eqref{eq:arity-degree-bound} gives
\[
 |\text{monomial}|\ge n(2m-2)+\ell.
\]
Coefficients in $R$ have nonnegative degree, so total arity is bounded
in every fixed cohomological degree. In particular, the construction
uses ordinary direct sums, with no completion. The exact sequence
\[
 0\longrightarrow Q\longrightarrow E\longrightarrow E/Q\longrightarrow0
\]
now shows that $Q\longrightarrow E$ is a quasi-isomorphism. Since its
composite with $\zeta$ is a quasi-isomorphism, so is $\zeta$.

A unital curved commutative cobar construction and its Koszul theorem
appear in \cite[Definition~132 and Theorem~135]{LeGrignou2019}.
Here the relative evaluation and the leaf-arity filtration retain
the bounds in powers of $I$ needed for the next step.
\end{proof}

\subsection{The module-to-algebra conversion}\label{subsec:proof-B}

The Sullivan model and its quasi-isomorphism are established. The input still
needed for Hess's induction is a retraction on the cone whose kernel
values lie in $I$. An ideal-valued nullhomotopy supplies that
input, without changing the chosen values on the cone layer.

\begin{theorem}[Module-to-algebra conversion]\label{thm:ideal-nullhomotopy}
Let $A=(\Lambda V,d)$ and $R=(A\otimes\Lambda W,d_R)$ be
simply connected minimal Sullivan algebras of finite type, with
$d_R|_A=d$ and $d_R(W)\subset (W)$. These are the standing split-model
hypotheses of \S\ref{sec:controlled-resolutions}; in particular,
$W^{\le1}=0$ and $W$ is finite-dimensional in each degree. Put $I=(W)$.
If $m\geq2$ and the inclusion $I^m\longrightarrow I$ is zero
in the derived category of $R$-modules, then the quotient
$R\longrightarrow R/I^m$ admits a homotopy retraction in the
category of CDGAs. More precisely, if $h\colon P\to I$ is an $R$-linear
nullhomotopy of the controlled resolution $\varepsilon\colon P\to I^m$
with values in $I$, then the controlled Sullivan model
$R\hookrightarrow E=R\otimes\Lambda Z$ of
Proposition~\ref{prop:filtered-sullivan-model} admits a CDGA
retraction $\rho\colon E\to R$ with $\rho(sp)=h(p)$ for $p\in P$ and
$\rho(Z)\subset I$.
\end{theorem}

\begin{proof}
Use the controlled resolution $P\xrightarrow{\varepsilon}I^m$
from which the model in Proposition~\ref{prop:filtered-sullivan-model}
is constructed. Its semifreeness and the
derived nullity give an $R$-linear map $h\colon P\to I$ of degree $-1$
with $d_Rh+hD_P=\varepsilon$; if such a map is given, we use it.
On the cone $Q=R\oplus sP$ define
$\rho_0|_R=\operatorname{id}$ and $\rho_0(sp)=h(p)$. This is an $R$-linear cochain
retraction. The kernel of $Q\longrightarrow R/I^m$ is, as a graded
module, $I^m\oplus sP$, and its image under $\rho_0$ lies in $I$.
The model of Proposition~\ref{prop:filtered-sullivan-model} satisfies
the three hypotheses of Proposition~\ref{prop:relative-hess-conversion}
with $\mathfrak a=I$: its cone layer is a subcomplex with the required
quasi-isomorphism, its differential has the shape
\eqref{eq:filtered-sullivan-shape}, which is \eqref{eq:general-hess-shape}
for $\mathfrak a=I$, and the retraction $\rho_0$ just defined sends $\ker(\zeta|_Q)$ into
$I$. Proposition~\ref{prop:relative-hess-conversion} gives $\rho$. The
construction there starts with $\Phi|_{Z_0}$ equal to the inclusion of
$sP$ into $Q$, so $\rho(sp)=\rho_0(sp)=h(p)$.
\end{proof}

The condition $\rho(Z)\subset I$ is the crucial conclusion. It makes
the retraction strictly compatible with the projection to $A$, which is
the additional requirement entering relative category; this is how
Theorem~\ref{thm:ideal-nullhomotopy} was used in the proof of
Theorem~\ref{thm:relative-category}. The proof of that theorem, and
consequently the arguments for Theorems~\ref{thm:main} and
\ref{thm:poincare-main}, are now complete.

\begin{table}[htbp]
\centering\small
\begin{tabular}{@{}p{.28\textwidth}p{.65\textwidth}@{}}
\toprule
Symbol & Role in the construction\\
\midrule
$\mathfrak a$, $N_p$ & General differential ideal and acyclic kernels used in the extension argument.\\
$C$, $\mathfrak n$ & Conormal module $I/I^2$ and augmentation ideal of $\Lambda_A C$.\\
$P_m=R\otimes G_m$, $K_m$ & Controlled resolution of $I^m$ and its augmentation kernel.\\
$D_m$, $D_P$ & Differential of $P_m$; $D_P$ is used when the exponent $m$ is fixed.\\
$D_{\mathrm{lin}}$, $D_{\mathrm{lin},m}$ & Length-preserving differential on $R$ and its induced operator on the resolution.\\
$D_E$, $\DD_R$ & Cobar differential on $E$; derived dual $\DD_R=\RHom_R(-,R)$.\\
$\mathsf h$, $H_{m,0}$ & Bar-retract homotopy and separate contraction of the linear augmentation kernel.\\
$\mathcal H_0$, $\mathcal H_{\mathcal P,m}$ & Hom contractions before and after the nonlinear perturbation, each losing at most one coefficient power.\\
$\mathcal E_m$, $\sigma_m$ & Binary comparison defect and its controlled primitive.\\
$\operatorname{Ass}$, $\mathcal F_3$ & Cone associator and its normalized ternary defect.\\
$\mathcal L$, $b_n$, $\theta$ & Harrison coalgebra, arity components of its coderivation, and curvature.\\
$Z_p$, $\mathcal C_n$ & Cobar generators of level $p$ and universal multilinear complex of leaf arity $n$.\\
\bottomrule
\end{tabular}
\par\vspace{8pt}
\caption{Notation for the technical construction; gradings are listed in
Table~\ref{tab:construction-gradings}. The notation $h$ and $\Dcat(R)$
is recalled in Section~\ref{sec:preliminaries}.}
\label{tab:technical-notation}
\end{table}

\appendix
\section{Universal splittings for the linear resolutions}
\label{app:universal-linear-resolution}

The splittings used in Lemma~\ref{lem:natural-linear-resolution}
must respect an internal differential whose matrix entries can lie in
the base algebra. We obtain this compatibility by splitting finite
complexes of rational symmetric-group representations before evaluating
on the differential module. Signed permutations commute with its
tensor differential, including the nonconstant coefficients of
Remark~\ref{ex:nonconstant-conormal}.

Two different homotopies are needed. The bar homotopy retracts a large
free resolution onto the linear one and is linear over the full free
algebra. The contraction of the augmentation kernel is only linear over
the base; it loses one unit of coefficient weight. Keeping them separate
is essential to the Hom construction in
\S\ref{sec:hom-contraction}.

The evaluation of representations used here is the usual Schur-functor
construction \cite[Sections~5.1 and~6.2]{LodayVallette2012}. We give the
splittings and their estimates explicitly, using the normalized
perturbation formula \eqref{eq:perturbation-formulas} to restore the
first bar face.

\subsection{Expressing the bar resolution universally}

The first task is to describe the bar generators independently of the
module on which they will be evaluated. Let $A$ be a graded-commutative algebra, let $C$ be a graded
free $A$-module, and put $R_{\mathrm{lin}}=\Lambda_A C$, $\mathfrak n=(C)$,
and fix $m\geq1$.  For the moment there is no internal
differential.  Give $C$ weight one.  The normalized relative
bar resolution has terms
\begin{equation}\label{eq:appendix-bar-terms}
 \mathcal B_{m,j}=R_{\mathrm{lin}}\otimes_A \mathsf B_{m,j}(C),
 \qquad \mathsf B_{m,j}(C)=\mathfrak n^{\otimes_Aj}\otimes_A \mathfrak n^m,
 \qquad j\geq0.
\end{equation}
The index $j$ is homological.  Write its elements as
$a_0[a_1|\cdots|a_j]u$.  We use untwisted coefficients before totalization: the bar differential
is the alternating sum of the $j+1$ consecutive multiplications, with
no additional Koszul signs at this stage. The signs of its cochain
totalization are fixed in \eqref{eq:appendix-totalization}. We separate the face that changes
the free coefficient, because it is the one that will induce the linear
resolution differential.

\Needspace*{4\baselineskip}
Split the bar differential as $\beta+\ell$, where $\ell$
multiplies $a_1$ into the free coefficient $a_0$, and $\beta$
is the sum of the remaining multiplications.  Thus $\beta$
acts on $\mathsf B_{m,*}(C)$ and
\begin{align}
 \beta[a_1|\cdots|a_j]u
  ={}&\sum_{i=1}^{j-1}(-1)^i
       [a_1|\cdots|a_ia_{i+1}|\cdots|a_j]u\notag\\
    &+(-1)^j[a_1|\cdots|a_{j-1}](a_ju).
                       \label{eq:appendix-bar-beta}
\end{align}
At $j=0$ this is zero. Associativity gives
$d_i d_k=d_{k-1}d_i$ for the multiplication faces with $i<k$;
the terms of the alternating double sum consequently cancel in pairs.
Thus $\beta$ alone is a differential on the generator complex.
For the full bar differential, the same face identities give
\[
 \beta\ell+\ell\beta+\ell^2=0,
 \qquad (\beta+\ell)^2=0.
\]
In general $\ell^2$ is nonzero: the first face can successively
multiply two bar entries into the coefficient. Its vanishing on the
linear subcomplex will follow from the weight calculation below.
The augmented complex $(\mathcal B_{m,*},\beta+\ell)\to \mathfrak n^m$
has augmentation
\[
 \varepsilon_{\mathrm{bar}}(a_0[]u)=a_0u,
 \qquad
 \varepsilon_{\mathrm{bar}}(\mathcal B_{m,j})=0\quad(j>0).
\]
To see exactness, insert the positive-weight part of $a_0$ as the
first bar entry and use $u\mapsto1[]u$ in augmented degree $-1$.
The bar-face identities give a contraction of the augmented complex
over $A$.

Fix total generator weight $N$. Each bar factor is a nonempty block
of the $N$ inputs, and the last ideal factor is a block of size at
least $m$. This gives the following universal representation.
For $N\geq m$, let $\mathsf E_{m,N,j}$ be the rational permutation
module on ordered partitions
\[
              (B_1|\cdots|B_j|B_0)
              \quad\text{of }\{1,\ldots,N\},
 \qquad |B_i|\geq1,\quad |B_0|\geq m.
\]
The symmetric group acts on labels, and we regard this as a
right $\mathbb Q[\mathfrak S_N]$-module.  Equivalently,
\begin{equation}\label{eq:appendix-representation-module}
 \mathsf E_{m,N,j}
  =\bigoplus_{\substack{a_1+\cdots+a_j+a_0=N\\
                       a_i\geq1\ (i>0),\ a_0\geq m}}
    \operatorname{Ind}_{\mathfrak S_{a_1}\times\cdots
                \times\mathfrak S_{a_j}\times\mathfrak S_{a_0}}
                      ^{\mathfrak S_N}\mathbb Q.
\end{equation}
The differential $\beta$ merges consecutive blocks, including
the last bar block with $B_0$, with the signs in
\eqref{eq:appendix-bar-beta}.  It is an equivariant map
$\mathsf E_{m,N,j}\to \mathsf E_{m,N,j-1}$.

For a right representation $\mathsf E$ define its evaluation by
\[
              \mathsf E\langle C\rangle
               =\mathsf E\otimes_{\mathbb Q[\mathfrak S_N]}
                         C^{\otimes_AN},
\]
where $\mathfrak S_N$ acts on the tensor power by signed
permutations. Semisimplicity splits every short exact sequence of rational
$\mathfrak S_N$-representations. Evaluation therefore preserves
exactness, kernels, and homology, even though the coefficient algebra
$A$ is arbitrary. In particular,
\begin{equation}\label{eq:appendix-evaluation}
 \mathsf B_{m,j}(C)_N=\mathsf E_{m,N,j}\langle C\rangle.
\end{equation}
Block merger evaluates to multiplication in $\Lambda_A C$. The signed
permutation action supplies the Koszul signs, so this description
handles odd and even generators in the same construction.

The remaining question is where the homology of this representation
complex occurs. Tensoring the bar resolution with the augmentation
$R_{\mathrm{lin}}\to A$ gives $\mathsf B_{m,*}(C)$, which computes
$\operatorname{Tor}^{R_{\mathrm{lin}}}(\mathfrak n^m,A)$. Linearity from
Lemma~\ref{lem:natural-linear-resolution} therefore predicts
\begin{equation}\label{eq:appendix-homology-weight}
                H_j(\mathsf E_{m,N,*})=0\qquad(N\ne m+j).
\end{equation}
To obtain this conclusion before evaluation, use $N$ distinct even
variables of the same degree and extract the multilinear summand. That summand of their
tensor power is the regular representation, so it recovers
$\mathsf E_{m,N,*}$ itself, with its differential. The vanishing just proved
after evaluation therefore proves \eqref{eq:appendix-homology-weight}
for the representation complex. This is the concentration needed to
choose its linear part canonically.

\subsection{Retracting the bar resolution onto its linear part}

We now retain the homology at its unique possible index and contract
the rest. Fix $m,N$ and put $\mathsf{Cyc}_j=\ker\beta_j$ and
$\operatorname{Bd}_j=\operatorname{im}\beta_{j+1}$. The largest index is $N-m$. By
\eqref{eq:appendix-homology-weight}, this is also the only index with
homology, and there are no incoming boundaries there. Its cycle
representatives are therefore canonical:
$\mathsf H_{N-m}=\mathsf{Cyc}_{N-m}$ and $\mathsf H_j=0$ otherwise.
Choose equivariant complements $\mathsf{Comp}_j$ so that
\begin{equation}\label{eq:appendix-splitting}
 \mathsf E_{m,N,j}=\operatorname{Bd}_j\oplus\mathsf H_j\oplus \mathsf{Comp}_j,
 \qquad \mathsf{Cyc}_j=\operatorname{Bd}_j\oplus\mathsf H_j.
\end{equation}
The restriction $\beta:\mathsf{Comp}_j\to \operatorname{Bd}_{j-1}$ is an isomorphism.
Define $i_0$ to be inclusion of $\mathsf H_j$, define $p_0$
to be projection onto it, and define $\mathsf h_0$ on $\operatorname{Bd}_j$ as the
inverse of $\beta:\mathsf{Comp}_{j+1}\to \operatorname{Bd}_j$, with $\mathsf h_0=0$ on
$\mathsf H_j\oplus \mathsf{Comp}_j$.  Then
\begin{equation}\label{eq:appendix-initial-sdr}
 \begin{gathered}
 p_0i_0=1,\qquad \beta \mathsf h_0+\mathsf h_0\beta=1-i_0p_0,\\
 \mathsf h_0i_0=p_0\mathsf h_0=\mathsf h_0^2=0.
 \end{gathered}
\end{equation}
These choices preserve weight $N$; $i_0,p_0$ preserve the index and
$\mathsf h_0$ raises it by one. Evaluate on $C$, sum over weights, and extend
$R_{\mathrm{lin}}$-linearly. This gives a normalized retract of the generator
complex, before the omitted face $\ell$ is restored. Define
\[
 G_{m,j}(C)=\mathsf{Cyc}_j(\mathsf E_{m,m+j,*})\langle C\rangle,
 \qquad P^{\mathrm{lin}}_{m,j}=R_{\mathrm{lin}}\otimes_A G_{m,j}(C),
\]
with totalization shift $-j$. The first face preserves this chosen
submodule for the following weight reason. In a top cycle $z$ of weight $N=m+j$, the $j$ bar blocks each have weight one
and the last block has weight $m$. Applying $\ell$ puts weight one
into the free coefficient and leaves a generator of weight $N-1$ at
index $j-1$. Since $\beta z=0$, the identity
$(\beta+\ell)^2z=0$ reads $\beta\ell z+\ell^2z=0$.
These two summands have different coefficient weights, respectively
one and two; each is therefore zero. Thus $\ell z$ is again a top
cycle. The chosen linear resolution is a canonical bar subcomplex,
with $\delta_m=\ell|_{P_m^{\mathrm{lin}}}$ and $\mathsf h_0\ell i_0=0$.
Although $\ell$ need not square to zero on the full bar complex, the
same calculation gives $\ell^2=0$ on this linear subcomplex.

The last identity says that restoring $\ell$ produces no correction
to the inclusion. The normalized perturbation formulas give
\begin{equation}\label{eq:appendix-bar-retract}
 \begin{gathered}
 i=i_0,\qquad \delta_m=\ell|_{P_m^{\mathrm{lin}}},\\
 p=p_0(1+\ell \mathsf h_0)^{-1},\qquad
 \mathsf h =\mathsf h_0(1+\ell \mathsf h_0)^{-1}.
 \end{gathered}
\end{equation}
All operators here use the totalization specified in
\eqref{eq:appendix-totalization}. The displayed inverses are finite on
each element: $\ell$ moves a nonempty first block from the generator
into the coefficient, thereby decreasing generator weight, whereas
$\mathsf h_0$ preserves that weight. Repeating $\ell \mathsf h_0$ must therefore stop.

The differential $\delta_m$ lowers $j$ by one and preserves
total weight.  Since the generator weights are $m+j$ and
$m+j-1$, its coefficients have weight one.  The maps $i,p$
preserve $j$ and total weight, while $\mathsf h $ raises $j$ by one
and preserves total weight.  They satisfy
\[
 pi=1,\qquad
 (\beta+\ell)\mathsf h +\mathsf h (\beta+\ell)=1-ip,
 \qquad \mathsf h i=p\mathsf h =\mathsf h^2=0.
\]
Thus the retract has the stated normalization and weight estimates.
The resolution and its inclusion are canonical, while the retraction
and homotopy depend on the equivariant complements. Its augmentation is
$\varepsilon_m^{\mathrm{lin}}=\varepsilon_{\mathrm{bar}}\circ i$.
It is multiplication on $G_{m,0}(C)=\Lambda_A^m C$ and is zero
for $j>0$.  The retraction is augmented: since
$\varepsilon_{\mathrm{bar}}\mathsf h =0$, the homotopy identity also
gives $\varepsilon_m^{\mathrm{lin}}p=\varepsilon_{\mathrm{bar}}$.

\subsection{The separate contraction of the augmentation kernel}

The bar retract does not contract the kernel of
$P_m^{\mathrm{lin}}\to \mathfrak n^m$; its homotopy acts on the complementary bar
summand. For the Hom equations we need a second map, defined directly
on this augmentation kernel. It will be $A$-linear, whereas the bar
homotopy was $R_{\mathrm{lin}}$-linear. We construct it at fixed total weight so that
its coefficient loss can be read off exactly.

At total weight $N$, the term of resolution index $j$ in
$P_m^{\mathrm{lin}}$ is represented by
\begin{equation}\label{eq:appendix-linear-representation}
 \mathsf V_{m,N,j}
  =\operatorname{Ind}_{\mathfrak S_{N-m-j}\times\mathfrak S_{m+j}}
                      ^{\mathfrak S_N}
       \bigl(\mathbb Q\otimes H_j(\mathsf E_{m,m+j,*})\bigr),
 \qquad N\geq m+j.
\end{equation}
Set this representation to zero otherwise. The map $\delta_m$ is
$\mathfrak S_N$-equivariant: it is the restriction of block merger
into the coefficient block. This restriction is defined before
evaluation: evaluate the block-merger map on $N$ distinct even
variables and take the multilinear summand, which is the regular
representation. The preservation of the linear subcomplex proved
above then detects the same preservation universally, so the map
lands in $\mathsf V_{m,N,j-1}$ itself. Multiplication
gives the augmentation $\mathsf V_{m,N,0}\to\mathbb Q$, where the
target is the trivial representation representing
$(\mathfrak n^m)_N=\Lambda_A^N C$.  Define
\[
 \mathsf V^K_{m,N,0}=\ker(\mathsf V_{m,N,0}\to\mathbb Q),
 \qquad \mathsf V^K_{m,N,j}=\mathsf V_{m,N,j}\quad(j>0).
\]
This kernel complex is exact: the augmented bar complex is exact,
and \eqref{eq:appendix-bar-retract} is an augmented deformation
retract. As before, evaluation on distinct even variables of the same degree followed by
multilinear extraction detects exactness of the representations
themselves.

Choose equivariant complements to its cycles.  If
$\mathsf{Cyc}^K_j=\ker(\delta_m:\mathsf V^K_{m,N,j}\to \mathsf V^K_{m,N,j-1})$ and
$\mathsf V^K_{m,N,j}=\mathsf{Cyc}^K_j\oplus \mathsf{Comp}^K_j$, exactness identifies
$\delta_m:\mathsf{Comp}^K_{j+1}\xrightarrow{\cong}\mathsf{Cyc}^K_j$.
Define $H_{m,0}$ to be this inverse on $\mathsf{Cyc}^K_j$ and zero on
$\mathsf{Comp}^K_j$.  It follows that
\begin{equation}\label{eq:appendix-kernel-contraction}
 \delta_mH_{m,0}+H_{m,0}\delta_m=1,
 \qquad H_{m,0}^2=0.
\end{equation}
Evaluation identifies this exact representation complex with the
augmentation kernel and produces its $A$-linear contraction. The
weight count explains its cost: at total weight $N$, a source
coefficient has weight $N-(m+j)$, while a target coefficient has
weight $N-(m+j+1)$. The homotopy therefore loses exactly one unit on
each nonzero homogeneous component, and at most one power of
$\mathfrak n$ on a general element. This is the loss used in
Proposition~\ref{prop:controlled-hom}.

\subsection{Evaluation over a differential graded base}

It remains to restore the internal differential and verify the claim
that motivated the appendix. Let $(A,d_A)$ be a CDGA and $(C,D_C)$ a
differential graded $A$-module. The tensor differential on
$C^{\otimes_AN}$ commutes with signed permutations. Every rational
representation map used above therefore commutes with that differential
after evaluation. This argument uses neither constant coefficients nor
a differential on a rational basis of $C$: it applies directly to
$D_Cw=au$ in Remark~\ref{ex:nonconstant-conormal}.

To pass from commuting homological maps to anticommuting cochain
operators, totalize as follows. Internal degree $u$ and resolution
index $j$ give total degree $u-j$. With untwisted coefficient action,
set
\begin{equation}\label{eq:appendix-totalization}
 D(x)=d_{\mathrm{int}}x+(-1)^u\beta x,
 \qquad H(x)=(-1)^u \mathsf h x
 \quad\text{for }|x|_{\mathrm{int}}=u.
\end{equation}
Use the same rule for $\ell,\delta_m$, and the other
homological homotopies.  Since internal maps commute with
the untotalized bar maps, these formulas give
$D^2=0$ and $d_{\mathrm{int}}H+Hd_{\mathrm{int}}=0$.
They also give the required graded $A$-linearity of the
homotopies of degree $-1$. This totalization is isomorphic to
the usual shifted-module convention by multiplying the
component of bidegree $(u,j)$ by $(-1)^{ju}$.

For the split minimal model $R=A\otimes\Lambda W$,
take $C=I/I^2\cong A\otimes W$ with its conormal
differential from Lemma~\ref{lem:conormal-differential}.
The evaluated internal differential is
$D_{\mathrm{lin},m}$.  Consequently the degree-zero retract
maps commute with it, while the homotopies, including
$H_{m,0}$, anticommute with it.  In
particular, \eqref{eq:appendix-kernel-contraction} remains a
contraction identity after $D_{\mathrm{lin},m}$ is added.
We have therefore obtained two distinct compatible constructions: an
$R$-linear bar retract and an $A$-linear contraction of the linear
augmentation kernel, preserving total weight and losing at most one
coefficient power. These are exactly the maps and bounds used in
Lemma~\ref{lem:natural-linear-resolution}.

\subsection{Pointwise finiteness of the Hom corrections}

The Hom contraction of Section~\ref{sec:hom-contraction} uses two
geometric series. The first lowers the source resolution index; the
second increases the target-minus-source index difference. The following
lemma makes both termination arguments precise.

\begin{lemma}[Pointwise finiteness]\label{lem:pointwise-hom-finiteness}
Let $\mathcal P=P_{b_1}\otimes_R\cdots\otimes_RP_{b_n}$, with $n\geq1$
and $m,b_1,\ldots,b_n\geq1$. On
$\operatorname{Hom}_R(\mathcal P,K_m)$, the series
\eqref{eq:koszul-hom-contraction} and \eqref{eq:full-hom-contraction}
terminate after evaluation on any fixed source generator and homogeneous
input map, with a bound that may depend on both. The associated geometric sums define two-sided inverses of
$1+\mathcal S\kappa$, $1+\kappa\mathcal S$, and
$1+\partial_+\mathcal H_0$.
\end{lemma}

\begin{proof}
The source need not have a bounded resolution index. We therefore fix
an input map and a source generator, and prove finiteness at that
single evaluation.

\emph{1. The source Koszul corrections.}
Every occurrence of $\mathcal S$ precomposes with $\delta_{\mathcal P}$ and
lowers the source index. The first series consequently terminates at
each generator.

\emph{2. The nonlinear corrections.}
Lowering the source index no longer suffices because
$\Delta_{\mathcal P}$ can raise it. Instead, for a nonzero matrix term
of a homogeneous map, consider
\[
                   \nu=j_{\mathrm{target}}-j_{\mathrm{source}}.
\]
Each $\kappa$ increases $\nu$ by exactly one, because $H_{m,0}$
maps cycles at index $j$ to the chosen complement at index $j+1$
in Appendix~\ref{app:universal-linear-resolution}.
Each $\mathcal S$ decreases $\nu$ by one.
Thus every summand of $\mathcal H_0$ increases $\nu$ by exactly one.
Neither component of $\partial_+$ decreases it: the target term does
not decrease the target index, while a source term $j\mapsto j+c$ evaluates the map
at $j+c$ and contributes an increase of $c\ge0$ relative to the original
source index.

Fix a source generator $g$ and a homogeneous map $f$. Let
$\operatorname{Dep}(g)$ consist of $g$ and all source generators
reached by iterating matrix terms of $\delta_{\mathcal P}$ and
$\Delta_{\mathcal P}$. Each step lowers cohomological degree, since
every nonzero matrix coefficient has degree at least two. Each
matrix has finite support and the source generators have a lower
degree bound. The dependency tree therefore has finite branching
and finite depth, so $\operatorname{Dep}(g)$ is finite.

Only the values of $f$ on $\operatorname{Dep}(g)$ can enter the series
at $g$. If all those values vanish, the series vanishes there.
Otherwise their finite supports give a common lower bound
$\nu_0$ for the initial index differences. The term
$\mathcal H_0(\partial_+\mathcal H_0)^a f$, evaluated on $g$,
has degree $|g|+|f|-1$ and index difference at least
$\nu_0+a+1$. For every target generator occurring in this value,
\eqref{eq:resolution-degree-bound} gives
\[
 j_{\mathrm{target}}\le |g|+|f|-1-2m,
 \qquad
 \nu_0+a+1\le |g|+|f|-1-2m-j(g).
\]
This bounds $a$. For $(\partial_+\mathcal H_0)^a f$, the
same argument uses degree $|g|+|f|$ and lower bound $\nu_0+a$.
Each term of the first series is already finite on the fixed
input, so both expanded series are pointwise finite on the
full Hom complex.

\emph{3. The geometric sums are genuine inverses.}
Pointwise termination also justifies composition with these sums.
At a fixed generator, each operator uses only finitely many values of
the input map on source generators.  This is
immediate for $\mathcal S,\kappa,\partial_+$ and follows for
$\mathcal H_0$ from the first, finite series.  Consequently
applying an operator to a pointwise finite sum commutes with
that sum: only finitely many input evaluations are involved,
and each of them stabilizes.  The identities
\[
 (1+\mathsf T)\sum_{a\geq0}(-\mathsf T)^a
   =\sum_{a\geq0}(-\mathsf T)^a(1+\mathsf T)=1
\]
now follow by telescoping, for
$\mathsf T=\mathcal S\kappa,\kappa \mathcal S$, or $\partial_+\mathcal H_0$.
For the first two, finiteness follows by lowering the source
index; for the last, it is the index-difference argument
just proved.
\end{proof}

\section{Explicit resolutions for the rational two-sphere}
\label{app:sphere-resolutions}

We verify the complete sphere data of Example~\ref{ex:sphere-comparison}
and the ideal-valued nullhomotopy of Example~\ref{ex:sphere-ideal-valued}.
The final calculation fixes the arity-two normalization of the cobar model.

\subsection{Resolutions and comparison maps}

We use the coordinates $\bar x=x_R-x_L$, $\bar a=a_R-a_L$
of \S\ref{subsec:running-example}, and write $\lambda=2x+\bar x$.
Thus $I=(\bar x,\bar a)$, $d\bar a=\bar x\lambda$, and
$d\lambda=0$. The resolutions $P_m^{S^2}$, their cell degrees and their
augmentations were introduced in Example~\ref{ex:sphere-comparison}
and \eqref{eq:sphere-resolution-differential}--\eqref{eq:sphere-resolution-augmentation}.
Here we verify these formulas, with the superscript $(m)$ suppressed
and $e_{-1}=f_{-1}=0$.

Extend $D$ by the differential-module rule. The terms in $D^2e_j$
and $D^2f_j$ cancel using $d\bar a=\bar x\lambda$ and $\bar a^2=0$; the same
two identities show that $\varepsilon_m$ is a cochain map. Ordering the
cells as $e_0,f_0,e_1,f_1,\ldots$ makes $P_m^{S^2}$ semifree, because the
differential of each cell involves only earlier ones. It remains to
prove that the augmentation is a quasi-isomorphism.

To verify that it resolves $I^m$, first take $m=1$ and consider
\[
 \mathcal E_{\mathrm{sph}}=R\otimes\Lambda(u,v),\qquad |u|=1,\quad |v|=2,
 \qquad Du=\bar x,\quad Dv=\bar a-\lambda u.
\]
Writing $\bar a'=\bar a-\lambda u$ identifies $\mathcal E_{\mathrm{sph}}$ over $A$
with the tensor product of the two contractible pairs
$(u,\bar x)$ and $(v,\bar a')$. Therefore its projection to $A=R/I$
is a quasi-isomorphism. Write $v^{[j]}=v^j/j!$. The assignments
\begin{equation}\label{eq:sphere-cone-identification}
 se_j^{(1)}\longmapsto(-1)^j u v^{[j]},
 \qquad
 sf_j^{(1)}\longmapsto(-1)^j v^{[j+1]}
\end{equation}
give an isomorphism of differential $R$-modules
$R\oplus sP_1^{S^2}\cong\mathcal E_{\mathrm{sph}}$, identical on $R$, with the
cone differential of \eqref{eq:general-module-cone}. It follows that
$\varepsilon_1:P_1^{S^2}\to I$ is a quasi-isomorphism. Multiplication
by $\bar x^{m-1}$ identifies $I$, with its degrees increased by
$2m-2$, with $I^m$: indeed, $I^m=\bar x^{m-1}I$ and multiplication
by $\bar x$ is injective. The same degree shift takes $P_1^{S^2}$ to
$P_m^{S^2}$ and its augmentation to \eqref{eq:sphere-resolution-augmentation}.
This proves the resolution assertion for every $m$.

The consecutive comparisons in
\eqref{eq:sphere-consecutive-comparison} multiply every cell by
$\bar x$. The differential matrices are independent of $m$, so these
are cochain maps; their augmentations show that they lift
$I^{m+1}\hookrightarrow I^m$, and their image lies in $IP_m^{S^2}$.

For the binary comparisons, the formulas in
\eqref{eq:sphere-binary-comparison} define a graded $R$-linear map
$\sigma_m:P_m^{S^2}\otimes_RP_m^{S^2}\to P_m^{S^2}$ of degree $-1$. The cell parities
and the symmetry of the binomial coefficients show that it is graded
antisymmetric. Every nonzero output has index at least one and a factor
$\bar x^{m-1}$, so its image lies in $I^{m-1}\ker\varepsilon_m$.

For a direct verification of the homotopy identity, transport
the product of $\mathcal E_{\mathrm{sph}}$ through
\eqref{eq:sphere-cone-identification}. The identities $u^2=0$
and
\[
 v^{[i]}v^{[j]}=\binom{i+j}{i}v^{[i+j]}
\]
give exactly \eqref{eq:sphere-binary-comparison} when $m=1$,
with the suspension convention
\[
 (sp)(sq)=(-1)^{|p|+1}s\sigma_1(p,q).
\]
The Leibniz rule for this product is precisely
\[
 D\sigma_1+\sigma_1D_{\otimes}
 =\bigl[p\otimes q\longmapsto
 \varepsilon_1(p)q-(-1)^{|p||q|}\varepsilon_1(q)p\bigr].
\]
For general $m$, the differential matrices in
\eqref{eq:sphere-resolution-differential} are unchanged, while
both the augmentation and the binary formula acquire the
closed even factor $\bar x^{m-1}$. The same identity therefore
holds with subscript $m$.

\subsection{An ideal-valued nullhomotopy}\label{ex:even-sphere}

Let $h:P_3^{S^2}\to I$ and $\omega=h(e_0)$ be the map and element defined in
Example~\ref{ex:sphere-ideal-valued}. We check all its module equations,
including those beyond the first primitive.
The nullhomotopy equation $dh+hD=\varepsilon_3$ follows on the
first two cells from
\[
 d\omega=\bar x^3,
 \qquad d(4a\bar a)+\lambda\omega=\bar x^2\bar a.
\]
On the next two cells, graded $R$-linearity gives
\[
 hDe_1=4\bar x a\bar a+\bar a\omega=0,
 \qquad hDf_1=4\bar a a\bar a=0.
\]
Every term on the remaining cells is zero. Therefore
$o_3(\Delta_{S^2})=0$.

\subsection{The arity-two cobar signs}\label{app:binary-cobar-signs}
For closed homogeneous $p,q$ of parities $\alpha,\gamma$, put
$c=[s^2p\mid s^2q]$. Equations~\eqref{eq:harrison-arity-two} and
\eqref{eq:curved-cobar-differential} give
\[
 D_E(tc)=-s\sigma(p,q)-(-1)^\alpha(sp)(sq).
\]
Here $e_2$ chooses the half-antisymmetrized representative, whereas
$\nabla$ is the full antisymmetric cobracket. Its two output terms
therefore combine with the cobar factor $1/2$ into the single displayed
product. Writing
$\mathcal E(p,q)=\varepsilon(p)q-(-1)^{\alpha\gamma}\varepsilon(q)p$, we obtain
\begin{align*}
 D_E\bigl(-s\sigma(p,q)\bigr)
  &=s\mathcal E(p,q)
    =(-1)^\alpha\varepsilon(p)sq-(-1)^{\alpha\gamma+\gamma}\varepsilon(q)sp,\\
 D_E\bigl(-(-1)^\alpha(sp)(sq)\bigr)
  &=-(-1)^\alpha\varepsilon(p)sq+(sp)\varepsilon(q).
\end{align*}
The remaining terms cancel because
$(sp)\varepsilon(q)=(-1)^{(\alpha-1)\gamma}\varepsilon(q)sp$.
For arbitrary $p,q$, the internal terms cancel by the same suspension
rule and the cochain identity for $\sigma$.

\section{Finiteness in the quoted rational results}\label{app:finite-type-literature}

\begin{remark}[Finiteness in quoted results]\label{rem:finite-type-literature}
The quoted results have different topological hypotheses. In particular,
\cite[Theorem~12]{CGV2014} applies to locally equiconnected spaces without
a finiteness assumption, whereas the formulation of
\cite[Corollary~4.1]{CKV2016} uses connected CW complexes of finite type. For the rational results of
\cite{CV2015,CV2016,CV2017,CGV2014,CKV2016,CPV2018}, we use their algebraic
proofs on the degreewise finite minimal models of
Section~\ref{subsec:conventions}.

Here are the constructions needed for this use. At each finite Ganea
stage, the Sullivan models of homotopy pullbacks and homotopy pushouts
are built by relative Sullivan factorizations, tensor products and
surjective replacements. These constructions apply to simply connected
spaces of finite rational type
\cite[Section~13\textup{(b)}, Propositions~13.5--13.6;
Section~14\textup{(a)}; Proposition~15.8]{FHT2001}; they do not require finitely many integral cells in
each degree. Iterating them
models the Ganea maps and their canonical structure maps. Consequently,
the quotient comparisons and the relative retractions in the cited
proofs are formed in the same rational model category and realize to
those of $f_\Q$. The module arguments use these semifree models; the
duality argument of \cite{CKV2016} uses the finite-dimensional Poincar\'e
cohomology algebra. None of these steps requires finite total rational
homotopy. Thus the degreewise rational finiteness used here supplies the
inputs of the quoted algebraic proofs.
\end{remark}

\smallskip
\noindent\emph{Note on citation numbering.}
Numbered citations to \cite{CV2015,CV2016,CGV2014,CKV2016,DEH2013,FaberEtAl2021,HRV2023}
follow the linked arXiv versions; those to \cite{LodayVallette2012}
follow the linked author version~0.99.

\Needspace*{5\baselineskip}
\bigskip
\noindent\textbf{Declaration on the use of generative artificial intelligence.}
During the preparation of this manuscript, the author used ChatGPT
(OpenAI) for language editing, \LaTeX{} preparation, bibliographic
assistance, and the checking of mathematical
arguments. The author also used Claude (Anthropic) to produce
mathematical notes, critical readings of drafts, and the reorganization
of the exposition. The author assumes full responsibility for the
mathematical statements, proofs, references, and conclusions of the
article.
\end{document}